\documentclass[11pt,reqno]{amsart}

\usepackage[margin=1in]{geometry}

\usepackage{amssymb, mathrsfs, amscd, bbm}
\usepackage{extpfeil}
\usepackage{bm}

\usepackage{graphicx}
\usepackage{colortbl}
\usepackage{float}
\usepackage{caption}
\usepackage{subcaption} 
\graphicspath{{Figure/}}
\usepackage[dvipsnames]{xcolor}

\usepackage{enumerate}
\usepackage{comment}
\usepackage{xcolor}
\usepackage{hyperref}

\newtheorem{theorem}{Theorem}[section]
\newtheorem{lemma}{Lemma}[section]

\newtheorem{proposition}{Proposition}[section]
\newtheorem{remark}{Remark}[section]

\numberwithin{equation}{section}

\newcommand{\R}{\mathbb{R}}
\newcommand{\e}{\varepsilon}
\newcommand{\pa}{\partial}

\newcommand{\us}{u^S}
\newcommand{\ur}{u^R}
\newcommand{\ds}{\delta_S}

\newcommand{\step}[1]{\vskip0.2cm \noindent{\it Step #1:} }

\newcommand{\abs}[1]{\left|#1\right|}
\newcommand{\RR}{{\mathbb{R}}}
\renewcommand{\a}{\alpha}

\renewcommand{\d}{\delta}

\newcommand{\et}{\eta}
\renewcommand{\th}{\theta}
\renewcommand{\k}{\kappa}

\newcommand{\m}{\mu}

\newcommand{\x}{\xi}

\newcommand{\s}{\sigma}

\newcommand{\util}{\tilde{u}}

\newcommand{\wtil}{\tilde{w}}

\title[Stability of Composite Waves for Modified KdVB Equation]{Large-Time Behavior towards Composite Waves of Degenerate Shock and Rarefaction Wave\\for Modified KdV--Burgers Equation}

\author[Eun]{Namhyun Eun}
\address[Namhyun Eun]{School of Mathematics, \newline Korea Institute for Advanced Study (KIAS), Seoul 02455, Republic of Korea}
\email{namhyuneun@kias.re.kr}

\author[Han]{Sungho Han}
\address[Sungho Han]{Department of Mathematical Sciences, \newline Korea Advanced Institute of Science and Technology (KAIST), Daejeon 34141, Republic of Korea}
\email{sungho\_han@kaist.ac.kr}

\author[Kim]{Jeongho Kim}
\address[Jeongho Kim]{Department of Applied Mathematics, \newline Kyung Hee University, Gyeonggi-do 17104, Republic of Korea}
\email{jeonghokim@khu.ac.kr}

\begin{document}

\begin{abstract}
In this paper, we study the modified Korteweg--de Vries--Burgers (mKdVB) equation 
\[
u_t + (u^3)_x = \m u_{xx} - \k u_{xxx}
\]
with \(\m>0\) and \(\k>0\).
Since the flux is non-convex, the (inviscid) Riemann problem may be solved by a composite wave of a degenerate Oleinik shock and a rarefaction wave.
We prove the global existence of solutions and the time-asymptotic stability of the corresponding viscous-dispersive composite wave, up to a time-dependent shift, under small \(H^1\) perturbations.
The shock strength \(\d_S\) and the rarefaction strength \(\d_R\) need not be small and are restricted only by the explicit ratio $\delta_R\le\frac{5}{18}\delta_S$ and the quantity $\kappa\delta_S^2/\mu^2$ governing the monotonicity of the shock.
The stability estimate is uniform with respect to the dispersion strength \(\k\).

As a key ingredient of the stability analysis, we establish structural properties of the degenerate viscous-dispersive shock profile in the monotone regime, including two-sided pointwise bounds on its derivative and the rates at which it converges to its two end states.
These decay rates coincide with those of the corresponding purely viscous profile, with constants independent of the dispersion coefficient.
In the absence of dispersion, namely, when \(\k=0\), the time-asymptotic stability of the corresponding viscous composite wave was established by Huang, Wang and Zhang \cite{HWZ-MA}.
The present work extends their result to the viscous-dispersive setting and relaxes the restriction on the rarefaction strength.
The proof relies on the method of \(a\)-contraction with shifts (for viscous conservation laws) developed by Kang and Vasseur \cite{KV-17,KV-JEMS21,KVW-ADV23}.
\end{abstract}



\thanks{\textbf{Acknowledgment.} N. Eun was supported by a KIAS Individual Grant (MG109001) at Korea Institute for Advanced Study and by Samsung Science and Technology Foundation under Project Number SSTF-BA2102-01.
S. Han was supported by the National Research Foundation of Korea (RS-2024-00361663 and NRF-2019R1A5A1028324) and by the Doomyung Fellowship.
J. Kim was supported by the National Research Foundation of Korea(NRF) grant funded by the Korea government(MSIT) (RS-2026-25590471).
The authors thank Professor Moon-Jin Kang for his valuable comments.}

\maketitle

\tableofcontents
%====================================================
\section{Introduction}
%====================================================
We consider the modified Korteweg--de Vries--Burgers (mKdVB) equation 
\begin{equation}\label{eq:gKdVB}
u_t + f(u)_x = \mu u_{xx} - \kappa u_{xxx}, \qquad (t,x) \in \mathbb{R}_+ \times \mathbb{R},
\end{equation}
with the cubic flux \(f(u)=u^3\).
We study the associated Cauchy problem with initial data 
\[
u(0,x) = u_0(x), \qquad x\in\RR,
\]
whose far-field states are given by 
\[
u_0(x) \to u_\pm \quad \text{as }x\to \pm\infty.
\]
Here, \(u=u(t,x)\in\RR\) is the unknown, \(\m>0\) denotes the viscosity coefficient, \(\k\neq 0\) the dispersion coefficient, and \(u_\pm\in\RR\) are the prescribed far-field states.
\begin{comment}
We consider the modified Korteweg--de Vries--Burgers (mKdVB) equation with a cubic flux:
\begin{equation} \label{eq:gKdVB}
\begin{aligned}
    &u_t + f(u)_x = \mu u_{xx} - \kappa u_{xxx}, & &f(u)=u^3, \qquad (t,x) \in \mathbb{R}_+ \times \mathbb{R}, \\
    &u(0,x) = u_0(x) \to u_\pm, & &\text{as} \quad x \to \pm \infty,
\end{aligned}
\end{equation}
where $u=u(t,x) \in \mathbb{R}$ is the unknown function, $\mu > 0$ is the viscosity coefficient, $\kappa \neq 0$ is the dispersion coefficient, $u_0(x)$ is the given initial data, and $u_\pm \in \mathbb{R}$ are the prescribed far-field states.
\end{comment}

\vspace{2mm}
The mKdVB equation is of considerable interest from both mathematical and physical perspectives and has been studied extensively over several decades.
Early developments can be traced back to the work of Su and Gardner \cite{SG-1969}, who derived the KdV and Burgers equations as long-wave models for weakly nonlinear dispersive and dissipative systems, respectively.
Johnson \cite{Johnson-1970} subsequently investigated the KdV--Burgers equation, which incorporates the interplay between dissipation and dispersion.
On the other hand, Nariboli--Lin derived a modified Burgers equation with cubic flux \cite{NL-73}, which was further studied in various works (for instance, \cite{LC-1987}). Later, Pakzad--Javidan \cite{PJ-09} derived the mKdVB equation with cubic flux in the context of dusty plasma.
The mKdVB equation has also attracted substantial mathematical interest as a canonical model for nonlinear waves involving the simultaneous effects of nonconvex convection, dissipation, and dispersion (see \cite{JMS-JDE95}).

\vspace{2mm}
One of the fundamental questions concerning \eqref{eq:gKdVB} is to understand the asymptotic behavior of its solutions as \(t\to\infty\).
The long-time behavior of the solution to \eqref{eq:gKdVB} with distinct far-field states \(u_\pm\) is governed by the corresponding inviscid Riemann problem
\begin{equation} \label{Riemann}
u_t + f(u)_x = 0, \qquad
u(0,x) = 
\begin{cases}
    u_-, & x<0,\\
    u_+, & x>0.
\end{cases}
\end{equation}
When the flux \(f\) is convex, the solution to \eqref{Riemann} is a shock wave for \(u_->u_+\) and a rarefaction wave for \(u_-<u_+\), and the asymptotically stable profiles of \eqref{eq:gKdVB} are the corresponding viscous (when \(\k=0\)) or viscous-dispersive waves.
In the purely viscous case (\(\k=0\)), the time-asymptotic stability of viscous shock and rarefaction waves is classical and goes back to the work of Il'in and Oleinik \cite{IlinOleinik}.
More recently, it was shown by Kang--Vasseur \cite{KV-17} that viscous shocks of the viscous Burgers equation satisfy an \(L^2\)-contraction property under arbitrarily large perturbations, where they made clever use of time-dependent shifts (temporal modulations) to compensate for the translation invariance of shock profiles (see also \cite{KO-AA25} for the multi-dimensional case).
In the viscous-dispersive case (\(\k\neq0\)), Pego \cite{Pego-1985} established the stability of viscous-dispersive shocks to the KdV--Burgers equation (with quadratic flux) using the anti-derivative method introduced in \cite{Goodman-1986}.
This was later extended to arbitrarily large perturbations in \cite{BBHY-JDE25,CEKS2,CEKS1} and the multi-D case in \cite{CRS-26}.

\vspace{2mm}
When the flux \(f\) is non-convex, the Riemann solution has a richer structure, in particular, the two states \(u_-\) and \(u_+\) may be connected by a composite of an Oleinik shock and a rarefaction wave.
It is well known that a shock connecting \(u_l\) and \(u_r\) propagates with the Rankine--Hugoniot speed 
\begin{equation} \label{RH}
\sigma = \frac{f(u_r)-f(u_l)}{u_r-u_l}.
\end{equation}
Such a shock is entropy admissible, or an Oleinik shock, if
\[
\frac{f(u)-f(u_l)}{u-u_l} \ge \sigma \ge \frac{f(u)-f(u_r)}{u-u_r}
\qquad \text{for all } u \text{ between } u_l \text{ and } u_r.
\]
In fact, the odd symmetry of the cubic flux allows us to reverse the orientation of the shock through the transformation \(u\mapsto -u\).
Thus, without loss of generality, we restrict our attention to increasing shocks with \(u_l < u_r\).
In addition, the Oleinik shock is degenerate (also known as a one-sided contact discontinuity) when the shock speed agrees with the characteristic speed at its right end state 
\begin{equation} \label{deg}
\sigma = f'(u_r)
\end{equation}
in contrast to a non-degenerate shock for which the corresponding characteristic inequality is strict.
Since the degenerate shock is characteristic at its right end state, it can be attached to a rarefaction wave, with the shock speed matching the characteristic speed at the left edge of the rarefaction.
This gives rise to a shock--rarefaction composite wave, and in this paper, we focus on the regime in which the corresponding Riemann solution consists of a degenerate Oleinik shock followed by a rarefaction wave.
This regime arises when $u_-<-u_-/2<u_+$, equivalently, when $u_-<0$ and $u_+>-u_-/2$, with the intermediate state $u_m=-u_-/2>0$.
The leading shock therefore crosses the inflection point \(u=0\) of the flux \(f\).
The stability of individual shock profiles for scalar viscous conservation laws with non-convex flux was studied in \cite{JGK-CPAM93,MN-CMP94}, with the degenerate case treated in the latter work; see also \cite{KM-CPAM94} for a related result for a viscoelasticity system with a non-convex constitutive relation.
In the dissipative-dispersive setting, Howard--Zumbrun \cite{HZ-ARMA00} proved the nonlinear stability of Lax and undercompressive shock profiles, while Dodd \cite{Dodd-07} established the spectral stability of certain undercompressive shocks for the cubic mKdVB equation.
Evans-function stability criteria for degenerate viscous shocks were studied in \cite{HZ-DCDS04}.
These results, however, concern single shock profiles rather than shock--rarefaction composite waves.
In the absence of dispersion, i.e., when \(\k=0\), Huang--Wang--Zhang \cite{HWZ-MA} showed the time-asymptotic stability of the corresponding viscous composite wave without any restriction on the shock strength.
Their proof is based on the method of \(a\)-contraction with shifts, together with a carefully chosen weight function.
However, to the best of our knowledge, no stability result for such a composite wave is available for the mKdVB equation in the presence of dispersion.
The main objective of this paper is to fill this gap by establishing the stability of the corresponding shock--rarefaction composite wave.
%We carry out the general classification of viscous-dispersive shocks in the next subsection, and from then on we restrict to this regime.

\vspace{2mm}
Our analysis relies on the method of \(a\)-contraction with shifts which was developed in \cite{KV-JEMS21}, building on the relative-entropy framework of \cite{LV-ARMA11}.
This carefully designed highly nonlinear energy method provides a powerful and robust tool for the stability analysis of viscous shocks.
A classical and widely used approach to the nonlinear stability of viscous shock waves is the anti-derivative method \cite{Goodman-1986}.
The integrated formulation underlying this approach, however, is not compatible with the direct energy estimates used for rarefaction waves.
This incompatibility constituted a major obstacle to the analysis of composite waves containing both a shock and a rarefaction.
By contrast, the \(a\)-contraction framework remains within a relative entropy energy formulation and has proved particularly effective in treating superpositions of elementary waves of different types.
Using this framework, Kang--Vasseur--Wang \cite{KVW-ADV23} established the time-asymptotic stability of a composite wave consisting of a viscous shock and a rarefaction wave for the barotropic compressible Navier--Stokes equations; see also \cite{HKK-SIMA23} for a composite wave formed by two viscous shocks.
The same authors later showed the stability of a shock--contact--rarefaction composite wave for the compressible Navier--Stokes--Fourier system \cite{KVW-ARMA25}.
The method has also been extended to composite waves involving a boundary layer in initial-boundary value problems on the half-line \cite{HKKKO-CMP26}, and even to the Boltzmann equation \cite{CKK-26,WY-25}.
Another strength of the \(a\)-contraction method is its ability to handle large perturbations (e.g. \cite{CKKV-M320,EEK-24,HWWW-SIMA24,KV-Inven21}).
The absence of a smallness assumption on the perturbation is useful in obtaining uniform estimates and hence in studying the vanishing viscosity limits.
We also refer to \cite{CFK-25,CKV-ARMA22,GKV-JHDE23} for further motivation for using the \(a\)-contraction method to study the stability of discontinuous solution patterns for the inviscid equation, namely, the compressible Euler equation.

The compressible Navier--Stokes--Korteweg (NSK) system is a representative viscous-dispersive model in fluid dynamics and shares with the mKdVB equation \eqref{eq:gKdVB} the simultaneous presence of dissipative and dispersive mechanisms: viscosity produces dissipation, while capillarity gives rise to dispersion.
Related stability results have been obtained for the NSK system using the method of \(a\)-contraction with shifts.
The stability of a viscous-dispersive shock was established in \cite{HKKL-JDE25}, and the stability of a shock--rarefaction composite wave was subsequently studied in \cite{HK-SIMA26}.
The stability of a single shock for the outflow problem on the half-line was investigated in \cite{HKO-26}.
Furthermore, the stability of a shock under arbitrarily large perturbations was proved in \cite{EKK-26}.

\section{Main Results}\label{sec:pre}
\setcounter{equation}{0}
The two main results of this paper are presented in this section.
The first concerns quantitative properties of the degenerate Oleinik shock profile, including its exponential and algebraic decays around the far-field states, while the second establishes the stability of a composite wave consisting of the degenerate shock and a rarefaction wave.
To formulate these results, we first introduce the degenerate shock, the rarefaction wave (and its smooth approximation) and their composite wave.

\subsection{Degenerate Oleinik Shock Wave} \label{subsec:degen-shock}
The structure of shock waves depends on the sign of the dispersion coefficient and on the relative strengths of the shock amplitude, viscosity and dispersion.
Indeed, the mKdVB equation \eqref{eq:gKdVB} with cubic flux admits various types of shock waves, including classical Oleinik shocks and nonclassical undercompressive shocks, whose profiles may be monotone or oscillatory; see \cite{EHS-SIMA17,JMS-JDE95}.
In the present work, we restrict our attention to the case of negative dispersion, i.e., \(\k>0\), and consider the following degenerate Oleinik shock.
Let \(u_m = -\frac{u_-}{2}>0\) and assume that \(u_m< u_+\).
Then, \(u_-\) and \(u_m\) are connected by a degenerate Oleinik shock, while \(u_m\) and \(u_+\) are connected by a rarefaction wave.
The corresponding viscous-dispersive shock profile is denoted by 
\[
\us = \us(\x), \qquad \x = x-\s t,
\]
where \(\s\) is given by \eqref{RH} with \(u_l = u_-\) and \(u_r = u_m\) (and so \(\s=f'(u_m)=3u_m^2\) follows from \eqref{deg}), and it satisfies 
\[
-\s \us_\x + (f(\us))_\x = \m \us_{\x\x} - \k \us_{\x\x\x}
\]
together with the far-field conditions 
\[
\us(-\infty) = u_-, \qquad
\us(+\infty) = u_m.
\]
Throughout the paper, we also impose the following monotonicity criterion (see \cite{EHS-SIMA17}) for the shock:
\begin{equation}\label{con:mon}
(2u_- + u_m)(u_- - u_m)\, \kappa \le \frac{\mu^2}{4}.
\end{equation}
Under this criterion, the shock profile \(\us\) is monotonically increasing.

\vspace{2mm}
We henceforth adopt the following normalization and notation:

\noindent\(\bullet\) \textbf{Normalization:} Without loss of generality, we assume throughout the remainder of the paper that \(\m=1\), so that the role of the dispersion coefficient can be seen more clearly.
The case of general viscosity coefficient \(\m>0\) is recovered by scaling \(u^S((x-\s t)/\m)\).

\noindent\(\bullet\) \textbf{Notation:} Since \(u_m=-u_-/2\), we write, for notational simplicity, 
\[
u_m = s >0, \qquad u_- = -2s <0.
\]
Then, the shock speed is \(\s=3s^2\) and the monotonicity criterion becomes
\begin{equation}\label{threshold}
36 \k s^2 \le 1.
\end{equation}

We now state the first main theorem concerning the properties of the degenerate Oleinik shock.
Its existence, uniqueness and monotonicity under \eqref{threshold} are proved below by a phase-plane argument.

\begin{theorem} \label{thm:shock_properties}
Assume that \(\m=1\) and that \(\k>0\) satisfies \eqref{threshold}.
Then there exists a smooth monotone traveling wave solution \(\us=\us(\x)\), with \(\x=x-\s t\), connecting \(u_-=-2s\) to \(u_m=s\).
The profile is unique up to translation, normalized by \(\us(0)=0\), and satisfies the following estimates:
\begin{equation} \label{eq:key-ineq}
(4\sqrt{3}-6)(\us(\x) - s)^2(\us(\x) + 2s) \le \us_{\xi}(\x) \le 2(\us(\x) - s)^2(\us(\x) + 2s), \qquad \forall \x\in\RR,
\end{equation}
which implies the following decay properties:
\begin{align}
&2s \cdot e^{-18s^2|\xi|} \le \us(\xi) + 2s \le 2s \cdot e^{-(4\sqrt{3}-6)s^2|\xi|},
&&\forall \x<0, \label{eq:left_decay_u}\\
&\frac{s}{1 + 6s^2|\xi|} \le s - \us(\xi) \le \frac{s}{1 + 2(4\sqrt{3}-6) s^2|\xi|},
&&\forall \x>0, \label{eq:right_decay_u}\\
&2(4\sqrt{3}-6)s^3 \cdot e^{-18s^2|\xi|} \le \us_{\xi}(\xi) \le 36s^3 \cdot e^{-(4\sqrt{3}-6)s^2|\xi|},
&&\forall \x<0, \label{eq:left_decay_up}\\
&\frac{2(4\sqrt{3}-6) s^3}{(1 + 6s^2|\xi|)^2} \le \us_{\xi}(\xi) \le \frac{6s^3}{(1 + 2(4\sqrt{3}-6) s^2|\xi|)^2},
&&\forall \x>0. \label{eq:right_decay_up}
\end{align}
Moreover, the higher-order derivatives satisfy
\begin{align}
-3s^2\us_{\xi} &\le \us_{\xi \xi} \le 18s^2 \us_{\xi}, &&\forall \x\in\RR,
\label{eq:second_deriv_bound}\\
-60 s^4 \us_{\xi}
&\le \us_{\xi \xi \xi}
\le 324 s^4 \us_{\xi}, &&\forall \x\in\RR.
\label{eq:third_deriv_bound}
\end{align}
This implies the following absolute bounds 
\begin{equation} \label{eq:absbdd}
|\us_{\xi \xi}| \le 18s^2 \us_{\xi} \quad \text{and} \quad
|\us_{\xi \xi \xi}| \le 324s^4 \us_{\xi}.
\end{equation}
\end{theorem}

\begin{remark}
Although Theorem \ref{thm:main} requires a smallness condition on the critical parameter \(\k s^2\), Theorem \ref{thm:shock_properties} applies to the whole monotone regime with \eqref{con:mon} (or \eqref{threshold}).
In addition, the decay properties obtained in Theorem \ref{thm:shock_properties} coincide with those in \cite[Lemma 2.1]{HWZ-MA} for the degenerate Oleinik shock associated with the following viscous scalar conservation law without dispersion: 
\[
u_t + (u^3)_x = u_{xx}.
\]
Thus, in the monotone regime, the shock profiles decay on the same scales as the corresponding shock profiles for the equation without dispersion.
\end{remark}

The idea of the proof is as follows.
The cornerstone of the argument is to establish \eqref{eq:key-ineq}, which is non-trivial due to the presence of the dispersion.
Integrating the traveling wave equation yields 
\begin{equation} \label{eq:shock-eq}
(\us - s)^2(\us + 2s) = \us_{\xi} - \kappa \us_{\xi \xi}.
\end{equation}
In the purely viscous case without dispersion, one immediately obtains 
\[
(\us - s)^2(\us + 2s) = \us_{\xi},
\]
and the decay properties then follow from a standard Gr\"onwall argument.
However, in the present dispersive case, the dispersion term prevents such a direct argument.
Nevertheless, if \eqref{eq:key-ineq} can be proved, then the same decay properties can be obtained.
The proof is provided in Section \ref{sec:first-pf}.

\subsection{Rarefaction Wave and Its Smooth Approximation}\label{subsec:inviscid-rare}

The second wave in the composite wave is the inviscid rarefaction wave $u^r(x/t;u_m,u_+)$ connecting the intermediate state $u_m$ to the right far-field state $u_+$. It is a self-similar entropy solution of the conservation law \eqref{Riemann}, given by
\begin{equation}\label{eq:rarefaction}
    u^r\Big(\tfrac{x}{t};u_m,u_+\Big)
    =
    \begin{cases}
        u_m, & x \le 3u_m^2\, t,\\[1mm]
        \sqrt{x/(3t)}, & 3u_m^2\, t \le x \le 3u_+^2\, t,\\[1mm]
        u_+, & x \ge 3u_+^2\, t,
    \end{cases}
\end{equation}
where $3u_m^2=f'(u_m)$ and $3u_+^2=f'(u_+)$ are the characteristic speeds at the two end states. Since $0<u_m<u_+$, the rarefaction wave \eqref{eq:rarefaction} is Lipschitz continuous and non-decreasing in $x$. 

To study the asymptotic stability of a composite wave involving a rarefaction wave, we introduce a smooth approximation of the (inviscid) rarefaction profile \(u^r(x/t)\) connecting the intermediate state \(u_m\) and the right far-field state \(u_+\).
Following the standard construction in \cite{MN92}, we consider the following Cauchy problem for the (inviscid) Burgers equation with smooth initial data:
\begin{equation}\label{eq: inviscid Burgers equation}
\begin{cases}
w_t+w w_x=0,\\
w(0,x)= w_0(x)= \dfrac{w_++w_m}{2}+\dfrac{w_+-w_m}{2}K_q\displaystyle\int_0^{\varepsilon_R x}\frac{1}{(1+y^2)^q}\, dy,
\end{cases}
\end{equation}
where \(w_m \coloneqq f'(u_m) = 3u_m^2\) and \(w_+\coloneqq f'(u_+)=3u_+^2\) denote the characteristic speeds at \(u_m\) and \(u_+\) respectively (so that \(w_m < w_+\)), \(q\ge10\) is a fixed constant, and \(K_q\) is the normalization constant satisfying $K_q\int_0^\infty (1+y^2)^{-q}\, dy=1$.
Moreover, \(\e_R>0\) is a small smoothing parameter, fixed at the end of the proof of Theorem~\ref{thm:main}.

Since we consider the case $u_+ > u_m > 0$, both characteristic speeds $w_m = 3u_m^2$ and $w_+ = 3u_+^2$ are positive. Moreover, the map $\Lambda(w):=(f')^{-1}(w)=\sqrt{w/3}$ is a smooth diffeomorphism from $[w_m, w_+]$ onto $[u_m, u_+]$. We thus define the smooth approximate rarefaction wave by
\begin{equation}\label{eq:uR-def}
	u^R(t,x) := \Lambda(w(t,x)) = \sqrt{\dfrac{w(t,x)}{3}}.
\end{equation}
Then, by construction, $u^R(t,x)$ satisfies the inviscid scalar conservation law
\begin{equation}\label{eq:uR-eq}
u^R_t + f(u^R)_x = 0,
\end{equation}
which follows from $w_t + ww_x = 0$ and $w = f'(u^R)$.
The quantitative properties of \(u^R\) used in the stability analysis are collected in the following lemma, whose proof is given in Section \ref{sec:first-pf}.

\begin{lemma}\label{lem:Burgers}
Let \(\wtil\coloneqq w_+-w_m>0\).
The unique global smooth solution $w(t,x)$ to \eqref{eq: inviscid Burgers equation} satisfies the following, where the constants $C_{p,q}$, $C_q$, and $C_{q,\theta}$ depend only on the indicated parameters:
\begin{itemize}
\item[(1)] $w_m< w(t,x)<w_+$ and $w_x(t,x)>0$ for all $t>0$ and $x\in\R$.
\item[(2)] For any $1\le p\le \infty$, there exists a constant $C_{p,q}>0$ such that, for all $t >0$,
\begin{align*}
\|w_x(t,\cdot)\|_{L^p(\RR)}
&\le C_{p,q}\min\bigg\{\tilde{w}\varepsilon_R^{1-1/p},\ \frac{\tilde{w}^{1/p}}{t^{1-1/p}}\bigg\}, \\
\|\pa_x^j w(t,\cdot)\|_{L^p(\RR)}
&\le C_{p,q}\min\Bigg\{\tilde{w}\varepsilon_R^{j-1/p},\ \frac{\tilde{w}^{(j-1)/q}  \varepsilon_R^{(j-1)-1/p+1/q}}{t^{1-(j-1)/q}}\Bigg\}, \qquad j=2,3,4.
\end{align*}
\item[(3)] There exists a constant \(C_q>0\) such that for any \(t\ge0\) and any \(x\le w_m t\),
\begin{align*}
|w(t,x)-w_m|
&\le \frac{C_q \tilde w}{\bigl(1+(\varepsilon_R|x-w_m t|)^2\bigr)^{q/2}}, \\
|w_x(t,x)|
&\le \frac{C_q \tilde w \varepsilon_R}{\bigl(1+(\varepsilon_R|x-w_m t|)^2\bigr)^{q}}.
\end{align*}
\item[(4)] There exists a constant \(C_q>0\) such that for any \(t\ge0\) and any \(x\ge w_+ t\),
\begin{align*}
|w(t,x)-w_+|
&\le \frac{C_q  \tilde w}{\bigl(1+(\varepsilon_R|x-w_+ t|)^2\bigr)^{q/2}}, \\
|w_x(t,x)|
&\le \frac{C_q  \tilde w \varepsilon_R}{\bigl(1+(\varepsilon_R|x-w_+ t|)^2\bigr)^{q}}.
\end{align*}
\item[(5)] Let \(w^r\) be the rarefaction wave of the Burgers equation connecting $w_m$ and $w_+$, that is, $w^r(\eta)=w_m$ for $\eta\le w_m$, $w^r(\eta)=\eta$ for $w_m\le\eta\le w_+$, and $w^r(\eta)=w_+$ for $\eta\ge w_+$.
Then, it holds that
\[
\lim_{t\to\infty}\sup_{x \in \mathbb{R}}|w(t,x)-w^r(x/t)|=0.
\]
\item[(6)] Let \(\a>1\) be a fixed constant and $T_0 \coloneqq\varepsilon_R^{-\alpha}$.
Then there exist constants $\delta_1 = \delta_1(q,\alpha,\tilde w) > 0$ and $C_{q,\theta}>0$ such that for all \(t\ge T_0\), \(0<\e_R<\d_1\), and \(\th \in (0,1)\),
\begin{align*}
|w(t,x)-w_+|
&\le C_{q,\th}   \tilde{w}   (\tilde{w}  \varepsilon_R  t)^{-\theta(1-\frac{1}{2q})} \bigl(1+\varepsilon_R|x-w_+ t|\bigr)^{-q(1-\theta)}, \quad x \ge w_+ t, \\
|w(t,x)-w_m|
&\le C_{q,\th}   \tilde{w}   (\tilde{w}  \varepsilon_R  t)^{-\theta(1-\frac{1}{2q})} \bigl(1+\varepsilon_R|x-w_m t|\bigr)^{-q(1-\theta)}, \quad x \le w_m t.
\end{align*}
\item[(7)] Under the assumptions of (6), there exists a constant \(C_q>0\) such that for all \(t\ge T_0\) and \(0<\e_R<\d_1\), 
\[
|w(t,x)-w^r(x/t)|
\le C_q   \tilde{w}^{\frac{1}{2q}}   \varepsilon_R^{-1+\frac{1}{2q}} t^{-1+\frac{1}{2q}},
\quad w_m t \le x \le w_+ t.
\]
\end{itemize}
\end{lemma}

Moreover, the approximate rarefaction wave $u^R$ defined in \eqref{eq:uR-def} satisfies the following:

\begin{lemma}\label{lem:approx_rarefaction}
Let $\delta_R \coloneqq u_+ - u_m > 0$ be the rarefaction wave strength.
Then, the smooth approximate rarefaction wave $u^R(t,x)$ defined in \eqref{eq:uR-def} satisfies the following properties, where $C$, $C_q$, $C_{p,q}$, and $C_{q,\theta}$ denote positive constants that may depend on the indicated parameters and on the end states $u_m$, $u_+$:
\begin{itemize}
\item[(1)] $u_m < u^R(t,x) < u_+$ and $\ur_x(t,x) > 0$ for all $t>0$ and $x\in\R$.
\item[(2)] For any $1\le p\le \infty$, there exists a constant $C_{p,q}>0$ such that, for all $t >0$,
\begin{align*}
\|\ur_x(t, \cdot)\|_{L^p(\RR)}
&\le C_{p,q} \min \bigg\{ \delta_R \varepsilon_R^{1-1/p},\ \frac{\delta_R^{1/p}}{t^{1-1/p}} \bigg\},\\
\|\partial_x^j u^R(t, \cdot)\|_{L^p(\RR)}
&\le C_{p,q} \min \Bigg\{ \delta_R \varepsilon_R^{j-1/p},\ \frac{\delta_R^{1/p} + \delta_R^{(j-1)/q}}{t^{1-(j-1)/q}} \Bigg\}, \qquad j = 2,3,4.
\end{align*}
Moreover, the following pointwise bound holds:
\[
|u^R_{xx}(t,x)| \le C\bigl(|\ur_x(t,x)|^2 + |\ur_x(t,x)|\bigr), \quad \forall\, t>0, x \in \RR.
\]
\item[(3)] There exists a constant \(C_q>0\) such that for any \(t\ge0\) and any \(x\le w_m t\),
\begin{align*}
|u^R(t,x) - u_m|
&\le \frac{C_q \, \delta_R}{\bigl(1 + (\varepsilon_R|x - w_m t|)^2\bigr)^{q/2}},\\
|\ur_x(t,x)|
&\le \frac{C_q \, \delta_R \, \varepsilon_R}{\bigl(1 + (\varepsilon_R|x - w_m t|)^2\bigr)^q}.
\end{align*}
\item[(4)] There exists a constant \(C_q>0\) such that for any \(t\ge0\) and any \(x\ge w_+ t\),
\begin{align*}
|u^R(t,x) - u_+|
&\le \frac{C_q \, \delta_R}{\bigl(1 + (\varepsilon_R|x - w_+ t|)^2\bigr)^{q/2}},\\
|\ur_x(t,x)|
&\le \frac{C_q \, \delta_R \, \varepsilon_R}{\bigl(1 + (\varepsilon_R|x - w_+ t|)^2\bigr)^q}.
\end{align*}
\item[(5)] Let \(u^r\) be the inviscid rarefaction wave introduced in \eqref{eq:rarefaction}.
Then, it holds that
\[
\lim_{t \to \infty} \sup_{x \in \mathbb{R}} |u^R(t,x) - u^r(x/t)| = 0.
\]
\item[(6)] Let \(\a>1\) be a fixed constant and $T_0 \coloneqq\varepsilon_R^{-\alpha}$.
Then there exists a constant $\delta_1>0$, depending only on $q$, $\alpha$, $u_m$, and $\delta_R$, such that, for all $t\ge T_0$, $0<\varepsilon_R<\delta_1$, and $\theta\in(0,1)$,
\begin{align*}
|u^R(t,x) - u_+|
&\le C_{q,\theta} \delta_R (\delta_R \, \varepsilon_R \, t)^{-\theta(1-\frac{1}{2q})}
\bigl(1+\varepsilon_R|x-w_+ t|\bigr)^{-q(1-\theta)}, \quad x \ge w_+ t, \\
|u^R(t,x) - u_m|
&\le C_{q,\theta} \delta_R (\delta_R \, \varepsilon_R \, t)^{-\theta(1-\frac{1}{2q})}
\bigl(1+\varepsilon_R|x-w_m t|\bigr)^{-q(1-\theta)}, \quad x \le w_m t.
\end{align*}
\item[(7)] Under the assumptions of (6), there exists a constant \(C_q>0\) such that for all \(t\ge T_0\) and \(0<\varepsilon_R<\delta_1\), 
\[
|u^R(t,x) - u^r(x/t)|
\le C_q \delta_R^{\frac{1}{2q}} \varepsilon_R^{-1+\frac{1}{2q}} t^{-1+\frac{1}{2q}},
\quad w_m t \le x \le w_+ t.
\]
\item[(8)] Under the assumptions of (6) with $\alpha=8$, that is, $T_0=\varepsilon_R^{-8}$, there exists a constant \(C_q>0\) such that, for \(0<\varepsilon_R<\delta_1\),
\begin{equation}\label{eq:inviscid-comparison}
    \bigl\|u^R(T_0,\cdot)-u^r(\cdot/T_0)\bigr\|_{H^1(\mathbb{R})}\le C_q\bigl(\delta_R^{\frac12}+\delta_R\bigr)\varepsilon_R^{\frac12}.
\end{equation}
\end{itemize}
\end{lemma}
The proof is provided in Section \ref{sec:first-pf}.

\subsection{Composite Wave of Shock and Rarefaction}\label{subsec:composite}
We now turn to the composite wave. Under the standing assumption
\begin{equation}\label{eq:far-field}
    u_- < -\frac{u_-}{2} < u_+,
\end{equation}
or equivalently $u_-<0$ and $u_+>-u_-/2$, the Riemann problem \eqref{Riemann} is solved by the degenerate Oleinik shock connecting $u_-$ to the intermediate state $u_m:=-u_-/2>0$, followed by the rarefaction wave connecting $u_m$ to $u_+$. Accordingly, the long-time behavior of the Cauchy problem \eqref{eq:gKdVB} is governed by the corresponding composite of a degenerate viscous-dispersive Oleinik shock and a rarefaction wave, described by the superposition
\begin{equation}\label{eq:composite}
    u^S\big(x-\sigma t - X(t);u_-,u_m\big) + u^r\big(\tfrac{x}{t};u_m,u_+\big) - u_m,
\end{equation}
where $u^S$ is the degenerate Oleinik shock of Theorem \ref{thm:shock_properties}, $u^r$ is the rarefaction wave of Section \ref{subsec:inviscid-rare}, and the time-dependent shift $X(t)$ on the Oleinik shock is constructed in Section \ref{sec:weight-shift}. We refer to Figure \ref{fig:composite} for the profiles of degenerate shock, rarefaction wave and their composition.

\begin{figure}[h!]
	\includegraphics[width=0.9\textwidth]{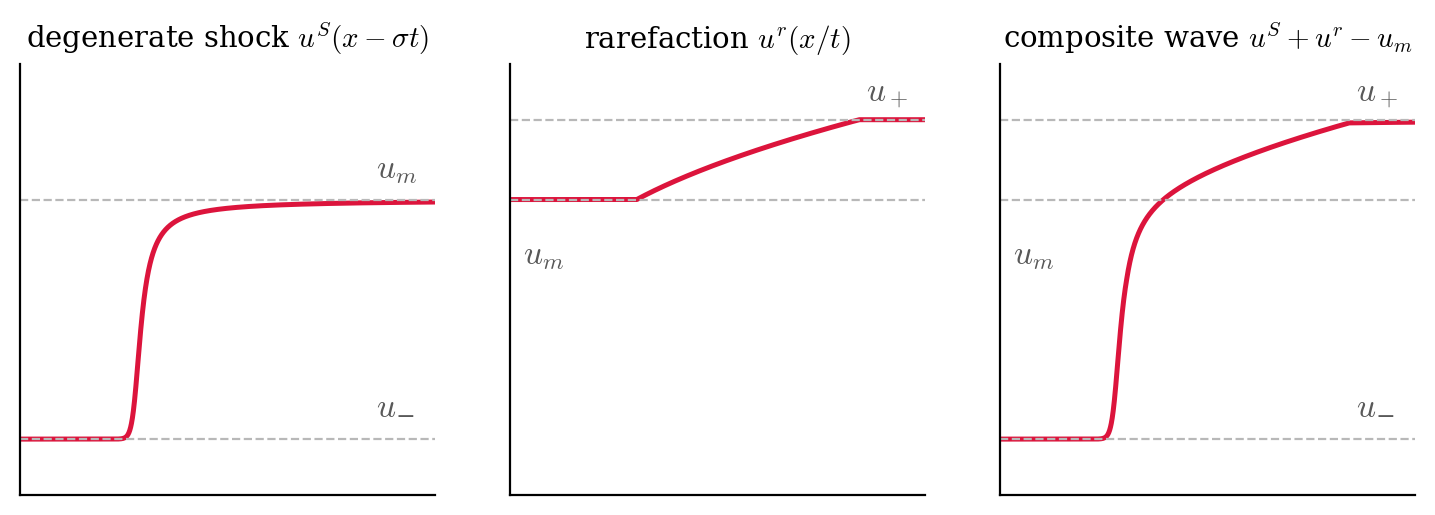}
	\caption{Profiles of the degenerate shock $\us$ (left), rarefaction wave $u^r$ (middle), and the composite wave $u^S+u^r-u_m$ (right). Observe that the degenerate shock has different decay rates at its two end, and the shock fronts of the degenerate shock and the rarefaction wave are attached.}
	\label{fig:composite}
\end{figure}

\subsection{Main Result} \label{subsec:main-stability}
With the composite wave \eqref{eq:composite} and the smooth approximate rarefaction \eqref{eq:uR-def} in hand, we state our main result on stability.

\begin{theorem} \label{thm:main}
Suppose that the far-field states $u_\pm$ satisfy \eqref{eq:far-field}, and set $u_m:=-u_-/2$ with the wave strengths $\delta_S:=u_m-u_-=3u_m$ and $\delta_R:=u_+-u_m$. There exist positive constants $\delta_0,\varepsilon_0$ such that the following holds. If the dispersion coefficient $\kappa>0$ and the wave strengths satisfy
\begin{equation} \label{eq:small-delta-constraint-main}
    \kappa u_m^2 \le \delta_0 \qquad\text{and}\qquad \delta_R \le \frac{5}{18} \delta_S \; \Bigl(= \frac{5}{6}u_m\Bigr),
\end{equation}
then there exists a constant $0<\varepsilon_R<1$ (depending only on the wave strengths $\delta_S$ and $\delta_R$) such that the following holds.

Let $u^R$ be the approximate rarefaction wave \eqref{eq:uR-def} with parameter $\varepsilon_R$, and set $T_0:=\varepsilon_R^{-8}$. If the initial data $u_0$ satisfies
\begin{equation}\label{eq:initial-smallness}
\begin{aligned}
    &\bigl\| u_0 - \bigl( u^S(\cdot\,;u_-,u_m) + {u^R(T_0,\,\cdot+\sigma T_0)} - u_m \bigr) \bigr\|_{L^2(\mathbb{R})} \\
    &\qquad + \bigl\| u_{0x} - \bigl( u^S_x(\cdot\,;u_-,u_m) + {u^R_x(T_0,\,\cdot+\sigma T_0)} \bigr) \bigr\|_{L^2(\mathbb{R})} < \varepsilon_0,
\end{aligned}
\end{equation}
then the Cauchy problem \eqref{eq:gKdVB} admits a unique global-in-time solution $u(t,x)$, and there exists a Lipschitz continuous shift $X(t)$ such that 
\begin{align*}
u(t,x) - \left( \us\bigl(x-\sigma t-X(t);u_-,u_m\bigr) + u^R\bigl(t+T_0,\,x-X(t)+\sigma T_0\bigr) - u_m \right) 
&\in C\big([0,\infty);H^1(\mathbb{R})\big),\\
\partial_x\left[\, u(t,x) - \left( \us\bigl(x-\sigma t-X(t);u_-,u_m\bigr) + u^R\bigl(t+T_0,\,x-X(t)+\sigma T_0\bigr) - u_m \right) \right] &\in L^2\big(0,\infty;H^1(\mathbb{R})\big).\\
%u_{xx}(t,x) - \us_{xx}\big(x-\sigma t - X(t);u_-,u_m\big) &\in L^2\big(0,\infty;L^2(\mathbb{R})\big).
\end{align*}
Moreover, the composite wave is time-asymptotically stable, i.e.,
\[
    \lim_{t\to\infty}\,\sup_{x\in\mathbb{R}}\,\Bigl| u(t,x) - \left( \us\big(x-\sigma t - X(t);u_-,u_m\big) + u^r\big(\tfrac{x}{t};u_m,u_+\big) - u_m \right) \Bigr| = 0,
\qquad
    \lim_{t\to\infty}|\dot{X}(t)| = 0.
\]
\end{theorem}

\begin{remark}\label{rem:hypotheses}
\textup{(i)} The hypotheses \eqref{eq:small-delta-constraint-main} do not require $\delta_S$ or $\delta_R$ to be small. They restrict only the ratio $\delta_R/\delta_S$ and the quantity $\kappa u_m^{2}$, which also governs the monotonicity of the shock in Theorem~\ref{thm:shock_properties}, so that for fixed $\kappa$ the strengths may be as large as $u_m\le\sqrt{\delta_0/\kappa}$. The constants $\varepsilon_R$ and $\varepsilon_0$ depend on $\delta_S$ and $\delta_R$ but not on $\kappa$.
The ratio condition $\delta_R\le\frac{5}{18}\delta_S$ is used in the proof of Lemma~\ref{lem:GS-N} to ensure that the coefficient in \eqref{eq:GSN-rare} is positive.
    
\noindent\textup{(ii)} In the viscous case \cite{HWZ-MA}, where $\delta_R$ is assumed small, the initial perturbation is measured against $u^S(x)$ for $x<0$ and $u^S(x)+u_+-u_m$ for $x>0$. Since $\delta_R$ is not assumed small here, the approximate rarefaction $u^R$ enters \eqref{eq:initial-smallness}. By Lemma~\ref{lem:approx_rarefaction}(8) and the choice \eqref{eq:eR-choice} of $\varepsilon_R$, Theorem~\ref{thm:main} also holds with $\varepsilon_0/2$ and $u^r(\cdot/T_0+\sigma)$ in place of $\varepsilon_0$ and $\ur(T_0,\cdot+\sigma T_0)$ in \eqref{eq:initial-smallness}.
\end{remark}

\begin{remark}\label{rem:extensions}
For a general viscosity $\mu>0$, the scaling $v(t,x):=u(\mu t,\mu x)$ normalizes the viscosity to $1$ and replaces $\kappa$ by $\kappa/\mu^{2}$, so that Theorem~\ref{thm:main} holds for any $\mu>0$ under the condition $\kappa u_m^{2}/\mu^{2}\le\delta_0$, with the initial condition imposed on the rescaled datum. Since the flux is odd, the reflection $u\mapsto-u$ leaves \eqref{eq:gKdVB} invariant and maps $u_\pm$ to $-u_\pm$. Applied to $-u$, Theorem~\ref{thm:main} covers the decreasing configuration $u_+<-u_-/2<u_-$, for which the composite wave consists of a decreasing degenerate Oleinik shock followed by a decreasing rarefaction.
\end{remark}

\begin{remark}\label{rem:shift}
Since $\dot{X}(t)\to0$, the shift grows at most sub-linearly, i.e., $X(t)/t\to0$, so that the shifted shock $\us(x-\sigma t-X(t))$ keeps its profile time-asymptotically.
\end{remark}

The proof, given in Section \ref{sec:pfmain}, is based on the method of \(a\)-contraction with shifts \cite{KV-JEMS21}.
The key ingredient of the \(a\)-contraction method is the following Poincar\'e-type inequality, which is used to handle the localization effect induced by the shock profile:
\begin{lemma}\label{lem:poincare}\cite[Lemma 2.9]{KV-JEMS21} 
For any $\psi:[0,1] \to \mathbb{R}$ satisfying $\int_0^1 y(1-y)|\psi'(y)|^2 \, dy <+ \infty$,
\begin{equation*}
2 \int_0^1 \psi^2 (y)\, dy - 2 \left( \int_0^1 \psi (y)\, dy \right)^2 \le  \int_0^1 y(1-y)|\psi'(y)|^2 \, dy.
\end{equation*}
\end{lemma}

\section{Proofs of Theorem \ref{thm:shock_properties} and Lemmas \ref{lem:Burgers}--\ref{lem:approx_rarefaction}} \label{sec:first-pf}
\setcounter{equation}{0}
In this section, we prove Theorem \ref{thm:shock_properties} and Lemmas \ref{lem:Burgers}--\ref{lem:approx_rarefaction}.
The former establishes the properties of the degenerate shock, while the latter concerns those of the approximate rarefaction wave.

\subsection{Proof of Theorem \ref{thm:shock_properties}}
$\bullet$ Proof of \eqref{eq:key-ineq}: The proof is inspired by the argument in \cite{CEKS1}, but we refine it to suit the present setting. For the case of the standard KdV--Burgers equation, see \cite{CEKS2}.

Throughout the proof, we use the scale-invariant variables
\begin{equation}\label{eq:scale-inv-var}
    z := \frac{\us}{s}\in(-2,1), \qquad \zeta := s^2 \x, \qquad \nu := \k s^2,
\end{equation}
so that $z(\zeta) = \us(\x)/s$. Letting $\,'=d/d\zeta$, we obtain
\begin{equation*}
    \us_\x = s^3 z', \qquad \us_{\x\x} = s^5 z'', \qquad \us_{\x\x\x} = s^7 z'''.
\end{equation*}
In terms of the variables \eqref{eq:scale-inv-var}, the monotonicity condition \eqref{threshold} is equivalent to
\begin{equation}\label{eq:nu-bound}
    \nu \le \frac{1}{36}.
\end{equation}
Furthermore, \eqref{eq:shock-eq} simplifies to
\begin{equation}\label{eq:scaled-shock-eq}
   p(z) = z' - \nu z'', \qquad \text{with} \quad p(z) := (z-1)^2(z+2),
\end{equation}
while the inequality \eqref{eq:key-ineq} becomes
\begin{equation}\label{eq:key-ineq-2}
    (4\sqrt{3}-6)p(z) \le z' \le 2p(z).
\end{equation}

We analyze \eqref{eq:scaled-shock-eq} and prove \eqref{eq:key-ineq-2} by reducing the second order ODE to a first-order ODE system. Letting $v:=z'$, we obtain
\begin{equation}\label{eq:planar}
\left\{
\begin{aligned}
    z' &= v, \\ 
    v' &= \frac{v-p(z)}{\nu},
\end{aligned}
\right.
\end{equation}
whose equilibria are $(z^-,0)=(-2,0)$ and $(z^+,0)=(1,0)$ that correspond to the two ends of the profile. The Jacobian at an equilibrium $(z^\pm,0)$ is
\begin{equation*}
    J = \begin{pmatrix} 0 & 1 \\ -\dfrac{p'(z^\pm)}{\nu} & \dfrac1\nu \end{pmatrix}, \qquad p'(z)=3(z^2-1),
\end{equation*}
with
\begin{equation*}
 \operatorname{tr}J=1/\nu \qquad \text{and} \quad   \det J=p'(z^\pm)/\nu.
\end{equation*}
The stability type of each equilibrium is determined by the signs of the trace and determinant, which in this case depend solely on the sign of $p'(z)=3(z^2-1)$.

At the right endpoint $z^+=1$, we have $p'(1)=0$, so $\det J=0$ and the eigenvalues are
\begin{equation*}
    \lambda_1^+ = 0, \qquad \lambda_2^+ = \frac1\nu > 0.
\end{equation*}
The center subspace $E^c$ is spanned by $(1,0)$, and the unstable subspace $E^u$ is spanned by $(\nu,1)$. Since there is no stable subspace, any orbit converging to this equilibrium as $\zeta\to+\infty$ must lie on the center manifold. By introducing $w:=z-1$ to shift the equilibrium to the origin, we can rewrite the planar system \eqref{eq:planar} as
\begin{equation*}
\left\{
\begin{aligned}
    w' &= v, \\ 
    v' &= \frac{v-p(1+w)}{\nu},
\end{aligned}
\right.
\qquad p(1+w)=w^2(w+3).
\end{equation*}
Then, by the center manifold theorem \cite{Carr}, there exists a smooth function $g$ with $g(0)=g'(0)=0$, whose graph
\begin{equation*}
    W^c = \big\{(w,v)\,:\, v=g(w)\big\}
\end{equation*}
is locally invariant and tangent to $E^c$ at the origin. Since the relation $v=g(w)$ holds along $W^c$, we differentiate this relation with respect to $\zeta$ and apply \eqref{eq:planar} to yield 
\begin{equation}\label{g-relation}
	\frac{v-p(1+w)}{\nu}=g'(w)v\qquad\Leftrightarrow\qquad \nu g'(w)g(w)=g(w)-p(1+w).
\end{equation}
Since $g(0)=g'(0)=0$, we substitute the expansion $g(w)=\sum_{k\ge2}a_k w^k$ into \eqref{g-relation} and match the coefficients of $w$ to get $a_2=3$ and $a_3=1+18\nu$. Thus, we have for $|w|\ll1$,
\begin{equation}\label{eq:cm-expansion}
    g(w) = 3w^2 + (1+18\nu)\,w^3 + O(w^4).
\end{equation}

We now construct the profile $z(\zeta)$ and simultaneously prove the upper bound of \eqref{eq:key-ineq-2} by continuing the center manifold curve $v=g(w)=g(z-1)=:h(z)$ backward and trapping it between two barriers. Consider the region
\begin{equation}\label{eq:trap-region}
    \mathcal{S} := \Big\{(z,v)\,:\, -2\le z\le1,\ \tfrac12\,p(z)\le v\le 2\,p(z)\Big\},
\end{equation}
and recall $p(z)>0$ for all $z\in(-2,1)$. On the upper boundary $v=2p(z)$, the system \eqref{eq:planar} gives
\begin{equation*}
    \big(v-2p(z)\big)'\big|_{v=2p(z)}
    = \frac{v-p(z)}{\nu}-2p'(z)\,v\,\bigg|_{v=2p(z)}
    = p(z)\Big(\frac{1}{\nu}-12(z^2-1)\Big) > p(z)\Big(\frac{1}{\nu}-36\Big)\ge0,
\end{equation*}
where we used the bound $z^2-1<3$ on $(-2,1)$ and the condition \eqref{eq:nu-bound}.\\
Similarly, on the lower boundary $v=\frac{1}{2}p(z)$, we have
\begin{equation*}
    \Big(v-\frac{1}{2}p(z)\Big)'\Big|_{v=\frac{1}{2}p(z)}
    = p(z)\Big(-\frac{1}{2\nu}-\frac34(z^2-1)\Big)
    \le p(z)\Big(-\frac{1}{2\nu}+\frac34\Big)<0,
\end{equation*}
where we used the bound $z^2-1\ge-1$ and the condition \eqref{eq:nu-bound}.
Thus, the vector field points strictly outward along both the upper boundary ($v-2p$ is increasing) and the lower boundary ($v-\frac{1}{2}p$ is decreasing). Equivalently, the region $\mathcal{S}$ is strictly invariant under the backward flow. We refer to Figure \ref{fig:phase_plane} for the schematic illustration.

\begin{figure}
	\includegraphics[width=0.9\textwidth]{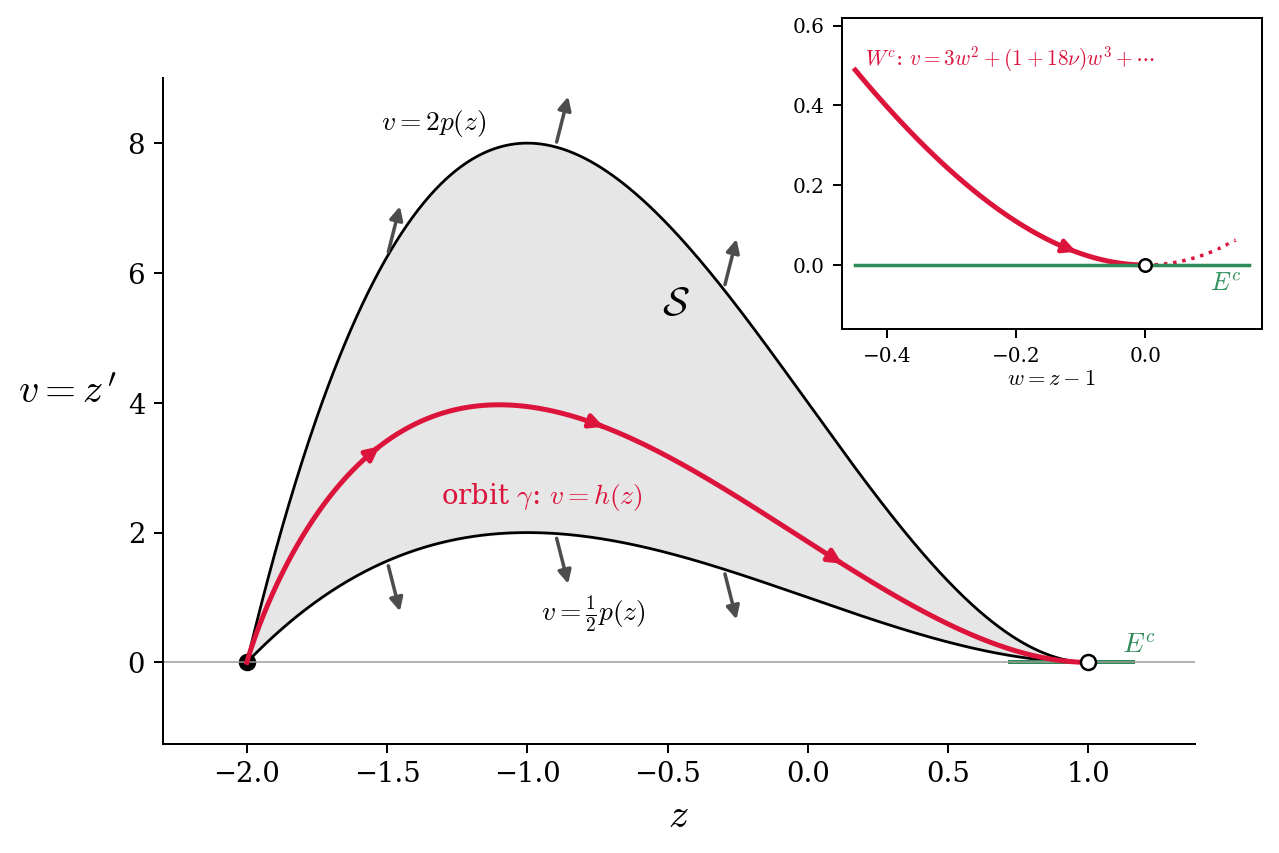}
	\caption{The heteroclinic orbit $\gamma$ (red) trapped in 
		the backward-invariant region $\mathcal{S}$ between $v=\tfrac12 p(z)$ and $v=2p(z)$. Arrows show the outward crossing of the forward flow on the boundary of $\mathcal{S}$. Inset: tangential approach to $E^c$ along the center manifold $W^c$ near the degenerate end $(1,0)$.}
	\label{fig:phase_plane}
\end{figure}

From $h(z)=g(z-1)$, we get
\[h'(z) = g'(w) = 6w+O(w^2)=6(z-1)+O((z-1)^2),\]
and therefore
\begin{equation}\label{eq:R-value}
\lim_{z\to 1}h'(z) = 0, \qquad
\lim_{z\to1}\frac{p(z)}{h(z)}=\lim_{z\to1}\frac{(z-1)^2(z+2)}{3(z-1)^2+O(z-1)^3}=1.
\end{equation}
Hence, along the center manifold curve $v=h(z)$, we have
\begin{equation*}
\lim_{z\to 1}\frac{v}{p(z)} = 1 .
\end{equation*}
Since $1\in(\tfrac12,2)$, this local center manifold lies inside $\mathcal{S}$ for $z\in(z_1,1)$ with some $z_1<1$.\\
Let $\gamma(\zeta)=(z(\zeta),v(\zeta))$ be the orbit on $W^c$ with
\begin{equation*}
\lim_{\zeta\to+\infty}\gamma(\zeta)=(1,0), \qquad z(\zeta)<1,
\end{equation*}
and fix $\zeta_0$ with $z(\zeta_0)\in(z_1,1)$. On $(z_1,1)$ the expansion \eqref{eq:cm-expansion} gives $z'=v=h(z)=3(z-1)^2+O((z-1)^3)>0$, so $z(\zeta)$ increases to $1$ for $\zeta\ge\zeta_0$ and $\gamma(\zeta)\in\mathcal{S}$ for all $\zeta\ge\zeta_0$.
Since $\mathcal{S}$ is compact and backward invariant, $\gamma(\zeta)\in\mathcal{S}$ for all $\zeta\le\zeta_0$ as well, and therefore $\gamma\subset\mathcal{S}$.

As $z'=v\ge\tfrac12 p(z)>0$ for $z\in(-2,1)$, the orbit $\gamma$ is strictly monotone in $z$ and bounded. Consequently, it must converge to an equilibrium as $\zeta\to-\infty$. Since $(-2,0)$ is the unique equilibrium in the region $z<1$, the backward limit should be $(-2,0)$. Combining this with the forward convergence to $(1,0)$ as $\zeta\to+\infty$, we have
\begin{equation*}
\lim_{\zeta\to-\infty}\gamma(\zeta) = (-2,0), \qquad \lim_{\zeta\to+\infty}\gamma(\zeta) = (1,0).
\end{equation*}
Thus, $\gamma$ represents the orbit of a smooth monotone profile $\us$ connecting the two endpoints.
Since $z$ is strictly increasing along $\gamma$, the orbit can be parameterized as the graph $v=h(z)$ of the function $h$, which is an extension of $h$ introduced earlier.
Moreover, $\us$ is unique up to translation, since any orbit converging to $(1,0)$ lies on the one-dimensional center manifold, on which $z'=h(z)>0$ near $z=1$.
The inclusion $\gamma\subset\mathcal{S}$ then yields
\begin{equation}\label{eq:h-upper}
z'=v=h(z) \le 2p(z), \qquad \text{for } z\in(-2,1).
\end{equation}
This proves the upper bound in \eqref{eq:key-ineq-2}. To attain the lower bound, we define $L_m(z):=h(z)-mp(z)$ for a fixed $m\in(0,4\sqrt3-6)$. We note that $z''=h'(z)z' = h'(z)h(z)$, and it follows from \eqref{eq:scaled-shock-eq} that 
\begin{equation}\label{eq:Hp-formula}
	\nu h(z)h'(z) = h(z)-p(z), \qquad\text{or equivalently,}\qquad
	h'(z) = \frac{h(z)-p(z)}{\nu h(z)}.
\end{equation}
We first claim that $L'_m<0$ at every zero of $L_m$. If $L_m(z_0)=0$, then \eqref{eq:Hp-formula} yields $h'(z_0)=\frac{m-1}{\nu m}<0$, and we have
\begin{equation}\label{est:Lmprime}
	\begin{aligned}
		L_m'(z_0) = h'(z_0)-mp'(z_0) &= \frac{m-1}{\nu m}-3m(z_0^2-1)
		\le \frac{m-1}{\nu m}+3m \\
		&= \frac{1}{\nu}\Big(\frac{m-1}{m}+3m\nu\Big)
		\le \frac{12(m-1)+m^2}{12\,\nu m}<0.
	\end{aligned}
\end{equation}
Here, we used the condition \eqref{eq:nu-bound} and the fact that $12(m-1)+m^2<0$ for $0<m<4\sqrt3-6<1$.
Now, we assume that there exists $z^*\in(-2,1)$ such that $L_m(z^*)<0$. If $L_m(z)\ge 0$ at some point $z\in (z^*,1)$, there exists a first zero 
\[z_0:=\inf \left\{z>z^*~:~L_m(z)=0\right\}.\]
Since $L_m(z)<0$ for $z\in(z^*,z_0)$, it must hold that $L_m'(z_0)\ge 0$, which contradicts to \eqref{est:Lmprime}. Thus, we conclude that $L_m(z)<0$ for all $z\in(z^*,1)$. However, since \eqref{eq:R-value} should hold and $m<1$, we must have 
\[\lim_{z\to1}\frac{L_m(z)}{p(z)}=1-m>0.\] 
This again contradicts to $L_m(z)<0$ for all $z\in (z^*,1)$. Therefore $h(z)\ge mp(z)$ for all $z\in(-2,1)$. Since this inequality should hold for all $m<4\sqrt{3}-6$, this implies
\begin{equation}\label{eq:h-lower}
    z'=v=h(z)\ge(4\sqrt3-6)p(z) \qquad\text{for}\quad z\in (-2,1).
\end{equation}
Combining \eqref{eq:h-upper} and \eqref{eq:h-lower} yields \eqref{eq:key-ineq-2}, and consequently \eqref{eq:key-ineq}.

\vspace{2mm}
\noindent $\bullet$ Proof of \eqref{eq:left_decay_u}--\eqref{eq:right_decay_up}: Returning to the original variables via \eqref{eq:scale-inv-var}, we now use a standard Gr\"onwall argument (see for instance, \cite{EEKO-JDE,KV-JEMS21}) so as to derive the decay estimates from \eqref{eq:key-ineq}. Since $\us(0)=0$, we have $-3s<\us(\x)-s<-s$ for $\x<0$; thus, \eqref{eq:key-ineq} implies
\begin{equation}\label{left}
    (4\sqrt{3}-6) s^2 (\us+2s)
    \le (\us+2s)_\x
    \le 18s^2 (\us+2s), \qquad \forall \x<0.
\end{equation}
Thus, the Gr\"onwall inequality yields
\begin{equation*}
    2s\,e^{-18s^2\abs{\x}} \le \us(\x)+2s \le 2s\,e^{-(4\sqrt{3}-6)s^2\abs{\x}}, \qquad \forall \x<0,
\end{equation*}
and substituting back into \eqref{left} gives
\begin{equation*}
    2(4\sqrt{3}-6)s^3\,e^{-18s^2\abs{\x}} \le \us_\x(\x) \le 36s^3\,e^{-(4\sqrt{3}-6)s^2\abs{\x}}, \qquad \forall \x<0.
\end{equation*}
These bounds establish \eqref{eq:left_decay_u} and \eqref{eq:left_decay_up}. For $\x>0$, we have $2s<\us+2s<3s$, and so \eqref{eq:key-ineq} gives
\begin{equation}\label{right}
    2(4\sqrt{3}-6)s(\us-s)^2
    \le (\us-s)_\x
    \le 6s(\us-s)^2, \qquad \forall \x>0,
\end{equation}
and applying the same Gr\"onwall argument yields
\begin{equation*}
    \frac{s}{1+6s^2\abs{\x}} \le s-\us(\x) \le \frac{s}{1+2(4\sqrt{3}-6)s^2\abs{\x}}, \qquad \forall \x>0,
\end{equation*}
which, together with \eqref{right}, gives
\begin{equation*}
    \frac{2(4\sqrt{3}-6)s^3}{(1+6s^2\abs{\x})^2} \le \us_\x(\x) \le \frac{6s^3}{(1+2(4\sqrt{3}-6)s^2\abs{\x})^2}, \qquad \forall \x>0.
\end{equation*}
These estimates prove \eqref{eq:right_decay_u} and \eqref{eq:right_decay_up}.

\vspace{2mm}
\noindent $\bullet$ Proof of \eqref{eq:second_deriv_bound}: Returning to the scale-invariant variables, we define
\begin{equation} \label{eq:F-def}
    F := h'(z) = \frac{z''}{z'}  =\frac{\us_{\x\x}}{s^2\,\us_\x},
\end{equation}
which is well defined since $z'>0$. We claim that
\begin{equation}\label{key:ineq-der}
    -3 \le F \le 18.
\end{equation}
We divide \eqref{eq:scaled-shock-eq} by $z'$ to get $\nu F = 1 - p(z)/z'$, and differentiate it with respect to $\zeta$ to obtain
\begin{equation*}
\nu F' =-\frac{d}{d\zeta}\left(\frac{p(z)}{z'}\right)
=-\frac{p'(z)(z')^2-p(z)z''}{(z')^2}=-p'(z)+\frac{p(z)}{z'}F
=- 3(z^2-1)+  F - \nu F^2.
\end{equation*}
Note that $-3\le 3(z^2-1)<9$ for all $z\in(-2,1)$. If $F(\zeta)=18$ at some $\zeta$, then
\begin{equation*}
\nu F' = - 3(z^2-1) + 18 - 18^2\nu  > 0,
\end{equation*}
where we used \eqref{eq:nu-bound} and $3(z^2-1)<9$.
Similarly, if $F(\zeta)=-3$ at some $\zeta$, then
\begin{equation*}
    \nu F' = - 3(z^2-1) -3 - 9\nu  \le -9\nu < 0.
\end{equation*}
Thus, if $F(\zeta_0)>18$ at some $\zeta_0$, then $F(\zeta)>18$ for all $\zeta>\zeta_0$. Likewise, if $F(\zeta_0)<-3$, then $F(\zeta)<-3$ for all $\zeta>\zeta_0$. In either case, this contradicts the limit $F\to 0$ as $z\to1$ established in \eqref{eq:R-value}. Hence, $F$ remains in $[-3,18]$ for all $\zeta$, which proves \eqref{key:ineq-der} and \eqref{eq:second_deriv_bound}.

\vspace{2mm}
\noindent $\bullet$ Proof of \eqref{eq:third_deriv_bound}--\eqref{eq:absbdd}: Define
\begin{equation} \label{eq:G-def}
    G := h'(z)^2 + h(z)h''(z) = \frac{z'''}{z'}=\frac{\us_{\x\x\x}}{s^4\,\us_\x}.
\end{equation}
We claim that
\begin{equation}\label{key:ineq-der2}
    -60 \le G \le 324.
\end{equation}
Differentiating \eqref{eq:scaled-shock-eq} with respect to $\zeta$ twice and substituting \eqref{eq:F-def} and \eqref{eq:G-def} yield
\begin{equation*}
    \nu G' = G - F^2 - 6z z'.
\end{equation*}
Suppose that $G(\zeta)=324$ at some $\zeta$. For $\zeta<0$, since $z<0$ and $z'>0$, we have $-6zz'>0$.\\
Then, applying \eqref{key:ineq-der} gives
\begin{equation*}
    \nu G' = 324 - F^2 - 6zz' > 324 - F^2 \ge 0.
\end{equation*}
For $\zeta\ge0$, we have $z\ge0$. Differentiating \eqref{eq:scaled-shock-eq} with respect to $\zeta$ and dividing by $z'$ gives
\[\nu G=F-3(z^2-1),\quad\mbox{that is}, \quad F=\nu G+3(z^2-1).\] 
As $z^2-1 < 0$ on $z \in [0,1)$, we obtain $F \le \nu G =324\nu \le 9$. Meanwhile, $6zz'\le36$ by \eqref{eq:right_decay_up}. Hence,
\begin{equation*}
    \nu G' = 324 - F^2 - 6zz' \ge 324 - 81 - 36 > 0.
\end{equation*}
On the other hand, suppose that $G(\zeta)=-60$ at some $\zeta$. For $\zeta\ge0$, we simply have
\begin{equation*}
    \nu G' = -60 - F^2 - 6zz' \le -60 < 0.
\end{equation*}
For $\zeta<0$, we have $z\in(-2,0)$ and the inequality \eqref{eq:key-ineq-2} implies 
\begin{equation*}
    |6zz'| \le -12z(z-1)^2(z+2) \le 27+18\sqrt{3} < 60, \quad\mbox{for}\quad z\in(-2,0).
\end{equation*}
Hence,
\begin{equation*}
    \nu G' = -60 - F^2 - 6zz' < -60 + 60 = 0.
\end{equation*}
Therefore, if $G(\zeta_0)>324$ at some $\zeta_0$, then $G(\zeta)>324$ for all $\zeta>\zeta_0$. Likewise, if $G(\zeta_0)<-60$, then $G(\zeta)<-60$ for all $\zeta>\zeta_0$. Since the orbit lies on $W^c$ as $z\to1$, we notice that $G=(g')^2+gg''\to0$. Therefore, either case yields contradiction and $G$ remains in $[-60,324]$, proving \eqref{key:ineq-der2} and \eqref{eq:third_deriv_bound}.

Finally, \eqref{eq:absbdd} follows immediately. This completes the proof of Theorem \ref{thm:shock_properties}. \qed

\subsection{Proof of Lemma \ref{lem:Burgers}}
%The global smooth solution $w(t,x)$ to \eqref{eq: inviscid Burgers equation} satisfies the following quantitative properties.
Although the proof follows the same argument of \cite{MN92}, we provide a detailed proof of the quantitative estimates (2), (3), (4), (6), and (7) for the reader's convenience; in particular the estimates in (2) with $j=3,4$ and the refined bounds (6)--(7), which are not needed in \cite{MN92}. Indeed, since $w_0$ is smooth and strictly increasing, the method of characteristics yields the unique global smooth solution
\[
w(t,x) = w_0(x_0(t,x)),\qquad x = x_0(t,x)+w_0(x_0(t,x))  t,
\]
from which (1) is immediate, while the convergence (5) follows from the same classical argument. Differentiating the representation of $w(t,x)$ yields
\begin{align*}
	&w_x(t,x) = \frac{w'_0(x_0)}{1+w'_0(x_0)t},\qquad\qquad\quad\quad\quad
    w_{xx}(t,x) = \frac{w''_0(x_0)}{(1+w'_0(x_0)t)^3},\\
	&w_{xxx}(t,x) = \frac{w'''_0(x_0)}{(1+w'_0(x_0)t)^4}-\frac{3t(w''_0(x_0))^2}{(1+w'_0(x_0)t)^5},\\
	&w_{xxxx}(t,x) = \frac{w''''_0(x_0)}{(1+w'_0(x_0)t)^5}-\frac{10t w''_0(x_0)w'''_0(x_0)}{(1+w'_0(x_0)t)^6}+\frac{15t^2(w''_0(x_0))^3}{(1+w'_0(x_0)t)^7}.
\end{align*}
Furthermore, it is straightforward to get
\begin{align*}
	&w'_0(x_0) = \frac{ K_q \tilde{w}   \e_R }{2(1+(\e_R x_0)^2)^q},\quad\quad
    w''_0(x_0) = \frac{-q K_q\tilde{w} \e_R^3 x_0 }{(1+(\e_R x_0)^2)^{q+1}},\\
	&w'''_0(x_0) = \frac{-q K_q \tilde{w}   \e_R^3}{(1+(\e_R x_0)^2)^{q+1}}+\frac{2q(q+1) K_q\tilde{w}   \e_R^5 x_0^2}{(1+(\e_R x_0)^2)^{q+2}},\\
	&w''''_0(x_0) = \frac{6q(q+1) K_q\tilde{w}  \e_R^5 x_0}{(1+(\e_R x_0)^2)^{q+2}}-\frac{4q(q+1)(q+2) K_q\tilde{w}   \e_R^7 x_0^3}{(1+(\e_R x_0)^2)^{q+3}}.
\end{align*}	
Comparing these formulas with that of $w'_0$, and using $|z|^k\le(1+z^2)^{k/2}$ for $k\in\{1,2,3\}$, we obtain
\begin{equation}\label{eq:w0-ratio}
	|w_0^{(j+1)}(x_0)| \le C_q \e_R^{ j} w'_0(x_0) \bigl(1+(\e_R x_0)^2\bigr)^{-1/2},\qquad j=1,2,3.
\end{equation}
Substituting \eqref{eq:w0-ratio} into the representations above, we obtain
\[
	|w_{xx}| = \frac{|w''_0|}{(1+w'_0t)^3} \le C_q \e_R \frac{w'_0}{1+w'_0t} = C_q \e_R w_x,
\]
and in the same way, using additionally $\frac{w'_0(x_0)t}{1+w'_0(x_0)t}\le1$ for the terms carrying powers of $t$, we get 
\begin{equation}\label{eq:wxx-pointwise}
	|\partial_x^j w(t,x)| \le C_q \e_R^{ j-1} w_x(t,x),\qquad j=2,3,4,\quad \forall  t\ge0,\ x\in\R.
\end{equation}

\noindent\textbf{Proof of} (2)\textbf{:} For the first-order derivative, we use $dx = (1+w'_0(x_0)t)\, dx_0$ to have
\begin{align*}
	\int_{\R}|w_x|^p\,dx = \int_{\R}\frac{|w'_0(x_0)|^p}{(1+w'_0(x_0)t)^{p-1}}\,dx_0
    &\le 2\int_{0}^\infty|w'_0(x_0)|^p\,dx_0\\
	&\le 2\Bigl(\frac{K_q\tilde{w}\,\e_R}{2} \Bigr)^p\int_0^\infty \frac{1}{(1+(\e_R x_0)^2)^{pq}}\,dx_0\le C \tilde{w}^p\e_R^{p-1}.
\end{align*}
On the other hand, we also estimate the $L^p$-norm of $w_x$ as
\begin{align*}
	\int_{\R}|w_x|^p\,dx &= \int_{\R}\frac{|w'_0(x_0)|^p}{(1+w'_0(x_0)t)^{p-1}}\,dx_0=2\int_{0}^\infty t^{-p+1}|w'_0(x_0)|\frac{|w'_0(x_0)t|^{p-1}}{(1+w'_0(x_0)t)^{p-1}}\,dx_0\\
	&\le 2\,t^{-p+1}\int_0^\infty |w'_0(x_0)|\,dx_0 \,=\,\tilde{w}\,t^{-p+1}.
\end{align*}
Thus, we get
\[\|w_x\|_{L^p}\lesssim \min\big\{\tilde{w} \e_R^{1-1/p},\tilde{w}^{1/p}t^{-1+1/p}\big\}.\]
For the second-order derivative, we derive
\begin{align*}
	\int_{\R} |w_{xx}|^p\,dx
    =\,2\int_0^\infty\frac{|w''_0(x_0)|^p}{(1+w'_0(x_0)t)^{3p-1}}\,dx_0
    &\le 2\int_0^\infty |w''_0(x_0)|^p\,dx_0\\
	&\le C\e_R^{2p}\tilde{w}^p\int_0^\infty \frac{(\e_R x_0)^p}{(1+(\e_R x_0)^2)^{p(q+1)}}\,dx_0\le C\e_R^{2p-1} \tilde{w}^p.
\end{align*}
Furthermore, we use
\[|w''_0(x_0)|\le C|w'_0(x_0)|^{1-\frac{1}{q}}\frac{\e_R(\e_R\tilde{w})^{1/q}(\e_R x_0)}{(1+(\e_R x_0)^2)^2}\]
to get
\begin{align*}
	\int_{\R} |w_{xx}|^p\,dx &=2\int_0^\infty\frac{|w''_0(x_0)|^p}{(1+w'_0(x_0)t)^{3p-1}}\,dx_0\le C\int_0^\infty \frac{|w'_0(x_0)|^{p-p/q}\e_R^p(\e_R\tilde{w})^{p/q}}{(1+w'_0(x_0)t)^{3p-1}}\frac{(\e_R x_0)^p}{(1+(\e_R x_0)^2)^{2p}}\,dx_0\\
	&\le C t^{-p+p/q}\e_R^p(\e_R\tilde{w})^{p/q}\int_0^\infty \frac{(\e_R x_0)^p}{(1+(\e_R x_0)^2)^{2p}}\,dx_0\le Ct^{-p+p/q}\e_R^{p-1}(\e_R \tilde{w})^{p/q}. 
\end{align*}
Therefore, we estimate the $L^p$-norm of the second-order derivative as
\[\|w_{xx}\|_{L^p} \lesssim \min\big\{\tilde{w}\e_R^{2-1/p},\tilde{w}^{1/q}\e_R^{1-1/p+1/q}t^{-1+1/q}\big\}.\]
For the third and fourth order derivatives, the only new ingredient is the interpolated analogue of the bound on $w''_0$ established above: from the formulas for $w'''_0$ and $w''''_0$, together with $|z|^k\le(1+z^2)^{k/2}$, we obtain
\begin{equation}\label{eq:w0-interp}
	|w_0^{(j+1)}(x_0)| \le C_q |w'_0(x_0)|^{1-\frac{j}{q}} \e_R^{ j} (\e_R\tilde{w})^{\frac{j}{q}} \bigl(1+(\e_R x_0)^2\bigr)^{-1},\qquad j=2,3.
\end{equation}
We detail the case $j=3$.
For the first term of $w_{xxx}$, taking $j=2$ in \eqref{eq:w0-interp} and $\frac{w'_0}{1+w'_0t}\le\frac1t$ yield
\[
	\int_0^\infty\frac{|w'''_0|^p}{(1+w'_0t)^{4p-1}}\,dx_0
	\le C\e_R^{2p}(\e_R\tilde{w})^{\frac{2p}{q}}\!\int_0^\infty\!\Bigl(\frac{w'_0}{1+w'_0t}\Bigr)^{p-\frac{2p}{q}}\frac{dx_0}{(1+(\e_R x_0)^2)^{p}}
	\le C t^{-p+\frac{2p}{q}}\e_R^{2p-1}(\e_R\tilde{w})^{\frac{2p}{q}},
\]
where we used $(1+w'_0t)^{-(3p-1+\frac{2p}{q})} \le 1$.
Moreover, the bound on $w''_0$ and $\bigl(\frac{w'_0t}{1+w'_0t}\bigr)^{p}\le1$ yield
\[
	\int_0^\infty\frac{(3t)^p(w''_0)^{2p}}{(1+w'_0t)^{5p-1}}\,dx_0
	\le C\e_R^{2p}(\e_R\tilde{w})^{\frac{2p}{q}}\!\int_0^\infty\!\Bigl(\frac{w'_0}{1+w'_0t}\Bigr)^{p-\frac{2p}{q}}\frac{dx_0}{(1+(\e_R x_0)^2)^{2p}}
	\le C t^{-p+\frac{2p}{q}}\e_R^{2p-1}(\e_R\tilde{w})^{\frac{2p}{q}},
\]
so that 
\[
\|w_{xxx}\|_{L^p}\le C_{p,q} \tilde{w}^{2/q} \e_R^{ 2-\frac1p+\frac2q} t^{-1+\frac2q}.
\]
The fourth order derivative is estimated in the same manner, applying \eqref{eq:w0-interp} with $j=3$ to $w''''_0$ and treating $t w''_0w'''_0$ and $t^2(w''_0)^3$ similarly to the second term above. Since $\e_R<1$, this yields
\[
	\|\partial_x^j w\|_{L^p} \le C_{p,q} \tilde{w}^{\frac{j-1}{q}} \e_R^{ (j-1)-\frac1p+\frac1q} t^{-1+\frac{j-1}{q}},\qquad j=3,4.
\]
The bound $\|\partial_x^j w\|_{L^p}\le C_{p,q} \tilde{w} \e_R^{ j-1/p}$ follows directly from \eqref{eq:wxx-pointwise}, which yields $\|\partial_x^j w\|_{L^p}\le C_q \e_R^{ j-1}\|w_x\|_{L^p}$.
This establishes (2).

\medskip

\noindent\textbf{Proof of} (3)\textbf{:} Using $w(t,x) = w_0(x_0)$, we get
\[|w(t,x)-w_m| = \frac{\tilde{w}}{2}\left(1+K_q\int_{0}^{\e_R x_0}\frac{1}{(1+y^2)^q}\,dy\right)=\frac{K_q\tilde{w}}{2}\int_{-\infty}^{\e_R x_0}\frac{1}{(1+y^2)^q}\,dy.\]
On the other hand, for $x\le w_m t$, we obtain 
\[x_0 = x-w_0(x_0)t< x- w_m t\le 0,\]
which implies $\e_R|x_0|>\e_R|x- w_m t|$.
Thus, for all $t >0$ and $x \le w_m t$,
\[|w(t,x)-w_m|\le C_q \tilde{w}\int_{\e_R|x_0|}^{\infty}\frac{1}{(1+y^2)^q}\,dy\le C_q \tilde{w}\int_{\e_R|x-w_m t|}^\infty \frac{1}{(1+y^2)^q}\,dy\le \frac{C_q\tilde{w}}{(1+(\e_R|x-w_m t|)^2)^{q/2}},\]
where we used $q>1$. Moreover, we have
\[
|w_x(t,x)|
\le w_0'(x_0)
=\frac{\e_R K_q\tilde w}{2(1+(\e_R x_0)^2)^q}
\le \frac{C_q\,\tilde w\,\e_R}{\bigl(1+(\e_R|x-w_m t|)^2\bigr)^q},
\qquad \forall\,t\ge0,\ x\le w_m t .
\]

\noindent\textbf{Proof of} (4)\textbf{:} Similarly, we get
\[|w(t,x)- w_+| = \frac{\tilde{w}}{2}\left(1-K_q\int_0^{\e_R x_0}\frac{1}{(1+y^2)^q}\,dy\right)=\frac{K_q\tilde{w}}{2}\int_{\e_R x_0}^{\infty}\frac{1}{(1+y^2)^q}\,dy.\]
For $x\ge w_+ t$, we obtain
\[
x_0 = x-w_0(x_0)t>x- w_+ t\ge0,
\]
which implies $\e_R|x_0|>\e_R|x- w_+ t|$. Thus, the desired estimate holds by the same argument in (3).\\

\noindent\textbf{Proof of} (6)\textbf{:}
It suffices to prove the bound for $x \le w_m t$, since the case $x \ge w_+ t$ is analogous.

The method of characteristics yields that
\begin{equation} \label{eq: characteristic}
    w(t,x) = w_0(x_0(t,x)),\qquad x = x_0(t,x)+w_0(x_0(t,x)) t.
\end{equation}
Evaluating \eqref{eq: characteristic} at $x = w_m t$, we obtain
\begin{equation} \label{eq:characteristic-relation}
|x_0(t, w_m t)| = \bigl(w_0(x_0(t, w_m t)) - w_m\bigr)\, t > 0, \qquad \forall\, t > 0.
\end{equation}
To make use of \eqref{eq:characteristic-relation}, we analyze the behavior of $w_0(x_0) - w_m$ for $x_0 < 0$.

For \(x_0 < 0\), the change of variables $y = -u$ gives
\begin{equation*}
    w_0(x_0) - w_m
    = \frac{\tilde{w}}{2} K_q \int_{-\infty}^{\varepsilon_R x_0} \frac{dy}{(1+y^2)^q}
    = \frac{\tilde{w}}{2} K_q \int_{|\varepsilon_R x_0|}^{\infty} \frac{du}{(1+u^2)^q}.
\end{equation*}
Note that for any $u \ge |\varepsilon_R x_0|$, we have $u^2 \le 1+u^2 \le u^2\bigl(1+|\varepsilon_R x_0|^{-2}\bigr)$, from which it follows that
\[
\frac{\bigl(1+|\varepsilon_R x_0|^{-2}\bigr)^{-q}}{2q-1}\, |\varepsilon_R x_0|^{-2q+1}
\le \int_{|\varepsilon_R x_0|}^{\infty} \frac{du}{(1+u^2)^q}
\le \frac{|\varepsilon_R x_0|^{-2q+1}}{2q-1}.
\]
Further, since $\bigl(1+|\varepsilon_R x_0|^{-2}\bigr)^{-q} \ge 2^{-q}$ for $x_0 \le -1/\varepsilon_R$, we find that
\begin{equation} \label{eq:rarefaction-(6)-3}
    C_* \tilde{w}\, (\varepsilon_R |x_0|)^{-2q+1}
    \le w_0(x_0) - w_m \le
    C^* \tilde{w}\, (\varepsilon_R |x_0|)^{-2q+1},
    \qquad \forall\, x_0 \le -1/\varepsilon_R,
\end{equation}
where $C_* := K_q 2^{-q}/(2(2q-1))>0$ and $C^* := K_q/(2(2q-1))>0$.

To use \eqref{eq:rarefaction-(6)-3}, we need to show that $|x_0(t, w_m t)| \ge 1/\varepsilon_R$ for all $t \ge T_0$.
Thanks to the monotonicity of \(w_0\), we evaluate \eqref{eq:rarefaction-(6)-3} at the endpoint $x_0 = -1/\varepsilon_R$ to find that
\[
w_0(x_0) - w_m
\ge w_0(-1/\varepsilon_R) - w_m
\ge C_* \tilde{w}, 
\qquad \forall\, x_0 \ge -1/\varepsilon_R.
\]
This combined with \eqref{eq:characteristic-relation} implies that $|x_0(t, w_m t)| \ge C_* \tilde{w} t$  whenever $|x_0(t, w_m t)| \le 1/\varepsilon_R$.
Then, since \(\a>1\), we choose $\varepsilon_R>0$ small enough that $T_0 = \varepsilon_R^{-\alpha} \ge 1/(C_* \tilde{w}\, \varepsilon_R)$, and thus we conclude that $|x_0(t, w_m t)| \ge 1/\varepsilon_R$ for all $t \ge T_0$.

We are now ready to apply \eqref{eq:rarefaction-(6)-3}.
Substituting \eqref{eq:rarefaction-(6)-3} into \eqref{eq:characteristic-relation}, we obtain
\begin{equation} \label{x0nh}
|x_0(t, w_m t)|^{2q} \le C \tilde{w} \varepsilon_R^{-2q+1} t, \qquad \forall t \ge T_0.
\end{equation}
On the other hand, by differentiating \eqref{eq: characteristic} in $x$, we observe 
\begin{equation} \label{eq:dx0dx}
\frac{\partial x_0}{\partial x} = \frac{1}{1+w_0'(x_0)\, t} > 0.
\end{equation}
Thus, $x_0(t, \cdot)$ is increasing and $x_0(t, x) \le x_0(t, w_m t) < 0$ for $x \le w_m t$.
Then, for any \(t\ge T_0\), we get
\begin{equation} \label{eq: w_decay_time}
    |w(t,x) - w_m|
    \le w_0(x_0(t, w_m t)) - w_m
    \;=\; |x_0(t, w_m t)|/t
    \le C \tilde{w}^{\frac{1}{2q}}\, \varepsilon_R^{-1+\frac{1}{2q}}\, t^{-1+\frac{1}{2q}}.
\end{equation}
Here, \eqref{x0nh} together with the monotonicity of \(w_0\) was used.
We now use the decay property in (3) to interpolate \eqref{eq: w_decay_time} as follows: for any $\theta \in (0,1)$,
\begin{align*}
|w(t,x) - w_m|
&= |w(t,x) - w_m|^\theta\, |w(t,x) - w_m|^{1-\theta} \\
&\le C_{q,\theta}\, \tilde{w}\, (\tilde{w}\varepsilon_R t)^{-\theta(1-\frac{1}{2q})}
\bigl(1+\varepsilon_R|x-w_m t|\bigr)^{-q(1-\theta)}
\end{align*}
which establishes the desired bound.

The estimate for $x \ge w_+ t$ follows by the same argument; the analogue of \eqref{eq: w_decay_time} reads
\begin{equation}\label{eq: w_decay_time_plus}
    |w(t,x) - w_+| \le |x_0(t, w_+ t)|/t
    \le C \tilde{w}^{\frac{1}{2q}}\, \varepsilon_R^{-1+\frac{1}{2q}}\, t^{-1+\frac{1}{2q}}, 
    \qquad \forall\, x \ge w_+ t.
\end{equation}

\noindent\textbf{Proof of} (7)\textbf{:}
On the rarefaction fan region, i.e., $w_m t \le x \le w_+ t$, the inviscid self-similar profile is $w^r(x/t) = x/t$.
Then, from \eqref{eq: characteristic}, we have 
\[
|x/t - w(t,x) | = |x_0(t,x)|/t.
\]
Moreover, from the monotonicity of \(x_0\) in \eqref{eq:dx0dx}, we also have 
\[
x_0(t, w_m t) \le x_0(t,x) \le x_0(t, w_+ t).
\]
Now, using \eqref{eq: w_decay_time} and \eqref{eq: w_decay_time_plus}, we obtain 
\[
\max(|x_0(t, w_m t)|, |x_0(t, w_+ t)|)
\le C \tilde{w}^{\frac{1}{2q}} \varepsilon_R^{-1+\frac{1}{2q}} t^{\frac{1}{2q}}.
\]
Thus, the following holds: 
\[
|w(t,x) - w^r(x/t)|
\le \max\left\{ \frac{|x_0(t, w_m t)|}{t}, \frac{|x_0(t, w_+ t)|}{t} \right\}
\le C_q \tilde{w}^{\frac{1}{2q}} \varepsilon_R^{-1+\frac{1}{2q}} t^{-1+\frac{1}{2q}},
\]
which completes the proof. \qed

\subsection{Proof of Lemma \ref{lem:approx_rarefaction}}
Recall from \eqref{eq:uR-def} that $u^R = \Lambda(w)$ and that $u^r(x/t) = \Lambda(w^r(x/t))$. By Lemma \ref{lem:Burgers}(1), $w$ takes values in $[w_m, w_+] \subset (0,\infty)$, on which $\Lambda$ is smooth and strictly increasing with $\Lambda(w_m) = u_m$ and $\Lambda(w_+) = u_+$. Moreover,
\[
6 u_m \delta_R \le \tilde{w} = 3(u_+^2 - u_m^2) = 3(u_+ + u_m)(u_+ - u_m) \le C\delta_R.
\]

\noindent\textbf{Proof of} (1)\textbf{:} By Lemma \ref{lem:Burgers}(1), we have $w_m < w(t,x) < w_+$ for all $t > 0$ and $x \in \mathbb{R}$. Since $\Lambda$ is strictly increasing with $\Lambda(w_m) = u_m$ and $\Lambda(w_+) = u_+$, we obtain $u_m < u^R(t,x) < u_+$. Differentiating with respect to $x$ yields
\[ \ur_x = \Lambda'(w) w_x = \frac{w_x}{6u^R} > 0, \]
since $w_x > 0$ and $u^R > 0$.

\medskip
\noindent\textbf{Proof of} (2)\textbf{:} Applying the chain rule to $u^R = \Lambda(w)$ yields
\begin{align*}
	%\ur_x  &= \Lambda'(w)w_x, \\
	u^R_{xx} &= \Lambda''(w)(w_x)^2 + \Lambda'(w)w_{xx}, \\
	\partial_x^3 u^R &= \Lambda'''(w)(w_x)^3 + 3\Lambda''(w)w_x w_{xx} + \Lambda'(w)w_{xxx}, \\
  \partial_x^4 u^R &= \Lambda''''(w)(w_x)^4 + 6\Lambda'''(w)(w_x)^2 w_{xx} + 3\Lambda''(w)(w_{xx})^2 + 4\Lambda''(w)w_x w_{xxx} + \Lambda'(w)w_{xxxx}.
\end{align*}
Since $\Lambda'$ is uniformly bounded on $[w_m, w_+]$, we have
\[ \|\ur_x\|_{L^p} \le C \|w_x\|_{L^p} \le C_{p,q} \min\left\{ \delta_R \varepsilon_R^{1-1/p},\ \frac{\delta_R^{1/p}}{t^{1-1/p}} \right\}, \]
where we used $\tilde{w} \le C\delta_R$. For $j = 2,3,4$, the highest-order term $\Lambda'(w)\partial_x^j w$ is estimated in the same way as Lemma \ref{lem:Burgers}(2), using in addition $\varepsilon_R < 1$ to estimate $\varepsilon_R^{(j-1)-1/p+1/q} \le 1$. For the remaining terms, the pointwise bound \eqref{eq:wxx-pointwise} gives
\[
	|w_x w_{xx}| \le C_q \varepsilon_R (w_x)^2, \qquad |(w_{xx})^2| + |w_x w_{xxx}| \le C_q \varepsilon_R^2 (w_x)^2, \qquad |(w_x)^2 w_{xx}| \le C_q \varepsilon_R (w_x)^3,
\]
so that each of them, as well as the purely nonlinear term $(w_x)^j$, is bounded by $C_q \varepsilon_R^{ j-r}(w_x)^r$ for some integer $2 \le r \le j$. Applying Lemma \ref{lem:Burgers}(2) with exponent $rp$, we obtain
\[
	\bigl\|\varepsilon_R^{ j-r}(w_x)^r\bigr\|_{L^p} = \varepsilon_R^{ j-r} \|w_x\|_{L^{rp}}^{r} \le C_{p,q}\min\left\{ \delta_R^{ r} \varepsilon_R^{ j-1/p},\ \frac{\delta_R^{1/p}}{t^{r-1/p}}\right\}.
\]
Since $p \ge 1$, we have $r - 1/p \ge 2 - 1/p \ge 1 > 1 - (j-1)/q$ for $j = 2,3,4$, so the time-decay rate of these terms is strictly faster than that of $\partial_x^j w$ for $t \ge 1$, while for $0 < t \le 1$ they are dominated by the time-independent bound in the minimum. Thus, all nonlinear terms are dominated by the highest-order term $\partial_x^j w$, yielding
\begin{equation*}
	\|\partial_x^j u^R\|_{L^p} \le C_{p,q}\min\left\{\delta_R \varepsilon_R^{j-1/p},\ \frac{\delta_R^{1/p} + \delta_R^{(j-1)/q}}{t^{1-(j-1)/q}}\right\}, \quad j = 2,3,4.
\end{equation*}
For the pointwise bound on $u^R_{xx}$, using $|w_{xx}| \le C|w_x|$, which follows from \eqref{eq:wxx-pointwise} and $\varepsilon_R<1$, we get
\[ |u^R_{xx}| \le |\Lambda''(w)| \cdot |w_x|^2 + |\Lambda'(w)| \cdot |w_{xx}| \le C(|w_x|^2 + |w_x|) \le C(|\ur_x|^2 + |\ur_x|). \]

\noindent\textbf{Proof of} (3)--(4)\textbf{:} %These follow from the Mean Value Theorem applied to $\Lambda$.
For the left tail ($x \le w_m t$), we apply the Mean Value Theorem to find that
\begin{align*}
	|u^R(t,x) - u_m| = |\Lambda(w(t,x)) - \Lambda(w_m)| 
	&\le \sup_{\xi \in [w_m, w_+]}|\Lambda'(\xi)| \cdot |w(t,x) - w_m| \\
	&\le \frac{C\delta_R}{\bigl(1+(\varepsilon_R|x - w_m t|)^2\bigr)^{q/2}},
\end{align*}
where we used Lemma \ref{lem:Burgers}(3) and $\tilde{w} \le C\delta_R$.
It is also obtained from Lemma \ref{lem:Burgers}(3) that
%The derivative bound is obtained {\color{blue}from the second estimate of Lemma \ref{lem:Burgers}(3)} by
\[ |\ur_x(t,x)| = |\Lambda'(w)| \cdot |w_x(t,x)| \le \frac{C\delta_R\varepsilon_R}{\bigl(1+(\varepsilon_R|x-w_m t|)^2\bigr)^q}. \]
The right tail (4) is proved in the same manner using Lemma \ref{lem:Burgers}(4) and $w_+$ in place of $w_m$.

\vspace{2mm}
\noindent\textbf{Proof of} (5)\textbf{:} The asymptotic uniform convergence is an immediate consequence of Lemma \ref{lem:Burgers}(5) and the Lipschitz continuity of $\Lambda$ on $[w_m, w_+]$:
\[ \lim_{t\to\infty}\sup_{x\in\mathbb{R}}|u^R(t,x) - u^r(x/t)| \le C \lim_{t\to\infty}\sup_{x\in\mathbb{R}}|w(t,x) - w^r(x/t)| = 0. \]

\noindent\textbf{Proof of} (6)--(7)\textbf{:} By the Mean Value Theorem and 
$\|\Lambda'\|_{L^\infty([w_m, w_+])} \le C$,
\begin{align*}
  |u^R(t,x) - u_m| &\le C |w(t,x) - w_m|, \\
  |u^R(t,x) - u_+| &\le C |w(t,x) - w_+|, \\
  |u^R(t,x) - u^r(x/t)| &\le C |w(t,x) - w^r(x/t)|.
\end{align*}
The estimates then follow from Lemma \ref{lem:Burgers}(6)--(7).
Here we used $\tilde{w}\le C\delta_R$ for the prefactors and $6u_m\delta_R\le\tilde{w}$ for the factor $(\tilde{w}\varepsilon_R t)^{-\theta(1-\frac{1}{2q})}$ in (6), whose exponent is negative.
Notice that (6) holds with the same constant $\delta_1$ as in Lemma \ref{lem:Burgers}(6).

\vspace{2mm}
\noindent\textbf{Proof of} (8)\textbf{:} Let $u^R=u^R(T_0,\cdot)$ and $u^r=u^r(\cdot/T_0)$, and set $I_1:=(-\infty,w_mT_0]$, $I_2:=[w_mT_0,w_+T_0]$, and $I_3:=[w_+T_0,\infty)$. Since $u^r$ is Lipschitz continuous, $u^R-u^r$ is absolutely continuous and
\[
    \|u^R-u^r\|_{H^1(\mathbb{R})}^2=\sum_{i=1}^{3}\Bigl(\|u^R-u^r\|_{L^2(I_i)}^2+\|u^R_x-u^r_x\|_{L^2(I_i)}^2\Bigr).
\]
On $I_1$, $u^r=u_m$ holds.
Applying (6) with $t=T_0$ and $\theta=\frac{q}{2q-1}$ (i.e., $\theta(1-\frac{1}{2q})=\frac12$) and (3), we have
\[
|u^R-u_m|\le C_q\delta_R(\delta_R\varepsilon_RT_0)^{-\frac12}\bigl(1+\varepsilon_R|x-w_mT_0|\bigr)^{-\frac{q(q-1)}{2q-1}},\qquad
|u^R_x|\le C_q\delta_R\varepsilon_R\bigl(1+(\varepsilon_R|x-w_mT_0|)^2\bigr)^{-q}.
\]
Since $\frac{2q(q-1)}{2q-1}>1$, integrating over $I_1$ and using $T_0=\varepsilon_R^{-8}$, we have
\begin{equation}\label{eq:cmp-I1}
    \|u^R-u_m\|_{L^2(I_1)}\le C_q\delta_R^{\frac12}\varepsilon_R^{-1}T_0^{-\frac12}=C_q\delta_R^{\frac12}\varepsilon_R^{3},\qquad
    \|u^R_x\|_{L^2(I_1)}\le C_q\delta_R\varepsilon_R^{\frac12}.
\end{equation}
On $I_3$ we have $u^r=u_+$, and (6) and (4) give in the same way that
\begin{equation}\label{eq:cmp-I3}
\|u^R-u_+\|_{L^2(I_3)}\le C_q\delta_R^{\frac12}\varepsilon_R^{3},\qquad
\|u^R_x\|_{L^2(I_3)}\le C_q\delta_R\varepsilon_R^{\frac12}.
\end{equation}
On $I_2$, (7) gives
\[
\sup_{I_2}|u^R-u^r|\le C_q\delta_R^{\frac{1}{2q}}(\varepsilon_RT_0)^{-1+\frac{1}{2q}}=C_q\delta_R^{\frac{1}{2q}}\varepsilon_R^{7(1-\frac{1}{2q})},
\]
and $|I_2|=\tilde{w}T_0\le C\delta_RT_0$.
Thus, we obtain
\begin{equation}\label{eq:cmp-I2}
    \|u^R-u^r\|_{L^2(I_2)}\le C_q\delta_R^{\frac12+\frac{1}{2q}}\,\varepsilon_R^{7(1-\frac{1}{2q})}\,T_0^{\frac12}=C_q\delta_R^{\frac12+\frac{1}{2q}}\,\varepsilon_R^{3-\frac{7}{2q}}.
\end{equation}
For the derivative, $u^r=\Lambda(w^r)$ with $w^r(x/T_0)=x/T_0$ and $\Lambda'=\frac{1}{6\Lambda}$, so that
\[
    u^R_x-u^r_x=\frac{w_x-\frac{1}{T_0}}{6u^R}+\frac{1}{6T_0}\Bigl(\frac{1}{u^R}-\frac{1}{u^r}\Bigr).
\]
By the representation $w_x=\frac{w_0'(x_0)}{1+w_0'(x_0)T_0}$ in the proof of Lemma~\ref{lem:Burgers}, it holds that
\[
    w_x-\frac{1}{T_0}=-\frac{1}{T_0\,(1+w_0'(x_0)T_0)},\qquad \Bigl|w_x-\frac{1}{T_0}\Bigr|\le\frac{1}{T_0},
\]
and since $u^R,u^r\ge u_m$ on $I_2$, we have
\begin{equation}\label{eq:cmp-I2x}
    \|u^R_x-u^r_x\|_{L^2(I_2)}\le \frac{C}{u_mT_0}\Bigl(1+\frac{\sup_{I_2}|u^R-u^r|}{u_m}\Bigr)|I_2|^{\frac12}\le C_q\delta_R^{\frac12}T_0^{-\frac12}=C_q\delta_R^{\frac12}\varepsilon_R^{4}.
\end{equation}
Since $\varepsilon_R<1$ and $\delta_R^{\frac12+\frac{1}{2q}}\le\delta_R^{\frac12}+\delta_R$, \eqref{eq:cmp-I1}--\eqref{eq:cmp-I2x} imply \eqref{eq:inviscid-comparison}.
This completes the proof. \qed

\section{Proof of Theorem \ref{thm:main}}\label{sec:pfmain}
\setcounter{equation}{0}

\subsection{Composite Wave and Perturbation Equation}\label{subsec:superposition}
Here, we define the composite wave that will eventually be shown to be asymptotically stable.
In the moving frame \(\x\coloneqq x-\s t\) with \(\s=3u_m^2\), the mKdVB equation \eqref{eq:gKdVB} is given as follows:
\begin{equation}\label{eq:gKdVB-xi}
    u_t - \sigma u_\xi + f(u)_\xi = u_{\xi\xi} - \kappa u_{\xi\xi\xi}.
\end{equation}
Let \(X(t)\) be a time-dependent shift to be determined in \eqref{eq:shift} and let \(T_0 = \varepsilon_R^{-\alpha}\) be a time-translation parameter with \(\a>1\).
Then, we define the shifted degenerate shock and rarefaction profiles:
\begin{equation}\label{eq:shifted-profiles}
(u^S)^X(t,\xi) \coloneqq u^S(\xi - X(t)), \qquad
(u^R)^{X,T_0}(t,\xi) \coloneqq u^R(t+T_0, \xi - X(t) + \sigma (t+T_0)).
\end{equation}
Note from \eqref{eq:shock-eq} and \eqref{eq:uR-eq} that the shifted shock and rarefaction satisfy the following:
\begin{align}
-\sigma (u^S)^X_\xi + f((u^S)^X)_\xi
&= (u^S)^X_{\xi\xi} - \kappa (u^S)^X_{\xi\xi\xi}, \label{eq:shock-shifted} \\
\partial_t (u^R)^{X,T_0}
&= -f((u^R)^{X,T_0})_\xi + (\sigma - \dot{X}(t))(u^R)^{X,T_0}_\xi. \label{eq:rarefaction-shifted}
\end{align}

We now define the composite wave as a superposition of the shifted shock and rarefaction:
\begin{equation} \label{eq:compositewave-def}
\tilde{u}^X(t,\xi) := (u^S)^X(t,\xi) + (u^R)^{X,T_0}(t,\xi) - u_m.
\end{equation}
We also write
\begin{equation}\label{eq:unshifted-composite}
    \tilde{u}^{0}(t,\xi):=u^S(\xi)+u^R\bigl(t+T_0,\,\xi+\sigma(t+T_0)\bigr)-u_m
\end{equation}
for the composite wave with $X\equiv0$, so that $\tilde{u}^{X}(t,\xi)=\tilde{u}^{0}(t,\xi-X(t))$.

Using \eqref{eq:compositewave-def} with \eqref{eq:shock-shifted} and \eqref{eq:rarefaction-shifted}, a direct computation shows that the composite wave $\tilde{u}^X$ satisfies
\begin{equation} \label{eq:compositewave-xi}
\tilde{u}^X_t
- \sigma \tilde{u}^X_\xi
+ f(\tilde{u}^X)_\xi
- \tilde{u}^X_{\xi\xi}
+ \kappa \tilde{u}^X_{\xi\xi\xi}
= E^X - \dot{X}(t)\bigl( (u^S)^X_\xi + (u^R)^{X,T_0}_\xi \bigr),
\end{equation}
where the error term $E^X$ is given by
\begin{align*}
E^X
&= \big[ f(\tilde{u}^X) - f((u^R)^{X,T_0}) - f((u^S)^X) \big]_\xi - (u^R)^{X,T_0}_{\xi\xi} + \kappa (u^R)^{X,T_0}_{\xi\xi\xi} \\
&=
\underbrace{\big[(f'(\tilde{u}^X) - f'((u^S)^X))(u^S)^X_\xi + (f'(\tilde{u}^X) - f'((u^R)^{X,T_0}))(u^R)^{X,T_0}_\xi\big]}_{\eqqcolon E_I^X}
\underbrace{- (u^R)^{X,T_0}_{\xi\xi} + \kappa (u^R)^{X,T_0}_{\xi\xi\xi}}_{\eqqcolon E_R^X}.
\end{align*}
Here, $E_I^X$ represents the wave interaction error between the shifted shock and rarefaction profiles due to the nonlinearity of the flux \(f\), and $E_R^X$ denotes the error arising from the shifted rarefaction.

\vspace{2mm}
We introduce the perturbation $\phi(t,\xi) \coloneqq u(t,\xi) - \tilde{u}^X(t,\xi)$ around the composite wave \(\util^X(t,\x)\).
Then, \eqref{eq:gKdVB-xi} and \eqref{eq:compositewave-xi} yield the following perturbation equation:
\begin{equation} \label{eq:perturbed_general}
\begin{cases}
\phi_t - \dot{X}(t)\bigl( (u^S)^X_\xi + (u^R)^{X,T_0}_\xi \bigr) - \sigma \phi_\xi + \bigl( f(\phi + \tilde{u}^X) - f(\tilde{u}^X) \bigr)_\xi - \phi_{\xi\xi} + \kappa \phi_{\xi\xi\xi}
= -E^X, \\
\phi(0,\xi) = \phi_0(\xi) := u_0(\xi) - \tilde{u}^{0}(0,\xi).
\end{cases}
\end{equation}

For notational simplicity, we henceforth suppress the superscripts $X$ and $T_0$, writing $\tilde{u}$, $u^S$, and $u^R$; these profiles are always evaluated at $ \xi - X(t)$ for the shock and $\xi - X(t) + \sigma(t+T_0)$ for the rarefaction, unless stated otherwise.
\begin{comment}
{\color{blue}We also abbreviate the initial profile of the composite wave by
\begin{equation}\label{eq:tilde-u-0}
    \tilde{u}_0 := \tilde{u}(0,\cdot) = u^S(\cdot\,;u_-,u_m) + u^R\bigl(T_0,\, \cdot + \sigma T_0\bigr) - u_m,
\end{equation}
which is independent of the shift because $X(0)=0$ in \eqref{eq:shift}, so that $\phi_0 = u_0 - \tilde{u}_0$ in \eqref{eq:perturbed_general}.
}
\end{comment}

\subsection{Weight and Shift Functions}\label{sec:weight-shift}
In the following subsections, we introduce our main proposition, which provides the \textit{a priori} estimate.
We then use this proposition to prove the main stability theorem. To state the proposition, we first need to define the weight and the shift functions.

\vspace{2mm}
We recall the degenerate shock setting with end states $u_m = s$ and $u_- = -2s$ ($s > 0$), for which the shock strength \(\ds=u_m-u_-=3s\).
Following \cite{HWZ-MA}, we introduce a piecewise weight function $w(\us)$ depending only on the shock.
Since integration by parts on the dispersive term $\phi_{\xi\xi\xi}$ gives rise to $w'''(\us)$, we require $C^3$-regularity at the gluing points $\us = 0$ and $\us = s/2$ in order to avoid (internal) boundary terms.
The weight function $w(\us)$ is defined as follows:
\begin{equation} \label{eq:weight}
w(\us) =
\begin{cases}
\frac{5}{2}s(s - \us), & \us \in [-2s, 0), \\[6pt]
P_6(\us), & \us \in [0, s/2), \\
\frac{15}{8}s^2, & \us \in [s/2, s),
\end{cases}
\end{equation}
where $P_6(\us)$ is the unique 6th-order interpolating polynomial
\begin{equation*}
P_6(\us) = \tfrac{5}{2}s(s - \us) + (\us)^4 P_2(\us), \qquad
P_2(\us) = \tfrac{50}{s^2} - \tfrac{120}{s^3}\us + \tfrac{80}{s^4}(\us)^2.
\end{equation*}

The following lemma, whose proof is given in Appendix~\ref{app:weight}, collects the properties of the weight function used below.
\begin{lemma}\label{lem:weight_properties}
Let $\us$ be the degenerate shock given by \eqref{eq:shock-eq}, which connects $u_- = -2s$ and $u_m = s$.
Then, the weight $w(\us)$ in \eqref{eq:weight} satisfies the following:
\begin{itemize}
\item[(1)] $w \in C^3([-2s, s))$, and $w^{(k)} \equiv 0$ for $k \in \{1,2,3\}$ on $[s/2, s)$.
\item[(2)] $w' \le 0$ and $w'' \ge 0$ on $[-2s,s)$; in particular, $\tfrac{15}{8}s^2 \le w \le \tfrac{15}{2}s^2$.
\item[(3)] The \(k\)-th order derivative for each \(k \in \{0,1,2,3\}\) satisfies the following uniform bound:
\[
|w^{(k)}(\us)| \le C_k s^{2-k}, \quad \forall\, \us \in [-2s, s),
\]
where the constants \(C_k\) are given by
\[
C_0=\tfrac{15}{2}, \qquad
C_1=\tfrac{5}{2}, \qquad
C_2=\tfrac{75}{8}, \qquad
C_3=\tfrac{100\sqrt{3}}{3}.
\]
\end{itemize}
\end{lemma}
\noindent In what follows we write $w$, $w'$, $w''$, and $w'''$ for $w(\us)$, $w'(\us)$, $w''(\us)$, and $w'''(\us)$, respectively.

\vspace{2mm}
We define a time-dependent shift function $X(t)$ as a solution to the following ODE:
\begin{equation} \label{eq:shift}
\dot{X}(t) = -\frac{32}{25 s^2} \int_{\mathbb{R}} \phi(t,\xi)\, w(\us)\, \us_\xi \, d\xi, \qquad X(0) = 0.
\end{equation}
%Since $\phi=u-\tilde{u}^{X}$ and $\tilde{u}^{X}(t,\xi)=\tilde{u}^{0}(t,\xi-X(t))$, the right-hand side of \eqref{eq:shift} depends on $X$.
The Cauchy--Lipschitz theorem (see \cite[Lemma A.1]{CKKV-M320}) implies that there exists a unique absolutely continuous solution $X$ on $[0,T]$ such that for all \(t\in[0,T]\)
\begin{equation}\label{eq:shift-crude}
    |\dot{X}(t)|\le C(1+\|u-\tilde{u}^{0}\|_{L^\infty(0,T;H^1(\mathbb{R}))}), \qquad |X(t)|\le C(1+\|u-\tilde{u}^{0}\|_{L^\infty(0,T;H^1(\mathbb{R}))})\,t.
\end{equation}

\subsection{Main Proposition}
We now state two propositions, on local well-posedness and on \textit{a priori} estimates, which together yield Theorem~\ref{thm:main}.
The local existence result is formulated around the composite wave \eqref{eq:unshifted-composite} at an arbitrary initial time, as required for the continuation argument below.

\begin{proposition}[Local Existence] \label{prop:local_existence}
For any \(M>0\), there exists \(T^*=T^*(M)>0\) such that the following holds for every \(t_0 \ge 0\):
If $u_{t_0}-\tilde{u}^{0}(t_0,\cdot)\in H^1(\mathbb{R})$ satisfies $\|u_{t_0}-\tilde{u}^{0}(t_0,\cdot)\|_{H^1(\mathbb{R})}\le M$, then \eqref{eq:gKdVB-xi} with initial value $u(t_0,\cdot)=u_{t_0}$ admits a unique solution $u$ on $[t_0,t_0+T^*]$ satisfying
\[
u-\tilde{u}^{0} \in C([t_0, t_0+T^*]; H^1(\mathbb{R})) \cap L^2(t_0, t_0+T^*; H^2(\mathbb{R})),
\quad
\sup_{t \in [t_0, t_0+T^*]} \|u(t)-\tilde{u}^{0}(t)\|_{H^1(\mathbb{R})} \le 2M.
\]
\end{proposition}

\begin{proof}
This follows from a standard argument; see, for instance, \cite{CEKS1,EEK-BE}. So we omit the details.
\end{proof}
\begin{comment}
\begin{proof}
This follows from a standard argument; see, for instance, \cite{CEKS1,EEK-BE}.
{Indeed, the Fourier multiplier of the semigroup generated by $\partial_\xi^2-\kappa\partial_\xi^3+\sigma\partial_\xi$ is $e^{t(-\eta^{2}+i(\kappa\eta^{3}+\sigma\eta))}$, whose modulus $e^{-t\eta^{2}}$ carries no dispersion, so that the smoothing estimates of the heat semigroup apply verbatim; the cubic nonlinearity is locally Lipschitz on $H^1(\mathbb{R})$ since $H^1(\mathbb{R})$ is an algebra, and the coefficients and the error term $E$ of \eqref{eq:perturbed_general} are bounded uniformly in time by Theorem~\ref{thm:shock_properties} and Lemma~\ref{lem:approx_rarefaction}, whence $T^*$ may be chosen independently of $t_0$. The membership in $L^2(t_0,t_0+T^*;H^2(\mathbb{R}))$ follows not from the smoothing estimate but from the energy identity obtained by multiplying \eqref{eq:perturbed_general} by $-\phi_{\xi\xi}$, in which the dispersive and transport terms integrate to zero as in \eqref{eq:H1-energy-1}. Finally, the shift enters \eqref{eq:perturbed_general} only through $\dot X$, which is linear in $\phi$ by \eqref{eq:shift}, and through translates of the profiles, which depend on $X$ in a Lipschitz manner; the fixed-point argument therefore applies to the pair $(\phi,X)$, and $T^*$ is independent of $t_0$ and $X_0$ because the bounds on the profiles are uniform in time and invariant under translation.}
We thus omit the {details}.
\end{proof}
\end{comment}

\begin{proposition}[\textit{A Priori} Estimate] \label{prop:a-priori_total}
Let the far-field states \(u_\pm\) satisfy \eqref{eq:far-field}, i.e., \(u_-<0\) and \(u_+>-u_-/2\), and set 
\[
u_m \coloneqq -\frac{u_-}{2}, \qquad
\d_S \coloneqq u_m-u_- = 3u_m, \qquad
\d_R \coloneqq u_+-u_m.
\]
Then, there exist positive constants \(\d_0, \e_1, \d_2\) and \(C_0\) such that the following holds.
Let \(0<\varepsilon_R<\delta_2\) and set $T_0:=\varepsilon_R^{-8}$ in the definition \eqref{eq:compositewave-def} of the composite wave.
Suppose that 
\[
\k>0, \qquad \kappa u_m^2 \le \delta_0, \qquad \delta_R \le \frac{5}{18}\delta_S.
\]
If the Cauchy problem \eqref{eq:perturbed_general} admits a solution $\phi \in C([0,T]; H^1(\mathbb{R})) \cap L^2(0,T; H^2(\mathbb{R}))$ for some time $T>0$ satisfying the a priori assumption
\begin{equation*}
\sup_{0 \le t \le T} \|\phi(t,\cdot)\|_{H^1(\mathbb{R})} \le \varepsilon_1,
\end{equation*}
then, for every \(t\in[0,T]\), 
\begin{align*}
&\|\phi(t)\|_{H^1(\RR)}^2
+ \int_0^t \|\phi_\xi(\tau)\|_{H^1(\RR)}^2 \,d\tau
+ \int_0^t\int_\mathbb{R} \phi^2 \bigl( \us_\xi + \ur_\xi\bigr) \, d\xi \, d\tau\\
&\qquad
+ \int_0^t |\dot{X}(\tau)|^2 \,d\tau 
\le C_0 \bigl( \|\phi(0)\|_{H^1(\RR)}^2 + \varepsilon_R^{1/6} \bigr),
\end{align*}
where the constant \(C_0\) is independent of \(\k\), \(\e_R\), and \(T\).
\end{proposition}

\subsection{Proof of Theorem \ref{thm:main}}
In what follows, we combine the two propositions above to establish Theorem \ref{thm:main}.
Let $\delta_0$, $\varepsilon_1$, $\delta_2$, and $C_0$ be the constants of Proposition~\ref{prop:a-priori_total}, where we may assume $C_0\ge1$ and \(\d_2\le1\).
Choose
\begin{equation}\label{eq:eR-choice}
\varepsilon_0 := \frac{\varepsilon_1}{4\sqrt{C_0}}, \qquad
\varepsilon_R := \min\Bigl\{ \frac{\delta_2}{2},\; \Bigl(\frac{\varepsilon_1^2}{4C_0}\Bigr)^{6},\; \Bigl(\frac{\varepsilon_0}{2C_q(\delta_R^{\frac12}+\delta_R)}\Bigr)^{2} \Bigr\}, \qquad
T_0:=\varepsilon_R^{-8},
\end{equation}
where $C_q$ is the constant of Lemma~\ref{lem:approx_rarefaction}(8).
Then 
\begin{equation}\label{eq:smallness-budget}
    0<\varepsilon_R<\delta_2\le1, \qquad C_0\,\varepsilon_R^{1/6}\le\frac{\varepsilon_1^2}{4}, \qquad C_0\,\varepsilon_0^{2}=\frac{\varepsilon_1^2}{16}.
\end{equation}
Since $C_0$, $\varepsilon_1$, and $\delta_2$ depend only on $q$, $\delta_S$, and $\delta_R$, the same is true of $\varepsilon_0$ and $\varepsilon_R$.

\vspace{2mm}
\noindent\textbf{Global Existence:}
Let \(u_0\) satisfy \eqref{eq:initial-smallness}, and define $\phi_0 \coloneqq u_0-\tilde{u}^{0}(0,\cdot)$.
Then, the initial smallness assumption gives 
\begin{equation}\label{eq:phi0-small}
\|\phi_0\|_{H^1(\mathbb{R})}
\le \|u_0-\tilde{u}^{0}(0,\cdot)\|_{L^2(\mathbb{R})}+\|u_{0x}-\partial_\xi\tilde{u}^{0}(0,\cdot)\|_{L^2(\mathbb{R})}
< \varepsilon_0.
\end{equation}
By Proposition~\ref{prop:local_existence} with $t_0=0$ and $M=\varepsilon_0$, there exists a unique solution $u$ of \eqref{eq:gKdVB-xi} on $[0,T^*]$, where $T^*=T^*(\varepsilon_0)$, such that 
\[
u-\tilde{u}^{0}\in C([0,T^*];H^1(\mathbb{R}))\cap L^2(0,T^*;H^2(\mathbb{R})), \qquad
\sup_{[0,T^*]}\|u(t)-\tilde{u}^{0}(t)\|_{H^1(\mathbb{R})}
\le 2\varepsilon_0
\le \frac{\varepsilon_1}{2}.
\]
Let \(X\) be the solution of \eqref{eq:shift} on \([0,T^*]\) and set \(\phi\coloneqq u-\tilde{u}^{X}\). 
Since $\tilde{u}^{X}(t,\cdot)=\tilde{u}^{0}(t,\cdot-X(t))$, Theorem~\ref{thm:shock_properties} and Lemma~\ref{lem:approx_rarefaction}(2) imply
\begin{equation}\label{eq:translation}
    \|\tilde{u}^{X}(t,\cdot)-\tilde{u}^{0}(t,\cdot)\|_{H^2(\mathbb{R})}\le|X(t)|\,\bigl\|\partial_\xi\tilde{u}^{0}(t,\cdot)\bigr\|_{H^2(\mathbb{R})}\le C|X(t)|,
\end{equation}
so that $\phi\in C([0,T^*];H^1(\mathbb{R}))\cap L^2(0,T^*;H^2(\mathbb{R}))$.
Moreover, by \eqref{eq:shift-crude}, we have $\|\phi(t)\|_{H^1(\mathbb{R})}\le\frac{\varepsilon_1}{2}+Ct$ on $[0,T^*]$.
Thus, choosing $T_1\in(0,T^*]$ sufficiently small such that $CT_1\le\frac{\varepsilon_1}{4}$, we obtain 
\[
\sup_{[0,T_1]}\|\phi(t)\|_{H^1(\mathbb{R})}\le\frac34\varepsilon_1.
\]
We then consider the maximal time 
\[
\begin{aligned}
    T_M := \sup\Bigl\{ T>0 \;\Big|\; & u \text{ solves } \eqref{eq:gKdVB-xi} \text{ on }[0,T]\text{ with } u-\tilde{u}^{0}\in C([0,T];H^1)\cap L^2(0,T;H^2)\\
    &\text{and } \sup_{[0,T]}\|\phi(t)\|_{H^1}\le\varepsilon_1 \Bigr\},
\end{aligned}
\]
where \(X\) and \(\phi\) are uniquely determined by \(u\) through \eqref{eq:shift} and the Cauchy--Lipschitz theorem.
The preceding construction shows that $T_M\ge T_1>0$.
Suppose, for a contradiction, that $T_M<\infty$.
For \(0\le t<T_M\), the shift choice \eqref{eq:shift} implies \(|\dot X(t)|\le C\|\phi(t)\|_{L^\infty(\RR)}\le C\varepsilon_1\) and \(|X(t)|\le C\varepsilon_1T_M\).
Consequently, \eqref{eq:translation} yields 
\[
\|u(t)-\tilde{u}^{0}(t)\|_{H^1(\mathbb{R})}
\le \varepsilon_1(1+CT_M) \eqqcolon M_*.
\]
The local existence time \(T^*(M_*)\) is independent of the restart time.
We thus apply Proposition~\ref{prop:local_existence} at a time \(t_0<T_M\) close enough to \(T_M\) to extend \(u\) beyond \(T_M\).
Since $t\mapsto\|\phi(t)\|_{H^1(\mathbb{R})}$ is continuous, the maximality of $T_M$ implies
\begin{equation}\label{eq:saturation}
\sup_{[0,T_M]}\|\phi(t)\|_{H^1(\mathbb{R})}=\varepsilon_1.
\end{equation}
On the other hand, Proposition~\ref{prop:a-priori_total} applies to every interval \([0,T]\) with \(T<T_M\).
Letting \(T\uparrow T_M\) and using \eqref{eq:phi0-small}, we find that
\[
\sup_{[0,T_M]}\|\phi(t)\|_{H^1(\mathbb{R})}^2
\le C_0\bigl(\|\phi_0\|_{H^1(\mathbb{R})}^2+\varepsilon_R^{1/6}\bigr)
\le \frac{\varepsilon_1^2}{16}+\frac{\varepsilon_1^2}{4}
<\varepsilon_1^2,
\]
which contradicts \eqref{eq:saturation}.
Hence, $T_M=\infty$, and this together with Proposition \ref{prop:a-priori_total} implies
\begin{equation}\label{eq:global-bounds}
\sup_{t\ge0}\|\phi(t)\|_{H^1(\RR)}^2
+ \int_0^\infty\Bigl(\|\phi_\xi(\tau)\|_{H^1(\RR)}^2
+|\dot{X}(\tau)|^2
+\int_\mathbb{R}\phi^2(\us_\xi+\ur_\xi)\,d\xi
\Bigr)d\tau
\le 2C_0\bigl(\varepsilon_0^2+\varepsilon_R^{1/6}\bigr)
<\infty.
\end{equation}
The definition of the shift \eqref{eq:shift}, with $\|w\|_{L^\infty(\RR)}=\tfrac{15}{2}s^2$ and $\|\us_\xi\|_{L^1(\RR)}=3s$, yields that
\begin{equation} \label{eq:Xdot-Linfty}
|\dot{X}(t)|
\le \frac{32}{25s^2}\,\|w\|_{L^\infty(\RR)}\|\us_\xi\|_{L^1(\RR)}\,\|\phi(t)\|_{L^\infty(\RR)} \le C\,\|\phi(t)\|_{L^\infty(\RR)}, \qquad t\ge0.
\end{equation}
Since the two regularity assertions of Theorem~\ref{thm:main} are exactly $\phi\in C([0,\infty);H^1(\RR))$ and $\phi_\x\in L^2(0,\infty;H^1(\RR))$, both follow from \eqref{eq:global-bounds} by the change of variables \(\x = x -\s t\).

\begin{comment}
By Lemma~\ref{lem:approx_rarefaction}, we have 
\[
\begin{gathered}
\ur(t+T_0,\cdot-X(t)+\s T_0)-u^r(\cdot/t)
\in C((0,\infty);H^1(\RR)),\\
\ur_{xx}(t+T_0,\cdot) \in L^2(0,\infty;L^2(\RR)).
\end{gathered}
\]
Together with $\phi\in C([0,\infty);H^1(\mathbb{R}))$ and \eqref{eq:global-bounds}, these properties yield both of the regularity assertions in Theorem~\ref{thm:main} by the definition of \(\phi\) and the change of variables \(\x = x -\s t\).
\end{comment}

%In particular, $\phi\in C([0,\infty);H^1(\mathbb{R}))$, which is the first regularity assertion of Theorem~\ref{thm:main}. Since $\phi_{\xi\xi}\in L^2(0,\infty;L^2(\mathbb{R}))$ by \eqref{eq:global-bounds} and $\|\ur_{xx}(t+T_0,\cdot)\|_{L^2(\mathbb{R})}\le C(t+T_0)^{-(1-1/q)}$ by Lemma~\ref{lem:approx_rarefaction}(2), we have
%\[
%u_{xx}(t,x)-\us_{xx}\bigl(x-\sigma t-X(t)\bigr)=\bigl(\phi_{\xi\xi}+\ur_{\xi\xi}\bigr)\bigl(t,\,x-\sigma t\bigr)\in L^2(0,\infty;L^2(\mathbb{R})),
%\]
%which is the second regularity assertion.

\vspace{2mm}
\noindent\textbf{Long-Time Behavior.}
Consider $g(t):=\|\phi_\xi(t)\|_{L^2(\mathbb{R})}^2$. We claim that
\begin{equation}\label{eq:g-W11}
    \int_0^\infty\bigl(|g(t)|+|g'(t)|\bigr)\,dt<\infty.
\end{equation}
The integrability of \(g\) follows from \eqref{eq:global-bounds}.
To estimate \(g'\), we use \eqref{eq:perturbed_general} and integrate by parts as in \eqref{eq:H1-energy-1}.
Thus, it holds that
\[
g'(t)
= -2\|\phi_{\xi\xi}\|_{L^2(\RR)}^2
-2\int_\mathbb{R}\phi_{\xi\xi}\Bigl[\dot{X}\bigl(\us_\xi+\ur_\xi\bigr)-\bigl(f(\phi+\tilde{u})-f(\tilde{u})\bigr)_\xi - E\Bigr]d\xi,
\]
where the terms with $\sigma$ and $\kappa$ vanish, since $\int_\mathbb{R}\phi_{\xi\xi}\phi_\xi\,d\xi=\int_\mathbb{R}\phi_{\xi\xi}\phi_{\xi\xi\xi}\,d\xi=0$.
The a priori bounds 
\[
\|\phi\|_{L^\infty(\RR)}\le\sqrt2\varepsilon_1 \qquad \text{and} \qquad
\|\us_\xi+\ur_\xi\|_{L^\infty(\RR)}+\|\us_\xi+\ur_\xi\|_{L^2(\RR)}\le C
\]
imply  
\[
|(f(\phi+\tilde{u})-f(\tilde{u}))_\xi|\le C\bigl(|\phi_\xi|+|\phi|(\us_\xi+\ur_\xi)\bigr),
\]
which, together with Young's inequality, yields 
\[
|g'(t)|
\le C\Bigl(\|\phi_\xi(t)\|_{H^1(\RR)}^2
+ |\dot{X}(t)|^2
+ \int_\mathbb{R}\phi^2(\us_\xi+\ur_\xi)\,d\xi + \|E(t)\|_{L^2(\RR)}^2\Bigr).
\]
The right-hand side is integrable on \((0,\infty)\) by \eqref{eq:global-bounds} and by the estimate of $\int_0^\infty\int_\mathbb{R}|E|^2\,d\xi\,d\tau$ in the proof of Proposition~\ref{prop:a-priori-H1}.
This proves \eqref{eq:g-W11}, and hence $\lim_{t\to\infty}\|\phi_\xi(t)\|_{L^2(\mathbb{R})}=0$. This together with the interpolation inequality $\|\phi\|_{L^\infty(\RR)}^2\le2\|\phi\|_{L^2(\RR)}\|\phi_\xi\|_{L^2(\RR)}$ and \eqref{eq:global-bounds} implies
\begin{equation}\label{eq:phi-Linfty-decay}
    \lim_{t\to\infty}\|\phi(t)\|_{L^\infty(\mathbb{R})}=0,
\end{equation}
and then \eqref{eq:Xdot-Linfty} yields $\lim_{t\to\infty}|\dot{X}(t)|=0$.
This establishes the last assertion of Theorem~\ref{thm:main}.

\vspace{2mm}
It remains to prove the sup-convergence. For $t>0$ and $x\in\mathbb{R}$, we have
\begin{multline*}
\Bigl|u(t,x)-\Bigl(\us\bigl(x-\sigma t-X(t)\bigr)+u^r\Bigl(\frac{x}{t}\Bigr)-u_m\Bigr)\Bigr|
\le
\|\phi(t)\|_{L^\infty(\RR)}\\
+\Bigl|\ur\bigl(t+T_0,\,x-X(t)+\sigma T_0\bigr)-u^r\Bigl(\frac{x-X(t)+\sigma T_0}{t+T_0}\Bigr)\Bigr|
+\Bigl|u^r\Bigl(\frac{x-X(t)+\sigma T_0}{t+T_0}\Bigr)-u^r\Bigl(\frac{x}{t}\Bigr)\Bigr|.
\end{multline*}
The first two terms tend to zero uniformly in $x$, by \eqref{eq:phi-Linfty-decay} and by Lemma~\ref{lem:approx_rarefaction}(5). 
To estimate the last term, recall that \(u^r\) is globally Lipschitz with Lipschitz constant $(6u_m)^{-1}$ and is constant outside $[w_m,w_+]$.
Fix \(0<\et<1\). Since \(X(t)/t \to 0\), we have $|X(t)|\le\eta t$ for all sufficiently large \(t\).
On the region \((w_m-1)t \le x \le (w_++1)t\),
\[
\Bigl|\frac{x-X(t)+\sigma T_0}{t+T_0}-\frac{x}{t}\Bigr|
\le \frac{|X(t)|+\sigma T_0}{t+T_0}+\frac{|x|\,T_0}{t(t+T_0)}
\le \eta+\frac{(\sigma+w_++1)\,T_0}{t}.
\]
The Lipschitz continuity of \(u^r\) therefore bounds the last term by $(6u_m)^{-1}(\eta+C t^{-1})$.
If $x>(w_++1)t$ and $(1-\eta)t\ge(w_+-\s)T_0$, then both $\frac{x}{t}$ and $\frac{x-X(t)+\sigma T_0}{t+T_0}$ exceed $w_+$, so both values of $u^r$ equal $u_+$ and the last term vanishes; if $x<(w_m-1)t$, then $\frac{x}{t}<w_m$ and $\frac{x-X(t)+\sigma T_0}{t+T_0}\le\frac{x+\eta t+\sigma T_0}{t+T_0}\le w_m$, and hence both values equal $u_m$.
Letting $t\to\infty$ and then $\eta\downarrow 0$, we conclude that
\[
\lim_{t\to\infty}\,\sup_{x\in\mathbb{R}}\,\Bigl|u(t,x)-\Bigl(\us\bigl(x-\sigma t-X(t)\bigr)+u^r\bigl(\frac{x}{t}\bigr)-u_m\Bigr)\Bigr|=0,
\]
which completes the proof of Theorem~\ref{thm:main}. \qed

\section{Proof of Proposition \ref{prop:a-priori_total}: \textit{A Priori} Estimates}\label{sec:apriori}
\subsection{Zeroth-Order Estimates}
This section is devoted to the proof of Proposition \ref{prop:a-priori_total}, which provides the uniform-in-time \textit{a priori} estimates.
We begin with the $L^2$-energy estimates and then proceed to the higher-order estimates.
Throughout this section, we assume that the Cauchy problem \eqref{eq:perturbed_general} has a solution $\phi \in C([0,T] $ $; H^1(\mathbb{R})) \cap L^2 (0,T; H^2(\mathbb{R}))$ for some $T > 0$.
In addition, we fix $\alpha = 8$ and write $T_0 := \varepsilon_R^{-8}$ (any $\alpha>7$ would do; the exponents in Propositions~\ref{prop:a-priori-L2} and \ref{prop:a-priori-H1} correspond to $\alpha=8$). The parameter $\varepsilon_R>0$ will be fixed at the end of the proof of Theorem~\ref{thm:main}, and $\delta_0$ is the constant of Lemma~\ref{lem:GS}. Throughout this section we assume $\kappa u_m^2\le\delta_0$, so that $\us$ is the monotone degenerate shock of Theorem~\ref{thm:shock_properties} and Lemma~\ref{lem:GS} applies.

\begin{proposition} \label{prop:a-priori-L2}
There exist positive constants $\delta_0, \varepsilon_1, \delta_2 > 0$ such that, for any $0<\varepsilon_R<\delta_2$ with $T_0:=\varepsilon_R^{-8}$, if the dispersion coefficient $\kappa > 0$, the shock strength $\ds=|u_m-u_-|=3u_m>0$, and the rarefaction strength $\delta_R=|u_+-u_m| > 0$ satisfy
\begin{equation} \label{eq:small-delta-constraint}
\kappa u_m^2 \le \delta_0 \quad \text{and} \quad
\delta_R \le \frac{5}{18} \delta_S \; \left(= \frac{5}{6}u_m\right),
\end{equation}
and the perturbation $\phi \in C([0,T];H^1(\mathbb{R})) \cap L^2(0,T;H^2(\mathbb{R}))$ solves \eqref{eq:perturbed_general} on $[0,T]$ with
\[
\sup_{t \in [0, T]} \| \phi (t,\cdot) \|_{H^1(\mathbb{R})} \le \varepsilon_1,
\]
then there exists a constant $C > 0$, independent of $\kappa$, $\varepsilon_R$, and $T$, such that for all $t \in [0, T]$,
\begin{equation} \label{eq:a-priori-L2}
\begin{gathered}
\|\phi(t,\cdot)\|_{L^2}^2 +\int_0^t \|\phi_\xi(\tau,\cdot) \|_{L^2}^2 \,d\tau  +\int_0^t \int_\mathbb{R} \phi^2 (u^S_\xi + u^R_\xi ) \,d\xi d\tau  \\
+ \int_0^t  |\dot{X}(\tau)|^2 \, d\tau \le  C \|\phi(0,\cdot)\|_{L^2}^2
+ C \varepsilon_R^{1/6}.
\end{gathered}
\end{equation}
\end{proposition}

We prove Proposition \ref{prop:a-priori-L2} through a sequence of auxiliary lemmas.
More precisely, Lemmas \ref{lem:weighted-L2-energy}--\ref{lem:R(t)} establish the estimates needed to show Lemma \ref{lem:a-priori-L^2}.
Proposition \ref{prop:a-priori-L2} then follows as a direct consequence of Lemma \ref{lem:a-priori-L^2}.

\begin{lemma}\label{lem:weighted-L2-energy}
Let \(w(\us)\) and \(X(t)\) denote the weight function and shift function defined in \eqref{eq:weight} and \eqref{eq:shift}, respectively, and let $\phi$ be the solution to \eqref{eq:perturbed_general} on $[0,T]$.
Then, the following holds:
\begin{equation} \label{eq:L2-energy-identity}
\frac{1}{2}\frac{d}{dt}\int_{\mathbb{R}} \phi^{2} w \,d\xi
= \dot{X}(t)\,Y(t) - \mathcal{J}^{good}(t) + \mathcal{J}^{bad}(t),
\end{equation}
where 
\begin{align*}
Y(t)
&\coloneqq \int_{\mathbb{R}} \phi w \us_{\xi} \, d\xi
+ \int_{\mathbb{R}} \phi w \ur_{\xi} \, d\xi
- \frac{1}{2}\int_{\mathbb{R}} \phi^{2} w' \us_{\xi} \, d\xi, \\
\mathcal{J}^{good}(t)
&\coloneqq \int_{\mathbb{R}} \phi_{\xi}^{2} w \, d\xi
+ 3\int_{\mathbb{R}} \phi^{2} \ur w \ur_{\xi} \, d\xi
+ 3\int_{\mathbb{R}} \phi^{2} (\ur - u_m) w \us_{\xi} \, d\xi \\
&\qquad
- 3\int_{\mathbb{R}} \phi^{2} (\us)^{2} w' \us_{\xi} \, d\xi
- \frac{3}{2}\int_{\mathbb{R}} \phi^{2} (\ur - u_m)^2 w' \us_{\xi} \, d\xi
- \frac{3}{4}\int_{\mathbb{R}} \phi^{4} w' \us_{\xi} \, d\xi,\\
\mathcal{J}^{bad}(t)
&\coloneqq
- 3\int_{\mathbb{R}} \phi^{2} \us w \us_{\xi} \,d\xi
- 3\int_{\mathbb{R}} \phi^{2} (\us - u_m) w \ur_{\xi} \, d\xi \\
&\qquad 
- 3u_m^2 \int_{\mathbb{R}} \phi^{2} w' \us_{\xi} \,d\xi
+ 3\int_{\mathbb{R}} \phi^{2} \us (\ur - u_m) w' \us_{\xi} \, d\xi\\
&\qquad
+ \frac{1}{2}\int_{\mathbb{R}} \phi^{2} w'' (\us_{\xi})^{2} \, d\xi
- \int_{\mathbb{R}} \phi^{3} w (\us_\xi + \ur_\xi) \,d\xi
+ 2\int_{\mathbb{R}} \phi^{3} (\us + \ur - u_m) w' \us_{\xi} \, d\xi \\
&\qquad
- \frac{3}{2}\kappa \int_{\mathbb{R}} \phi_\xi^2 w' \us_\xi \, d\xi
+ \frac{\kappa}{2} \int_{\mathbb{R}} \phi^2 (w' \us_{\xi\xi\xi} + w_{\xi \xi \xi}) \, d\xi
- \int_{\mathbb{R}} \phi E w \, d\xi.
\end{align*}
\end{lemma}

\noindent Every term in $\mathcal{J}^{good}(t)$ is non-negative, since $w \ge 0$, $w' \le 0$, $w'' \ge 0$, $\us_\xi > 0$, $\ur_\xi > 0$, and $\ur > u_m$.

\begin{proof}
Multiplying the perturbation equation \eqref{eq:perturbed_general} by $w\phi$ and integrating over $\mathbb{R}$ yields
\begin{equation} \label{eq: energy-1}
\begin{aligned}
&\frac{1}{2}\frac{d}{dt}\int_{\mathbb R}\phi^2 w \,d\xi 
= \dot{X}(t)\left[\, - \frac{1}{2}\int_{\mathbb R}\phi^2 w' \us_\xi\,d\xi\right] + \int_\mathbb{R} \phi \phi_t w \, d\xi \\
&= \dot{X}(t)\left[\, - \frac{1}{2}\int_{\mathbb R}\phi^2 w' \us_\xi\,d\xi\right] \\
&\qquad
+ \int_\mathbb{R} \phi w \left[\, \dot{X} (u^S_\xi + u^R_\xi) + \sigma \phi_\xi - (f(\phi+\tilde{u}) - f(\tilde{u}))_\xi + \phi_{\xi \xi} - \kappa \phi_{\xi \xi \xi} -E \right]  d\xi \\
&= \dot{X}(t)\left[\,\int_{\mathbb R}\phi w (\us_\xi + \ur_\xi)\,d\xi - \frac{1}{2}\int_{\mathbb R}\phi^2 w' \us_\xi\,d\xi\right] - \frac{\sigma}{2} \int_\mathbb{R} \phi^2 w' u^S_\xi \, d \xi\\
&\qquad
+ \int_\mathbb{R} (f(\phi+\tilde{u}) - f(\tilde{u})) \phi_\xi w \, d\xi + \int_\mathbb{R} (f(\phi+\tilde{u}) - f(\tilde{u})) \phi w' u^S_\xi \, d\xi  \\
&\qquad
- \int_{\mathbb R}\phi_\xi^2 w\,d\xi + \frac{1}{2}\int_{\mathbb R}\phi^2 w_{\xi\xi}\,d\xi
- \frac{3}{2}\kappa\int_{\mathbb R}\phi_\xi^2 w' \us_\xi\,d\xi + \frac{\kappa}{2}\int_{\mathbb R}\phi^2 w_{\xi\xi\xi}\,d\xi - \int_{\mathbb R}\phi w E\,d\xi.
\end{aligned}
\end{equation}
Since it holds that \((f(\phi+\tilde{u}) - f(\tilde{u})) = (\phi + \tilde{u})^3-\tilde{u}^3= \phi^3 + 3\tilde{u}\phi^2 + 3\tilde{u}^2 \phi\), we obtain
\begin{multline*}
\int_\mathbb{R} (f(\phi+\tilde{u}) - f(\tilde{u})) \phi_\xi w \, d\xi 
= -\frac{1}{4}\int_{\mathbb R}\phi^4 w' \us_\xi\,d\xi
- \int_{\mathbb R}\phi^3 w \tilde{u}_\xi \,d\xi \\
-\int_\mathbb{R} \phi^3 w' \tilde{u} \us_\xi \, d\xi
- 3\int_{\mathbb R}\phi^2 w \tilde u \tilde{u}_\xi d\xi
-\frac{3}{2} \int_\mathbb{R} \phi^2 w' \tilde{u}^2 \us_\xi \, d\xi
\end{multline*}
and 
\[
\int_\mathbb{R} (f(\phi+\tilde{u}) - f(\tilde{u})) \phi w' u^S_\xi \, d\xi
= \int_\mathbb{R} \phi^4 w' \us_\xi \, d \xi
+ 3 \int_\mathbb{R} \phi^3 w' \tilde{u} \us_\xi \, d \xi 
+3\int_\mathbb{R} \phi^2 w' \tilde{u}^2 \us_\xi \, d\xi.
\]
Then, substituting these relations into \eqref{eq: energy-1}, we have
\begin{equation} \label{eq: energy-2}
\begin{aligned}
&\frac{1}{2}\frac{d}{dt}\int_{\mathbb R}\phi^2 w \,d\xi 
= \dot{X}(t)\left[\,\int_{\mathbb R}\phi w (\us_\xi + \ur_\xi)\,d\xi
- \frac{1}{2}\int_{\mathbb R}\phi^2 w' \us_\xi\,d\xi\right]\\
&\qquad
+ \frac{3}{4}\int_{\mathbb R}\phi^4 w' \us_\xi\,d\xi
+ 2\int_{\mathbb R}\phi^3 w'\tilde u \us_\xi\,d\xi
- \int_{\mathbb R}\phi^3 w(\us_\xi + \ur_\xi)\,d\xi \\
&\qquad
- 3\int_{\mathbb R}\phi^2 w\tilde u\tilde u_\xi\,d\xi 
+ \frac{3}{2}\int_{\mathbb R}\phi^2 w'(\tilde u^2 - u_m^2)\us_\xi\,d\xi \\
&\qquad
- \int_{\mathbb R}\phi_\xi^2 w\,d\xi
+ \frac{1}{2}\int_{\mathbb R}\phi^2 w_{\xi\xi}\,d\xi 
- \frac{3}{2}\kappa\int_{\mathbb R}\phi_\xi^2 w' \us_\xi\,d\xi
+ \frac{\kappa}{2}\int_{\mathbb R}\phi^2 w_{\xi\xi\xi}\,d\xi
- \int_{\mathbb R}\phi w E\,d\xi.
\end{aligned}
\end{equation}
We apply the chain rule together with \eqref{eq:shock-eq} so as to find that 
\[
w_{\xi\xi} = w''(\us_\xi)^2 + 3w'((\us)^2 - u_m^2)\us_\xi + \kappa w' \us_{\xi\xi\xi}.
\]
Plugging this into \eqref{eq: energy-2}, the quadratic terms are reorganized as follows:
\begin{multline*}
\int_{\mathbb R}\phi^2\left[-3w\tilde u\tilde u_\xi + \tfrac{3}{2}w'(\tilde u^2 - u_m^2)\us_\xi + \tfrac{1}{2}w_{\xi\xi} + \tfrac{\kappa}{2}w_{\xi\xi\xi}\right]d\xi \\
= \int_{\mathbb R}\phi^2 \Big[-3w\tilde u\tilde u_\xi + \tfrac{3}{2}w'(\tilde u^2 - u_m^2)\us_\xi + \tfrac{3}{2}w'((\us)^2 - u_m^2)\us_\xi
+ \tfrac{1}{2}w''(\us_\xi)^2 + \tfrac{\kappa}{2}(w' \us_{\xi\xi\xi} + w_{\xi\xi\xi})\Big]d\xi.
\end{multline*}
This together with the identities
\begin{align*}
\tilde{u}^2 - u_m^2
&= ((\us)^2 - u_m^2) + 2\us(\ur - u_m) + (\ur - u_m)^2, \\
\tilde{u}\tilde{u}_\xi
&= \us \us_\xi + (\us - u_m)\ur_\xi + (\ur - u_m)\us_\xi + \ur \ur_\xi,
\end{align*}
and \eqref{eq: energy-2} implies the desired conclusion \eqref{eq:L2-energy-identity}.
\begin{comment}
\begin{align*}
\frac{1}{2}\frac{d}{dt}\int_{\mathbb R}\phi^2 w \,d\xi 
&= \dot{X}(t)\left[\,\int_{\mathbb R}\phi w (\us_\xi + \ur_\xi)\,d\xi - \frac{1}{2}\int_{\mathbb R}\phi^2 w' \us_\xi\,d\xi\right]\\
&\qquad \underbrace{+ \frac{3}{4}\int_{\mathbb R}\phi^4 w' \us_\xi\,d\xi}_{\text{Good}} \underbrace{+ 2\int_{\mathbb R}\phi^3 w'\tilde u \us_\xi\,d\xi - \int_{\mathbb R}\phi^3 w(\us_\xi + \ur_\xi)\,d\xi}_{\text{Bad}} \\
&\qquad \underbrace{- \int_{\mathbb R}\phi_\xi^2 w\,d\xi}_{\text{Good}} \underbrace{- \frac{3}{2}\kappa\int_{\mathbb R}\phi_\xi^2 w' \us_\xi\,d\xi  - \int_{\mathbb R}\phi w E\,d\xi}_{\text{Bad}}\\
&\qquad -3\int_{\mathbb R}\phi^2 w \Big[ \underbrace{\us \us_\xi + (\us - u_m)\ur_\xi}_{\text{Bad}} + \underbrace{(\ur - u_m)\us_\xi + \ur \ur_\xi}_{\text{Good}} \Big] d\xi \\
&\qquad +3\int_{\mathbb R}\phi^2 w'\Big[ \underbrace{(\us)^2}_{\text{Good}} \underbrace{- u_m^2 + \us(\ur - u_m)}_{\text{Bad}} + \underbrace{\frac{1}{2} (\ur - u_m)^2}_{\text{Good}} \Big] \us_\xi \, d\xi \\
&\qquad +\int_{\mathbb R}\phi^2 \Big[ \underbrace{\frac{1}{2}w''(\us_\xi)^2 + \frac{\kappa}{2}(w' \us_{\xi\xi\xi} + w_{\xi\xi\xi})}_{\text{Bad}}\Big]d\xi.
\end{align*}
\end{comment}
This completes the proof of Lemma \ref{lem:weighted-L2-energy}.
\end{proof}

\begin{comment}
\begin{equation} \label{eq: energy-3}
\begin{aligned}
    \frac{1}{2}\frac{d}{dt}\int_{\mathbb R}\phi^2 w \,d\xi 
    %&= \dot{X}(t)\left[\,\int_{\mathbb R}\phi w (\us_\xi + \ur_\xi)\,d\xi - \frac{1}{2}\int_{\mathbb R}\phi^2 w' \us_\xi\,d\xi\right]\\
    &{\color{blue}= \dot{X}(t)\left[\,\int_{\mathbb R}\phi w \us_\xi \,d\xi - \frac{1}{2}\int_{\mathbb R}\phi^2 w' \us_\xi\,d\xi\right]}\\
    &\quad \textcolor{blue}{+ \frac{3}{4}\int_{\mathbb R}\phi^4 w' \us_\xi\,d\xi} \textcolor{blue}{+ 2\int_{\mathbb R}\phi^3 w'\tilde u \us_\xi\,d\xi - \int_{\mathbb R}\phi^3 w(\us_\xi + \ur_\xi)\,d\xi }\\
    &\quad \textcolor{blue}{- \int_{\mathbb R}\phi_\xi^2 w\,d\xi}  \textcolor{blue}{- \frac{3}{2}\kappa\int_{\mathbb R}\phi_\xi^2 w' \us_\xi\,d\xi  - \int_{\mathbb R}\phi w E\,d\xi}\\
    &\quad -3\int_{\mathbb R}\phi^2 w \Big[ \textcolor{blue}{\us \us_\xi } \textcolor{blue}{+ (\us - u_m)\ur_\xi } \textcolor{blue}{+(\ur - u_m)\us_\xi} + \textcolor{blue}{\ur \ur_\xi} \Big] d\xi \\
    &\quad +3\int_{\mathbb R}\phi^2 w'\Big[ (\textcolor{blue}{(\us)^2} \textcolor{blue}{- u_m^2}) \textcolor{blue}{+ \us(\ur - u_m) }\textcolor{blue}{+ \frac{1}{2} (\ur - u_m)^2 } \Big] \us_\xi \, d\xi \\
    &\quad +\int_{\mathbb R}\phi^2 \Big[ \textcolor{blue}{\frac{1}{2}w''(\us_\xi)^2 + \frac{\kappa}{2}(w' \us_{\xi\xi\xi} + w_{\xi\xi\xi})}\Big]d\xi. \\
\end{aligned}
\end{equation}
\end{comment}

To derive the zeroth-order energy estimate, we first consider the following decomposition of the right-hand side of \eqref{eq:L2-energy-identity}.
\begin{lemma}\label{lem:L2-decomposition}
The right-hand side of \eqref{eq:L2-energy-identity} can be decomposed as follows:
\begin{equation} \label{eq:L2-decomposition}
\dot{X}(t)\,Y(t)
+ \mathcal{J}^{bad}(t)
- \mathcal{J}^{good}(t)
= - G^S(t) - G^{R}(t) - G^{SR}(t) + N(t) + J(t) \,+\, R(t),
\end{equation}
where
\begin{align*}
G^S(t)
&\coloneqq \int_{\mathbb{R}} \phi_{\xi}^{2} w \, d\xi 
+ \int_{\mathbb{R}} \phi^{2}\us_{\xi} \left( 3u_m^2 w' - 3(\us)^{2}w' + 3\us w - \frac{w''}{2} \us_{\xi} \right) d\xi \\
&\qquad
- \dot{X}(t)\int_{\mathbb{R}} \phi w \us_{\xi} \, d\xi
- \frac{3}{4}\int_{\mathbb{R}} \phi^{4} w' \us_{\xi} \, d\xi
+ \frac{3}{2}\kappa \int_{\mathbb{R}} \phi_\xi^2 w' \us_\xi \, d\xi
- \frac{\kappa}{2} \int_{\mathbb{R}} \phi^2 (w' \us_{\xi\xi\xi} + w_{\xi \xi \xi}) \, d\xi, \\
G^R(t)
&\coloneqq 3\int_{\mathbb{R}} \phi^{2} \ur w \ur_{\xi} \,d \xi, \\
G^{SR}(t)
&\coloneqq 3\int_{\mathbb{R}} \phi^{2} (\ur - u_m) w \us_{\xi}\, d\xi
- \frac{3}{2}\int_{\mathbb{R}} \phi^{2} (\ur - u_m)^2 w' \us_{\xi} \, d\xi,
\end{align*}
and 
\begin{align*}
N(t)
&\coloneqq \dot{X}(t) \int_{\mathbb{R}} \phi w \ur_{\xi} \, d\xi -\frac{1}{2}\dot{X}(t)\int_{\mathbb{R}} \phi^{2} w' \us_{\xi} \, d\xi
%&\coloneqq \dot{X}(t) \int_{\mathbb{R}} \phi w \ur_{\xi} \,d\xi -\frac{1}{2}\dot{X}(t)\int_{\mathbb{R}} \phi^{2} w' \us_{\xi} \,d\xi \\
- \int_{\mathbb{R}} \phi^{3} w (\us_\xi + \ur_\xi) \, d\xi \\
&\qquad
+ 2\int_{\mathbb{R}} \phi^{3} (\us + \ur - u_m) w' \us_{\xi} \, d\xi
+ 3\int_{\mathbb{R}} \phi^{2} \us (\ur - u_m) w' \us_{\xi} \, d\xi, \\
J(t)
&\coloneqq - 3\int_{\mathbb{R}} \phi^{2} (\us - u_m) w \ur_{\xi} \, d\xi,
\hspace{28mm} R(t)
\coloneqq -\int_{\mathbb{R}} \phi E w \, d\xi.
\end{align*}
\end{lemma}

\subsubsection{Estimation of $G^S(t)$}
In what follows, we derive an estimate for \(G^S(t)\), which consists of the terms localized by the shock derivative and the weight derivative.
To this end, we first apply the change of variables $\xi-X(t)\mapsto\xi$ for each fixed \(t\), so that $\us$ and $w$ are evaluated at $\xi$.
We also introduce the following notation: let \(s \coloneqq u_m\) (from which \(u_-=-2s\) follows) and set $u_* \coloneqq s/2$.

We now state the following lemma, which provides a control on \(G^S(t)\).

\begin{lemma}\label{lem:GS}
Let \(\us\) be the degenerate shock profile connecting $u_- = -2s$ and $u_m = s$ ($s>0$) with speed $\sigma = 3s^2$ and strength $\ds = 3s$.
Let \(G^S(t)\) be as in Lemma~\ref{lem:L2-decomposition} and \(X(t)\) be defined in \eqref{eq:shift}.
Then, there exists a constant \(\d_0>0\) such that for any \(\k>0\) and \(s>0\) satisfying \(\k s^2 \le \d_0\), the following holds: for all $t\in[0,T]$,
\begin{equation} \label{eq:GS-est}
\begin{aligned}
G^S(t)
&\ge
\frac{15}{64} s^2 \int_\mathbb{R}\phi_\xi^2 \, d\xi
+ \frac{1}{2} s^3 \int_{\mathbb R}\phi^2 \us_\xi\, d\xi
+ \frac{25}{64}s^2|\dot X(t)|^2+ \frac{3}{4}\int_{\mathbb{R}}\phi^{4}|w^{\prime}|u_{\xi}^{S} \, d\xi.
\end{aligned}
\end{equation}
\end{lemma}

Since the weight function is constant on \([u_*,u_m]\), we split $G^S(t)$ at the point \(\x_*\) which is defined by \(\us(\x_*)=u_*\).
In particular, $w' = w'' = w''' \equiv 0$ for $\xi \in (\xi_*, \infty)$.
Hence we consider the decomposition $G^S(t) = G_1^S(t) + G_2^S(t)$ by adding and subtracting the quadratic term $\frac{8}{25s^2} \bigl( \int_{-\infty}^{\xi_{*}} \phi w(\us) \us_\xi \,d\xi \bigr)^2$:
    \begin{align}
G_1^S(t)
&\coloneqq
\int_{-\infty}^{\xi_{*}} \phi_{\xi}^{2} w \, d\xi
+ \int_{-\infty}^{\xi_{*}} \phi^{2}\us_{\xi} \Big( 3s^2 w' - 3(\us)^{2}w' + 3\us w - \frac{w''}{2} \us_{\xi} \Big) d\xi \notag \\
&\qquad
+ \frac{8}{25s^2} \left( \int_{-\infty}^{\xi_{*}} \phi w \us_{\xi} \,d\xi \right)^{2} + \frac{3}{2}\kappa \int_{-\infty}^{\xi_{*}} \phi_\xi^2 w' \us_\xi \, d\xi
- \frac{\kappa}{2} \int_{-\infty}^{\xi_{*}} \phi^2 (w' \us_{\xi\xi\xi} + w_{\xi \xi \xi}) \, d\xi,
\label{eq:GS1}\\
G_2^S(t)
&\coloneqq
\int_{\xi_{*}}^{\infty} \phi_{\xi}^{2} w \,d\xi + 3\int_{\xi_{*}}^{\infty} \phi^{2} \us w \us_{\xi} \,d\xi 
- \dot{X}(t)\int_{\mathbb{R}} \phi w \us_{\xi} \, d\xi
- \frac{8}{25s^2} \left( \int_{-\infty}^{\xi_{*}} \phi w \us_{\xi} \, d\xi \right)^{2} \notag \\
&\qquad
- \frac{3}{4}\int_{\mathbb{R}} \phi^{4} w' \us_{\xi} \, d\xi.
\label{eq:GS2}
\end{align}

To estimate $G_1^S$ below, we introduce the following quantities determined by the weight function $w$ and the shock profile $\us$. We define
\begin{equation}\label{eq:H1H2-def}
\begin{aligned}
H_1 &\coloneqq
5s \left[ 3s^2 w' - 3(\us)^2 w' + 3\us w - \frac{1}{2} w'' (\us+2s)(s-\us)^2 \right], \\
H_2 &\coloneqq
w \left[ w'' ( s - 2 \us ) (\us +2s) + 4w - w' (4\us+3s)\right],
\end{aligned}
\end{equation}
and
\begin{equation} \label{eq:M1234-def}
\begin{aligned}
M_{11}
&\coloneqq \sup_{\xi < \xi_*} \frac{3 |w'| \us_{\xi}}{2 s^2 w},
&\hspace{-20mm}M_{12}
&\coloneqq \sup_{\xi < \xi_*} \frac{2w}{s^2} \left| \frac{s - 2\us}{5s(s-\us)^2} \right| \frac{|\us_{\xi\xi}|}{\us_\xi}, \\
M_{21}
&\coloneqq \sup_{\xi < \xi_*} \frac{|\us_{\xi\xi}|}{5 s^6} \left| w w'' \frac{s - 2\us}{(s-\us)^2} - w w' \frac{2\us}{(s-\us)^3} - \frac{5s}{2} w'' \right|, \\
M_{22}
&\coloneqq \sup_{\xi < \xi_*} \frac{|\us_{\xi\xi\xi}|}{5 s^6 \us_{\xi}} \left| w w' \frac{s - 2\us}{(s-\us)^2} \right|,
&\hspace{-20mm}M_{23}
&\coloneqq \sup_{\xi < \xi_*} \frac{|w' \us_{\xi\xi\xi} + w_{\xi\xi\xi}|}{2 s^5 \us_{\xi}}, \\
M_{24}
&\coloneqq \sup_{\xi < \xi_*} \frac{2(w')^2}{s^5} \left| \frac{s - 2\us}{5s(s-\us)^2} \right| |\us_{\xi\xi}|.
\end{aligned}
\end{equation}

These functions and constants satisfy the following properties:
\begin{lemma}\label{lem:weight-shock}
Assume that $\kappa s^2\le\tfrac{1}{36}$.
Let $\us$ be the degenerate shock profile given in Theorem~\ref{thm:shock_properties} and $w$ be the weight function defined in \eqref{eq:weight}.
Then, the following holds:
\begin{itemize}
\item[(1)] For any $\us \in [-2s, s/2]$, we have
\[
1 - w(\us)\,\frac{s - 2\us}{5s(s-\us)^2} \ge \frac{1}{6}, \qquad \text{and} \qquad H_1 + H_2 > 4s^4.
\]
\item[(2)] The constants $M_{11}, M_{12}, M_{21}, M_{22}, M_{23}, M_{24}$ defined in \eqref{eq:M1234-def} are bounded uniformly with respect to $s$ and $\kappa$.
\end{itemize}
\end{lemma}

\begin{proof}
The proof is given in Appendix~\ref{app:weight}.
\end{proof}

We now provide the proof of Lemma \ref{lem:GS}.
\begin{proof}[Proof of Lemma \ref{lem:GS}]
The proof is divided into two steps, estimating \(G_1^S\) and \(G_2^S\), respectively.

\step{1} 
First of all, we estimate \(G_1^S\). We claim that
\begin{equation} \label{eq:GS1-est}
G_1^S(t)
\ge \frac{1}{8} \int_{-\infty}^{\xi_{*}} \phi_{\xi}^{2} w\, d\xi
+ \frac{1}{2}s^3 \int_{-\infty}^{\xi_{*}} \phi^{2}\us_{\xi} \, d\xi.
\end{equation}
This can be verified as follows.
To apply Lemma \ref{lem:poincare}, we first introduce the change of variables
\[
y \coloneqq \frac{u^S(\xi) - u_-}{u_* - u_-} = \frac{2}{5s}(u^S(\xi) + 2s).
\]
This map is bijective from $\xi \in (-\infty, \xi_*]$ onto $y \in [0, 1]$, by the monotonicity of \(\us\) from \eqref{threshold}.\\
We define $\psi(t,y) \coloneqq \phi(t,\xi)w(u^S(\xi))$ for each fixed \(t>0\).
Then, Lemma \ref{lem:poincare} yields
\begin{equation} \label{eq:poincare-in-GS1}
2 \int_0^1 \psi^2 (y)\, dy
- 2 \left( \int_0^1 \psi (y)\, dy \right)^2
\le  \int_0^1 y(1-y)|\psi'(y)|^2 \, dy.
\end{equation}
Returning to the \(\x\)-coordinate, we can write the terms on the left-hand side of \eqref{eq:poincare-in-GS1} as 
\begin{equation} \label{eq:computation1}
2 \left( \int_0^1 \psi (y)\, dy \right)^2
= \frac{8}{25s^2} \left( \int_{-\infty}^{\xi_{*}} \phi w u^S_{\xi} \, d\xi \right)^{2}, \hspace{14mm}
2 \int_0^1 \psi^2 (y)\, dy
= \frac{4}{5s}\int_{-\infty}^{\xi_{*}} \phi^2 w^2 u^S_{\xi} \, d\xi.
\end{equation}
For the right-hand side, we note that $dy = \frac{2}{5s} u^S_\xi \, d\xi$ and $\psi'(y) = \frac{5s}{2 u^S_\xi} (\phi w)_\xi$.
Since we have
\[
y(1-y)
= \frac{2(\us + 2s)(s - 2\us)}{25s^2}, \hspace{14mm}
\frac{\us + 2s}{\us_\xi}
= \frac{1}{(s-\us)^2}\bigg(1 - \frac{\kappa \us_{\xi\xi}}{\us_\xi} \bigg),
\]
we find that
\begin{equation}\label{eq:computation}
\begin{aligned}
\int_0^1 y(1-y)|\psi'(y)|^2 \, dy 
&= \frac{1}{5s}\int_{-\infty}^{\xi_*} [(\phi w)_\xi]^2 \frac{(u^S + 2s)(s - 2u^S)}{u^S_\xi}  \, d\xi \\
&=  \int_{-\infty}^{\xi_*} [(\phi w)_\xi]^2 Q \, d\xi  - \kappa \int_{-\infty}^{\xi_*} [(\phi w)_\xi]^2 Q \frac{u^S_{\xi\xi}}{u^S_\xi} \, d\xi,
\end{aligned}
\end{equation}
where \(Q\) is given by 
\[
Q(u^S) \coloneqq \frac{s - 2u^S}{5s(s-u^S)^2}.
\]
To proceed further, we expand the first term on the right-hand side as follows: we first observe
\begin{multline*}
\int_{-\infty}^{\xi_*} [(\phi w)_\xi]^2 Q \, d\xi
= \int_{-\infty}^{\xi_*} \left[ \phi_{\xi}^{2} w^2 + 2 \phi \phi_\xi w w'  \us_\xi +  \phi^2 (w')^2  (\us_\xi)^2 \right] Q \, d\xi  \\
= \int_{-\infty}^{\xi_*} \phi_{\xi}^{2} w^2 Q \, d\xi
- \int_{-\infty}^{\xi_*} \phi^2 (\us_\xi)^2 \left( w w'' Q + w w' Q' \right) d\xi   - \int_{-\infty}^{\xi_*} \phi^2 w w' Q \us_{\xi\xi} \, d\xi. 
\end{multline*}
It follows from \eqref{eq:shock-eq} that the last term can be written as 
\begin{align*}
- \int_{-\infty}^{\xi_*} \phi^2 w w' Q \us_{\xi\xi} \, d\xi
&= - \int_{-\infty}^{\xi_*} \phi^2 w w' Q \Big[ 3((\us)^2 - s^2) \us _\xi + \kappa \us_{\xi\xi\xi} \Big] \, d\xi  \\
&= - \int_{-\infty}^{\xi_*} \phi^2 \us_\xi  \Big[ 3 ((\us)^2 - s^2) w w' Q \Big] \, d\xi - \kappa \int_{-\infty}^{\xi_*} \phi^2 w w' Q \us_{\xi\xi\xi} \, d\xi. 
\end{align*}
Combining the above observations, we have
\begin{equation} \label{eq:computation2-1}
\begin{aligned}
\int_{-\infty}^{\xi_*} [(\phi w)_\xi]^2 Q \,d\xi 
&= \int_{-\infty}^{\xi_*} \phi_{\xi}^{2} w^2 Q \,d\xi
- \kappa \int_{-\infty}^{\xi_*} \phi^2 w w' Q \us_{\xi\xi\xi} \,d\xi \\
&\qquad - \int_{-\infty}^{\xi_*} \phi^2 \us_\xi \left[ w w'' Q \us_\xi + w w' Q' \us_\xi + 3 ((\us)^2 - s^2) w w' Q \right] d\xi.
\end{aligned}
\end{equation}
Then, using \eqref{eq:shock-eq} together with $Q = \frac{s - 2\us}{5s(s-\us)^2}$ and $Q' = \frac{-2\us}{5s(s-\us)^3}$, we get
\begin{equation} \label{eq:poly-computation1}
\begin{aligned}
w w'' Q \us_\xi 
&= w w'' \frac{s - 2\us}{5s(s-\us)^2} \bigl((\us+2s)(s-\us)^2+ \kappa \us_{\xi\xi} \bigr) \\
&= w w'' \frac{( s - 2 \us )(\us+2s)}{5s} + \kappa w w'' \frac{s - 2\us}{5s(s-\us)^2} \us_{\xi\xi},
\end{aligned}
\end{equation}
and
\begin{equation} \label{eq:poly-computation2}
\begin{aligned}
&w w' \Big( Q' \us_\xi + 3((\us)^2 - s^2) Q \Big) 
= w w' Q' \bigl( (\us+2s)(s-\us)^2 + \kappa \us_{\xi\xi} \bigr) + 3((\us)^2 - s^2) w w' Q  \\
&\qquad\qquad
= w w' \left[ \frac{-2\us(\us+2s)(s-\us)^2}{5s(s-\us)^3} - \frac{3(s-2\us)(s-\us)(s+\us)}{5s(s-\us)^2} \right] + \kappa w w' Q' \us_{\xi\xi} \\
&\qquad\qquad
= -  w w' \frac{(4\us+3s)}{5s} - \kappa w w' \frac{2\us}{5s(s-\us)^3} \us_{\xi\xi}.
\end{aligned}
\end{equation}
Now, substituting \eqref{eq:poly-computation1} and \eqref{eq:poly-computation2} into \eqref{eq:computation2-1}, we have
\begin{equation} \label{eq:computation2-2}
\begin{aligned}
\int_{-\infty}^{\xi_*} [(\phi w)_\xi]^2 Q \, d\xi 
&= \frac{1}{5s} \int_{-\infty}^{\xi_*} \phi_{\xi}^{2} w^2\frac{(s - 2\us)}{(s-\us)^2}  \,d\xi  \\
&\qquad
- \frac{1}{5s} \int_{-\infty}^{\xi_{*}} \phi^{2}\us_{\xi} \left[ w w'' ( s - 2 \us ) (\us +2s) - w w' (4\us+3s)\right] \, d\xi  \\
&\qquad
- \frac{\kappa}{5s} \int_{-\infty}^{\xi_*} \phi^2 \us_\xi \us_{\xi\xi} \left[ w w'' \frac{s - 2\us}{(s-\us)^2} - w w' \frac{2\us}{(s-\us)^3} \right] d\xi  \\
&\qquad
- \frac{\kappa}{5s} \int_{-\infty}^{\xi_{*}} \phi^{2} w w' \frac{s - 2\us}{(s-\us)^2} \us_{\xi\xi\xi} \, d\xi.
\end{aligned}
\end{equation}
Gathering \eqref{eq:poincare-in-GS1}, \eqref{eq:computation1}, \eqref{eq:computation}, and \eqref{eq:computation2-2}, we obtain
\begin{equation} \label{eq:G1-lower-bound-1}
\begin{aligned}
G^S_1(t)
&\ge \int_{-\infty}^{\xi_*} \phi_{\xi}^{2} w \left[ 1 - w \frac{s - 2\us}{5s(s-\us)^2} \right] \, d\xi
+\frac{1}{5s} \int_{-\infty}^{\xi_*} \phi^2 \us_{\xi} (H_1 + H_2) \, d\xi \\
&\qquad
+ \frac{\kappa}{5s} \int_{-\infty}^{\xi_*} \phi^2 \us_\xi \us_{\xi\xi} \left[ w w'' \frac{s - 2\us}{(s-\us)^2} - w w' \frac{2\us}{(s-\us)^3} - \frac{5s}{2} w'' \right] d\xi  \\
&\qquad
+ \frac{\kappa}{5s} \int_{-\infty}^{\xi_{*}} \phi^{2} w w' \frac{s - 2\us}{(s-\us)^2} \us_{\xi\xi\xi} \, d\xi 
+ \frac{3}{2}\kappa \int_{-\infty}^{\xi_{*}} \phi_\xi^2 w' \us_\xi \, d\xi\\
&\qquad
- \frac{\kappa}{2} \int_{-\infty}^{\xi_{*}} \phi^2 (w' \us_{\xi\xi\xi} + w_{\xi \xi \xi}) \, d\xi
+ \kappa \int_{-\infty}^{\xi_*} [(\phi w)_\xi]^2 \frac{s - 2\us}{5s(s-\us)^2} \frac{\us_{\xi\xi}}{\us_\xi} \, d\xi, 
\end{aligned}
\end{equation}
where \(H_1\) and \(H_2\) are given by \eqref{eq:H1H2-def}.
Using Lemma \ref{lem:weight-shock}(1) and Young's inequality, we deduce that
\begin{align*}
G^S_{1}(t)
&\ge
\frac{1}{6} \int_{-\infty}^{\xi_{*}} \phi_{\xi}^{2} w \, d\xi
+ \frac{4}{5}s^3 \int_{-\infty}^{\xi_{*}} \phi^{2}\us_{\xi} \, d\xi\\
&\qquad
+ \frac{\kappa}{5s} \int_{-\infty}^{\xi_*} \phi^2 \us_\xi \us_{\xi\xi} \left[ w w'' \frac{s - 2\us}{(s-\us)^2} - w w' \frac{2\us}{(s-\us)^3} - \frac{5s}{2} w'' \right] d\xi  \\
&\qquad
+ \frac{\kappa}{5s} \int_{-\infty}^{\xi_{*}} \phi^{2} w w' \frac{s - 2\us}{(s-\us)^2} \us_{\xi\xi\xi} \, d\xi
+ \frac{3}{2}\kappa \int_{-\infty}^{\xi_{*}} \phi_\xi^2 w' \us_\xi \,d\xi\\
&\qquad
- \frac{\kappa}{2} \int_{-\infty}^{\xi_{*}} \phi^2 (w' \us_{\xi\xi\xi} + w_{\xi \xi \xi}) \, d\xi
- 2\kappa \int_{-\infty}^{\xi_*} w^2\phi_\xi^2 \left| \frac{s - 2\us}{5s(s-\us)^2} \right| \frac{|\us_{\xi\xi}|}{\us_\xi} \, d\xi \\
&\qquad
- 2\kappa \int_{-\infty}^{\xi_*} (w')^2\phi^2(\us_\xi)^2 \left| \frac{s - 2\us}{5s(s-\us)^2} \right| \frac{|\us_{\xi\xi}|}{\us_\xi} \, d\xi.
\end{align*}
To control the dispersive terms, namely the terms involving \(\k\), we use the constants $M_{ij}$ in \eqref{eq:M1234-def}, which are bounded uniformly in \(s\) and \(\k\) by Lemma \ref{lem:weight-shock}(2).
Then, we obtain 
\[
G_1^S(t)
\ge \left( \frac{1}{6} - M_1 \kappa s^2 \right) \int_{-\infty}^{\xi_{*}} \phi_{\xi}^{2} w \, d\xi
+ \left( \frac{4}{5} - M_2 \kappa s^2 \right) s^3 \int_{-\infty}^{\xi_{*}} \phi^{2}\us_{\xi} \, d\xi,
\]
where 
\[
M_1 \coloneqq M_{11} + M_{12}, \qquad M_2 \coloneqq M_{21} + M_{22} + M_{23} + M_{24}.
\]
Therefore, by taking $\delta_0 > 0$ sufficiently small so that
\[
\delta_0 \le \frac{1}{36}, \qquad
M_1 \delta_0 \le \frac{1}{24}, \qquad
\text{and} \qquad
M_2 \delta_0 \le \frac{3}{10},
\]
we obtain the desired lower bound \eqref{eq:GS1-est}.
Here, the restriction $\delta_0\le\tfrac{1}{36}$ guarantees the monotonicity condition \eqref{threshold} used above.

\step{2} In this step, we estimate \(G_2^S\).
We here claim that
\begin{equation} \label{eq:GS2-est}
\begin{aligned}
G_{2}^{S}(t)
&\ge \int_{\xi_{*}}^{\infty} \phi_{\xi}^{2} w\,d\xi + \Big(\frac{45}{16}-\frac{9}{8}\ln 2\Big)s^{3}\int_{\xi_{*}}^{\infty}\phi^{2}u_{\xi}^{S} \, d\xi
+ \frac{25}{64}s^{2}|\dot{X}(t)|^{2} + \frac{3}{4}\int_{\mathbb{R}}\phi^{4}|w^{\prime}|u_{\xi}^{S} \, d\xi.
\end{aligned}
\end{equation}
In fact, this estimate for \(G_2^S\) follows from the same argument as \cite[Lemma 4.4]{HWZ-MA}.
Although the shock profiles differ in the dispersive and non-dispersive cases, the terms arising in the estimate share the same structure. Since the proof relies only on standard bounds such as H\"older's and Young's inequalities, we omit the details.

\vspace{2mm}
The desired result \eqref{eq:GS-est} is an immediate consequence of \eqref{eq:GS1-est} and \eqref{eq:GS2-est}.
\end{proof}

Henceforth, $\delta_0>0$ denotes the constant determined in Lemma~\ref{lem:GS}; in particular $\delta_0\le\tfrac{1}{36}<1$.

\subsubsection{Estimation of $N(t)$}
We now estimate $N(t)$ in the shifted coordinates, where $\us$, $w$, and $\ur$ are evaluated as in \eqref{eq:shifted-profiles} (in contrast to the estimate of $G^S(t)$).

\begin{lemma}\label{lem:N(t)}
Let $N(t)$ be defined in Lemma \ref{lem:L2-decomposition}.
Let $0<\varepsilon_R<\delta_1$, where $\delta_1$ is the constant in Lemma~\ref{lem:approx_rarefaction}(6). Then there exists a constant $C>0$, depending only on $q$ and $\theta$, such that for any $\theta\in(0,1)$ and for all $t\in[0,T]$,
\begin{align*}
|N(t)|
&\le \frac{107\sqrt{2}}{2} s^2 \|\phi\|_{H^1(\RR)} \int_{\mathbb{R}} \phi^2 \us_\xi \, d\xi
+ C \left[ \frac{1}{T_0} + s^2\delta_R (\delta_R \varepsilon_R T_0)^{-\theta(1-\frac{1}{2q})} \right] \int_{\mathbb{R}} \phi^2 \us_\xi \, d\xi \\
&\qquad + \Big(\sqrt{2}\|\phi\|_{H^1(\RR)} + \frac{16}{5}\delta_R + C\delta_R (\delta_R \varepsilon_R T_0)^{-\theta(1-\frac{1}{2q})}\Big) \int_{\mathbb{R}} \phi^2 \ur_\xi\, w \, d\xi
+ \frac{25}{128}s^2 |\dot{X}(t)|^2.
\end{align*}
\end{lemma}

\begin{proof}
We split \(N(t)\) into four pieces $N(t) = \sum_{i=1}^4 N_i(t)$ as follows:
\begin{equation*}
\begin{aligned}
N_1(t) &\coloneqq -\int_{\mathbb{R}} \phi^3 (w - 2\us w') \us_\xi \, d\xi, \\
N_2(t) &\coloneqq -\frac{1}{2}\dot{X}(t)\int_{\mathbb{R}} \phi^2 w' \us_\xi \, d\xi, \\
N_3(t) &\coloneqq -\int_{\mathbb{R}} \phi^3 w \ur_\xi \,d\xi + \dot{X}(t)\int_{\mathbb{R}} \phi w \ur_\xi \, d\xi, \\
N_4(t) &\coloneqq 2\int_{\mathbb{R}} \phi^3(\ur - u_m)w' \us_{\xi} \,d\xi + 3\int_{\mathbb{R}} \phi^{2} \us (\ur - u_m) w' \us_{\xi} \, d\xi.
\end{aligned}
\end{equation*}
Then, we first note that
\begin{equation} \label{Liftybd}
\|w\|_{L^\infty(\RR)}=\frac{15}{2}s^2, \qquad
\|w'\|_{L^\infty(\RR)}=\frac{5}{2}s, \qquad
\|u^S\|_{L^\infty(\RR)}=2s.
\end{equation}
The first two terms, \(N_1\) and \(N_2\), can be estimated similarly to \cite[Lemma 4.6]{HWZ-MA}.
Thus we briefly present the argument for completeness.
First, using $w' \le 0$ and $\us_\xi > 0$ with the Sobolev embedding $\|\phi\|_{L^\infty(\RR)} \le \sqrt{2}\|\phi\|_{H^1(\RR)}$, we obtain
\begin{equation*}
|N_1(t)|
\le \|\phi\|_{L^\infty(\RR)}\bigl(\|w\|_{L^\infty(\RR)}
+ 2\|\us\|_{L^\infty(\RR)}\|w'\|_{L^\infty(\RR)}\bigr) \int_{\mathbb{R}} \phi^2 \us_\xi \, d\xi
\le \frac{35\sqrt{2}}{2} s^2 \|\phi\|_{H^1(\RR)} \int_{\mathbb{R}} \phi^2 \us_\xi \, d\xi.
\end{equation*}
Second, by substituting $\dot{X}(t) = -\frac{32}{25s^{2}}\int_{\mathbb{R}}\phi wu_{\xi}^{S} \,d\xi$, we have
\begin{equation*}
    |N_2(t)|
\le \frac{16}{25s^2} \|\phi\|_{L^\infty(\RR)} \|w\|_{L^\infty(\RR)}\|w'\|_{L^\infty(\RR)}
\int_{\mathbb{R}} \us_\xi \, d\xi
\int_{\mathbb{R}} \phi^2 \us_\xi \, d\xi
\le 36\sqrt{2} s^2 \|\phi\|_{H^1(\RR)} \int_{\mathbb{R}} \phi^2 \us_\xi \, d\xi.
\end{equation*}
For $N_3$, using Young's inequality, we find that
\begin{equation*}
\begin{aligned}
|N_3(t)|
&\le  \|\phi\|_{L^\infty(\RR)} \int_\mathbb{R} \phi^2 w u^R_\xi \, d\xi
+\frac{25}{128}s^2|\dot{X}(t)|^2
+ \frac{32}{25 s^2} \left( \int_\mathbb{R} \phi w u^R_\xi \, d \xi \right)^2\\ 
&\le  \|\phi\|_{L^\infty(\RR)} \int_\mathbb{R} \phi^2 w u^R_\xi \, d\xi
+\frac{25}{128}s^2|\dot{X}(t)|^2
+ \frac{32}{25 s^2} \int_\mathbb{R} \phi^2 w u^R_\xi \, d \xi \int_\mathbb{R} w u^R_\xi \, d \xi.
\end{aligned}
\end{equation*}
On $\{\us\ge0\}$, since $w$ is nonincreasing in $\us$, we have $w\le w(0)=\tfrac{5}{2}s^2$, while on $\{\us<0\}$ we use $\|w\|_{L^\infty(\RR)}=\tfrac{15}{2}s^2$. Splitting $\int_\mathbb{R} w\,\ur_\xi\,d\xi$ over these two regions as 
\begin{equation}\label{eq:N3-overlap}
\begin{aligned}
\int_\mathbb{R} w\,\ur_\xi\,d\xi
&\le \frac{5}{2}s^2\int_{\{\us\ge0\}}\ur_\xi\,d\xi
+\frac{15}{2}s^2\int_{\{\us<0\}}\ur_\xi\,d\xi\\
&\le \frac{5}{2}s^2 \delta_R
+\frac{15}{2}s^2\bigl(\ur(t+T_0,\sigma(t+T_0))-u_m\bigr).
%&\le \frac{5}{2}s^2\,\delta_R +C s^2\delta_R(\delta_R\varepsilon_R T_0)^{-\theta(1-\frac{1}{2q})},
\end{aligned}
\end{equation}
For the last term, evaluating the rarefaction at $x=\sigma(t+T_0)$, Lemma~\ref{lem:approx_rarefaction}(6) with $\sigma=w_m$ gives
\begin{equation}\label{eq:l2-est}
\begin{aligned}
  \ur(t+T_0,\sigma(t+T_0))-u_m
&\le C\delta_R\bigl(\delta_R\varepsilon_R(t+T_0)\bigr)^{-\theta(1-\frac{1}{2q})}\bigl(1+\varepsilon_R|\sigma-w_m|(t+T_0)\bigr)^{-q(1-\theta)} \\
&\le C\delta_R(\delta_R\varepsilon_R T_0)^{-\theta(1-\frac{1}{2q})}.  
\end{aligned}
\end{equation}
Combining \eqref{eq:N3-overlap} and \eqref{eq:l2-est} with $\|\phi\|_{L^\infty(\RR)}\le\sqrt2\|\phi\|_{H^1(\RR)}$,
\begin{equation}\label{eq:N3-final}
|N_3(t)|\le\Bigl(\sqrt2\|\phi\|_{H^1(\RR)}+\frac{16}{5}\delta_R+C\delta_R(\delta_R\varepsilon_R T_0)^{-\theta(1-\frac{1}{2q})}\Bigr)\int_\mathbb{R}\phi^2 w\,\ur_\xi\,d\xi
+\frac{25}{128}s^2|\dot X(t)|^2.
\end{equation}
It remains to estimate \(N_4(t)\).
To this end, we first notice from \eqref{eq:weight} that $w'$ does not vanish only when $\{ \xi : \xi - X(t) \le \xi_* \}$.
Since \(\ur\) is strictly increasing, the following supremum is attained at the boundary $\xi - X(t) = \xi_*$:
\[
\sup_{\xi - X(t) \le \xi_*} \big|\ur(t+T_0, \xi - X(t) + \sigma(t+T_0)) - u_m\big|.
\]
Then, by the mean value theorem, we obtain
\begin{equation}\label{eq:overlap_split}
\begin{aligned}
&|\ur(t+T_0, \xi_*+\sigma(t+T_0)) - u_m| \\
&\le |\ur(t+T_0, \xi_*+\sigma(t+T_0)) - \ur(t+T_0, \sigma(t+T_0))| + |\ur(t+T_0, \sigma(t+T_0)) - u_m| \\
&\le \xi_* \|\ur_{\xi}(t+T_0, \cdot)\|_{L^\infty(\RR)} + l_2(t)
\eqqcolon l_1(t)+l_2(t).
\end{aligned}
\end{equation}
For $l_1(t)$, we first bound $\xi_*$. Since $\us(\xi_*)=s/2$, \eqref{eq:right_decay_u} at $\xi_*$ gives
\[
\frac{s}{2}=s-\us(\xi_*)\le\frac{s}{1+2(4\sqrt{3}-6)s^2\xi_*},
\qquad\text{hence}\qquad
\xi_*\le\frac{1}{2(4\sqrt{3}-6)s^2}.
\]
Therefore, by Lemma~\ref{lem:approx_rarefaction}(2), we obtain
\begin{equation} \label{eq:l1-est}
l_1(t)\le\xi_*\,\|\ur_\xi(t+T_0,\cdot)\|_{L^\infty}\le\xi_*\,\frac{C}{t+T_0}\le\frac{C}{s^2 T_0},
\end{equation}
while $l_2(t)$ is bounded in \eqref{eq:l2-est}.

Substituting \eqref{eq:l1-est} and \eqref{eq:l2-est} into \eqref{eq:overlap_split} and using \eqref{Liftybd} with $\|\phi\|_{L^\infty(\RR)} \le \sqrt{2}\|\phi\|_{H^1(\RR)}$, we get
\begin{align*} 
|N_4(t)|
&\le C \left[ \frac{1}{s^2 T_0} + \delta_R (\delta_R \varepsilon_R T_0)^{-\theta(1-\frac{1}{2q})} \right]
\Big(\frac{5}{2}s\Big)
\bigl(2\sqrt{2}\|\phi\|_{H^1(\RR)} + 6s\bigr) \int_{\mathbb{R}} \phi^2 \us_\xi \, d\xi \\
&\le C \left[ \frac{1}{T_0} + s^2\delta_R (\delta_R \varepsilon_R T_0)^{-\theta(1-\frac{1}{2q})} \right]
\int_{\mathbb{R}} \phi^2 \us_\xi \, d\xi.
\end{align*}

Finally, combining the estimates for $N_i(t)$ ($i=1,2,3,4$) yields the desired bound for $N(t)$.
\end{proof}

\begin{lemma} \label{lem:GS-N}
Let $\delta_0$ and $\delta_1$ be the constants given in Lemma~\ref{lem:GS} and Lemma~\ref{lem:approx_rarefaction}(6), respectively. There exist positive constants $\delta_2\le\min\{\delta_1,1\}$ (depending only on $q$, $u_m$, and $\delta_R$) and $\varepsilon_1=\varepsilon_1(\delta_S)$ such that if the dispersion coefficient $\kappa>0$ and the wave strengths $\ds$ and $\delta_R$ satisfy
\begin{equation*}
\kappa u_m^2 \le \delta_0,\qquad \delta_R \le \frac{5}{18} \delta_S = \frac{5}{6}u_m,\qquad 0 < \varepsilon_R < \delta_2,
\end{equation*}
and if the perturbation $\phi$ solving \eqref{eq:perturbed_general} satisfies $\sup_{t \in [0, T]} \| \phi (t,\cdot) \|_{H^1(\mathbb{R})} \le \varepsilon_1$, then for all $t \in [0, T]$,
\begin{multline*}
\frac{1}{2} \int_\mathbb{R} (\phi^2 w)(t,\xi) \, d\xi +\frac{15}{64}u_m^2 \int_0^t \int_\mathbb{R} \phi_\xi^2 \,d\xi \,d\tau +\frac{1}{4}u_m^3 \int_0^t \int_\mathbb{R} \phi^2 (u^S_\xi + u^R_\xi ) \,d\xi \,d\tau\\
\quad +\frac{25}{128}u_m^2 \int_0^t  |\dot{X}(\tau)|^2 \, d\tau \le \frac{1}{2} \int_\mathbb{R} (\phi^2 w)(0,\xi) \, d\xi + \int_0^t  J(\tau)\, d\tau +\int_0^t R(\tau)\, d\tau.
\end{multline*}
\end{lemma}

\begin{proof}
Integrating \eqref{eq:L2-energy-identity} with \eqref{eq:L2-decomposition} over $[0,t]$, we have
\begin{equation}\label{eq:GSN-integrated}
\frac{1}{2}\int_\mathbb{R}(\phi^2 w)(t,\xi)\,d\xi+\int_0^t\!\bigl(G^S+G^R+G^{SR}\bigr)d\tau
=\frac{1}{2}\int_\mathbb{R}(\phi^2 w)(0,\xi)\,d\xi+\int_0^t\!\bigl(N+J+R\bigr)d\tau.
\end{equation}
Since $\ur>u_m$, $w>0$, $w'\le0$, and $\us_\xi>0$, both integrals in $G^{SR}$ are nonnegative, so $G^{SR}\ge0$. Moreover, by Lemma~\ref{lem:GS} under $\kappa u_m^2\le\delta_0$ and by $\ur\ge u_m$, we have
\begin{equation}\label{eq:GSN-lower}
G^S\ge\frac{15}{64}u_m^2\!\int_\mathbb{R}\phi_\xi^2\,d\xi+\frac{1}{2} u_m^3\!\int_\mathbb{R}\phi^2\us_\xi\,d\xi+\frac{25}{64}u_m^2|\dot X|^2,
\qquad
G^R\ge 3u_m\!\int_\mathbb{R}\phi^2 w\,\ur_\xi\,d\xi.
\end{equation}
It remains to absorb $N$ into these lower bounds. Since $0<\varepsilon_R<\delta_2\le\delta_1$, where $\delta_2$ is fixed in \eqref{eq:delta2} below, Lemma~\ref{lem:N(t)} applies; using the a priori bound $\|\phi\|_{H^1(\RR)}\le\varepsilon_1$ and $T_0=\varepsilon_R^{-8}$, we have
\begin{align*}
N &\le \underbrace{\frac{107}{2}\sqrt2\,u_m^2\varepsilon_1\!\int_\mathbb{R}\phi^2\us_\xi\,d\xi}_{=:N^S_1}
+\underbrace{C\Bigl[\frac{1}{T_0}+u_m^2\delta_R(\delta_R\varepsilon_R T_0)^{-\theta(1-\frac{1}{2q})}\Bigr]\!\int_\mathbb{R}\phi^2\us_\xi\,d\xi}_{=:N^S_2} \\
&\quad +\underbrace{\Bigl(\sqrt2\varepsilon_1+\frac{16}{5}\delta_R+C\delta_R(\delta_R\varepsilon_R T_0)^{-\theta(1-\frac{1}{2q})}\Bigr)\!\int_\mathbb{R}\phi^2 w\,\ur_\xi\,d\xi}_{=:N^R}
+\frac{25}{128}u_m^2|\dot X|^2.
\end{align*}
The shift term $\tfrac{25}{128}u_m^2|\dot X|^2$ is absorbed by the $\tfrac{25}{64}u_m^2|\dot X|^2$ in $G^S$. We fix $\varepsilon_1$ so that $\sqrt2\,\varepsilon_1\le(428)^{-1}u_m$.
Then, we obtain
\begin{equation}\label{eq:NS1-bound}
N^S_1\le\frac18 u_m^3\int_\mathbb{R}\phi^2\us_\xi\,d\xi.
\end{equation}
By $T_0=\varepsilon_R^{-8}$, the two coefficients of $N^S_2$ are
\begin{equation}\label{eq:GSN-shock-coeff}
\frac{C}{T_0}=C\varepsilon_R^{8},
\qquad
Cu_m^2\delta_R(\delta_R\varepsilon_R T_0)^{-\theta(1-\frac{1}{2q})}=Cu_m^2\delta_R^{\,1-\theta(1-\frac{1}{2q})}\varepsilon_R^{\,7\theta(1-\frac{1}{2q})}.
\end{equation}
If we set
\begin{equation}\label{eq:delta2}
\delta_2\coloneqq\min\left\{\delta_1,\ \bigl(16C u_m^{-3}\bigr)^{-1/8},\ \Bigl(16C u_m^{-1}\delta_R^{\,1-\theta(1-\frac{1}{2q})}\Bigr)^{-\frac{1}{7\theta(1-\frac{1}{2q})}}\right\},
\end{equation}
then each coefficient in \eqref{eq:GSN-shock-coeff} is at most $\tfrac{1}{16}u_m^3$ for $\varepsilon_R<\delta_2$, so that
\begin{equation}\label{eq:NS2-bound}
N^S_2\le\frac18 u_m^3\int_\mathbb{R}\phi^2\us_\xi\,d\xi.
\end{equation}
Combining \eqref{eq:NS1-bound} and \eqref{eq:NS2-bound}, we have
\begin{equation*}
N^S_1+N^S_2\le\frac14 u_m^3\int_\mathbb{R}\phi^2\us_\xi\,d\xi,
\end{equation*}
which is absorbed by the $\tfrac12 u_m^3\int_\mathbb{R}\phi^2\us_\xi\,d\xi$ in $G^S$, leaving $\tfrac14 u_m^3\int_\mathbb{R}\phi^2\us_\xi\,d\xi$.

For the rarefaction contribution, subtracting $N^R$ from the lower bound for $G^R$ in \eqref{eq:GSN-lower}, we have
\begin{equation}\label{eq:GSN-rare}
G^R-N^R\ge\Bigl(3u_m-\frac{16}{5}\delta_R-\sqrt2\varepsilon_1-C\delta_R(\delta_R\varepsilon_R T_0)^{-\theta(1-\frac{1}{2q})}\Bigr)\!\int_\mathbb{R}\phi^2 w\,\ur_\xi\,d\xi.
\end{equation}
Since $\delta_R\le\tfrac{5}{18}\delta_S=\tfrac56 u_m$, we have $\tfrac{16}{5}\delta_R\le\tfrac{8}{3}u_m$. Combined with $\sqrt2\varepsilon_1\le(428)^{-1}u_m$ and the estimate \eqref{eq:delta2}, the coefficient in \eqref{eq:GSN-rare} is bounded below by
\begin{equation}\label{eq:GSN-rare-coeff}
3u_m-\frac{8}{3}u_m-(428)^{-1}u_m-\frac{1}{16}u_m\ge\frac{2}{15}u_m.
\end{equation}
Moreover, since $w\ge\tfrac{15}{8}u_m^2$ on $\ur_\xi>0$, we obtain $\int_\mathbb{R}\phi^2 w\,\ur_\xi\,d\xi\ge\tfrac{15}{8}u_m^2\int_\mathbb{R}\phi^2\ur_\xi\,d\xi$, which yields
\[
G^R-N^R\ge\frac{2}{15}u_m\cdot\frac{15}{8}u_m^2\int_\mathbb{R}\phi^2\ur_\xi\,d\xi=\frac14 u_m^3\int_\mathbb{R}\phi^2\ur_\xi\,d\xi.
\]
Substituting these bounds into \eqref{eq:GSN-integrated}, we obtain the desired estimate.
\end{proof}

\subsubsection{Estimation of $J(t)$ and $R(t)$}
The following lemma provides a set of inequalities which will be useful for the estimation of \(J(t)\) and \(R(t)\).
\begin{lemma}\label{lem:wave_interaction}
Let $0<\varepsilon_R<\delta_1$, with $\delta_1$ as in Lemma~\ref{lem:approx_rarefaction}(6). Then, for any $\theta\in(0,1)$, there exists a constant $C>0$, depending only on $q$ and $\theta$, such that for all $t\ge0$ the following holds:
    \begin{align}
        &\int_{-\infty}^{0} |u^S - u_m| u^R_\xi \,d\xi \le C \delta_S \delta_R (\delta_R \varepsilon_R (t+T_0) )^{-\theta(1-\frac{1}{2q})}, \label{eq:int1_c} \\
        &\int_{0}^{\infty} |u^S - u_m| u^R_\xi \,d\xi \le C \delta_S^{2\theta-1} \delta_R^\theta (t+T_0)^{-1+\theta}, \label{eq:int2_c} \\
        &\int_{-\infty}^{0} |u^R - u_m| u^S_\xi \,d\xi \le C \delta_S \delta_R (\delta_R \varepsilon_R (t+T_0))^{-\theta(1-\frac{1}{2q})}, \label{eq:int3_c} \\
        &\int_{0}^{(w_+-w_m)(t+T_0)} |u^R - u^r| u^S_\xi \,d\xi \le C \delta_S \delta_R (\delta_R \varepsilon_R (t+T_0) )^{-(1-\frac{1}{2q})}, \label{eq:int4_c} \\
        &\int_{0}^{(w_+-w_m)(t+T_0)} |u^r - u_m| u^S_\xi \,d\xi \le C \delta_S^{-1} (t+T_0)^{-1} \ln\bigl(1 + C\delta_S^2 \delta_R (t+T_0)\bigr), \label{eq:int5_c} \\
       &\int_{(w_+-w_m)(t+T_0)}^{\infty} |u^R - u_m| u^S_\xi \,d\xi \le C \delta_S \delta_R \bigl(1 + C \delta_S^2 \delta_R (t+T_0)\bigr)^{-1}, \label{eq:int6_c}
    \end{align}
    where $\us=\us(\xi), \ur=\ur(t+T_0, \xi+\sigma(t+T_0)),$ and $u^r=u^r\left(\frac{\xi+\sigma(t+T_0)}{t+T_0}\right).$
\end{lemma}

\begin{proof}[Proof of Lemma \ref{lem:wave_interaction}]
$\bullet$ Proof of \eqref{eq:int1_c}: By Theorem~\ref{thm:shock_properties} and Lemma~\ref{lem:approx_rarefaction}(2), we obtain
\begin{align*}
\int_{-\infty}^{0} |u^S - u_m| u^R_\xi \,d\xi &\le C \delta_S \int_{-\infty}^{0} u^R_\xi(t+T_0, \xi+\sigma(t+T_0)) \,d\xi \\
&= C \delta_S |u^R(t+T_0, \sigma(t+T_0)) - u_m|
\le C \delta_S \delta_R (\delta_R \varepsilon_R (t+T_0) )^{-\theta(1-\frac{1}{2q})}.
\end{align*}

\noindent$\bullet$ Proof of \eqref{eq:int2_c}: Applying Hölder's inequality for any $\theta \in (0,1)$, we have
\begin{equation*}
\begin{aligned}
&\int_{0}^{\infty} |u^S(\xi) - u_m| u^R_\xi(t+T_0,\xi+\sigma(t+T_0))  \,d\xi \\
&\qquad\le \|\ur_{\xi}(t+T_0, \cdot)\|_{L^{1/\theta}}
\left( \int_{0}^{\infty} |u^S(\xi) - u_m|^{\frac{1}{1-\theta}} \,d\xi \right)^{1-\theta}  \\
&\qquad\le C\delta_S \left( C \delta_R^\theta (t+T_0)^{-1+\theta} \right)
\left( \int_{0}^{\infty} (1+C\delta_S^2 |\xi|)^{-\frac{1}{1-\theta}} \,d\xi \right)^{1-\theta}  
\le C \delta_S^{2\theta-1} \delta_R^\theta (t+T_0)^{-1+\theta}.
\end{aligned}
\end{equation*}

\noindent$\bullet$ Proof of \eqref{eq:int3_c}: For $\xi \le 0$, since $\ur$ is strictly increasing, Lemma \ref{lem:approx_rarefaction}(6) with $\sigma=w_m$ implies 
\[
\sup_{\xi \le 0} |\ur(t+T_0, \xi + \sigma(t+T_0)) - u_m|
= |\ur(t+T_0, \sigma (t+T_0)) - u_m| 
\le  C \delta_R (\delta_R \varepsilon_R (t+T_0) )^{-\theta(1-\frac{1}{2q})}.
\]
Therefore, we obtain
\begin{align*}
\int_{-\infty}^{0} |u^R - u_m| u^S_\xi \,d\xi 
&\le \sup_{\xi \le 0} |u^R(t+T_0, \xi + \sigma(t+T_0)) - u_m| \int_{-\infty}^{0} u^S_\xi \,d\xi \\
&\le C \delta_S \delta_R  (\delta_R \varepsilon_R (t+T_0))^{-\theta(1-\frac{1}{2q})}.
\end{align*}

\noindent$\bullet$ Proof of \eqref{eq:int4_c}: It follows from Lemma \ref{lem:approx_rarefaction}(7) that
\begin{align*}
\int_{0}^{(w_+-w_m)(t+T_0)} |u^R - u^r| u^S_\xi \,d\xi &\le \left\| u^R(t+T_0, \cdot) - u^r\left( \frac{\cdot}{t+T_0} \right) \right\|_{L^\infty} \int_{0}^{\infty} u^S_\xi \,d\xi \\
&\le C \delta_S  \delta_R^{\frac{1}{2q}} \varepsilon_R^{-1+\frac{1}{2q}} (t+T_0)^{-1+\frac{1}{2q}} .
\end{align*}

\noindent$\bullet$ Proof of \eqref{eq:int5_c}: Inside the centered rarefaction fan where $\sigma = w_m$, the inviscid centered rarefaction wave satisfies
\[
\left| u^r\left(\frac{\xi+\sigma (t+T_0)}{t+T_0}\right) - u_m \right| \le \frac{C|\xi|}{t+T_0}.
\]
Then, using Theorem~\ref{thm:shock_properties}, we obtain
\begin{align*}
\int_{0}^{(w_+-w_m)(t+T_0)} |u^r - u_m| u^S_\xi \,d\xi &\le \int_0^{(w_+-w_m)(t+T_0)} \frac{C\xi}{t+T_0} \frac{C\delta_S^3}{(1+C\delta_S^2 \xi)^2} \, d\xi \\
&\le C \delta_S^{-1} (t+T_0)^{-1} \ln\bigl(1+C\delta_S^2 \delta_R (t+T_0) \bigr).
\end{align*}

\noindent$\bullet$ Proof of \eqref{eq:int6_c}: Using $|u^R - u_m| \le \delta_R$ and Theorem~\ref{thm:shock_properties}, we get 
\[
\int_{(w_+-w_m)(t+T_0)}^{\infty} |u^R - u_m| u^S_\xi \, d\xi
\le \delta_R \int_{(w_+-w_m)(t+T_0)}^{\infty} \us_\xi \,d\xi 
\le C \delta_S \delta_R \bigl(1+C\delta_S^2 \delta_R (t+T_0)\bigr)^{-1}.
\]
%\begin{align*}
%\int_{(w_+-w_m)(t+T_0)}^{\infty} |u^R - u_m| u^S_\xi \, d\xi &\le \delta_R \int_{(w_+-w_m)(t+T_0)}^{\infty} \us_\xi \,d\xi \\
%&\le \frac{C \delta_S \delta_R}{1+C\delta_S^2 \bigl((w_+-w_m)(t+T_0)\bigr)} \\
%&\le C \delta_S \delta_R \bigl(1+C\delta_S^2 \delta_R (t+T_0)\bigr)^{-1}.
%\end{align*}
This completes the proof of Lemma \ref{lem:wave_interaction}.
\end{proof}

\begin{lemma}\label{lem:J(t)}
Recall $J(t)=-3\int_{\mathbb{R}}\phi^2(\us-u_m)w\ur_\xi\,d\xi$. Let $0<\varepsilon_R<\delta_1$, with $\delta_1$ as in Lemma~\ref{lem:approx_rarefaction}(6).
Then there exists a constant $C>0$, depending only on $q$, such that for all $t\in[0,T]$,
\begin{equation} \label{eq:J}
\int_0^t |J(\tau)| \,d\tau \le \frac{15}{256} u_m^2 \int_0^t \|\phi_\xi\|_{L^2}^2 \,d\tau + C u_m^2 \max\{ \delta_S^2 \varepsilon_R^{-8/5}, \delta_S^{-6/5} \} \delta_R^{2/5} T_0^{-3/5} \sup_{0 \le \tau \le t} \|\phi(\tau)\|_{L^2}^2.
\end{equation}
\end{lemma}
\begin{proof}
Applying the Sobolev embedding $\|\phi\|_{L^\infty}^2 \le 2\|\phi\|_{L^2}\|\phi_\xi\|_{L^2}$ and Young's inequality, we estimate the integrand of $J(t)$:
\begin{equation} \label{eq:J(t)-est}
\begin{aligned}
|J(t)| &\le 3 \|w\|_{L^\infty} \|\phi\|_{L^\infty}^2 \int_{\mathbb{R}} |u^S - u_m| u^R_\xi \,d\xi
\le 45 u_m^2 \|\phi\|_{L^2} \|\phi_\xi\|_{L^2} \int_{\mathbb{R}} |u^S - u_m| u^R_\xi \,d\xi \\
&\le \frac{15}{256} u_m^2 \|\phi_\xi\|_{L^2}^2 + C u_m^2 \|\phi\|_{L^2}^2 \left( \int_{\mathbb{R}} |u^S - u_m| u^R_\xi \,d\xi \right)^2.
\end{aligned}
\end{equation}
We apply the wave interaction estimates \eqref{eq:int1_c} and \eqref{eq:int2_c} from Lemma \ref{lem:wave_interaction}:
\begin{equation*}
\begin{aligned}
\int_{\mathbb{R}} |u^S - u_m| u^R_\xi \,d\xi 
&\le C \delta_S \delta_R \bigl(\delta_R \varepsilon_R (\tau+T_0)\bigr)^{-\theta_1(1-\frac{1}{2q})} + C \delta_S^{2\theta_2-1} \delta_R^{\theta_2} (\tau+T_0)^{-1+\theta_2},
\end{aligned}
\end{equation*}
where the constants $\theta_1, \theta_2 \in (0, 1)$ are chosen such that $\theta_1(1-\frac{1}{2q}) = \frac{4}{5}$ and $\theta_2 = 1/5$. Thus, we have
\begin{equation} \label{eq:J-interaction-est}
\begin{aligned}
\int_0^t \left( \int_{\mathbb{R}} |u^S - u_m| u^R_\xi \,d\xi \right)^2 \, d \tau
&\le C \max\{ \delta_S^2 \varepsilon_R^{-8/5}, \delta_S^{-6/5} \} \delta_R^{2/5} \int_0^\infty (\tau+T_0)^{-8/5} \, d \tau \\
&\le C \max\{ \delta_S^2 \varepsilon_R^{-8/5}, \delta_S^{-6/5} \} \delta_R^{2/5} T_0^{-3/5}.
\end{aligned}
\end{equation}

Substituting \eqref{eq:J-interaction-est} into \eqref{eq:J(t)-est} yields \eqref{eq:J}. This completes the proof of Lemma \ref{lem:J(t)}.
\end{proof}

\begin{lemma}\label{lem:R(t)}
Recall $R(t) = -\int_{\mathbb{R}} \phi E w \,d\xi$.
Let $0<\varepsilon_R<\delta_1$, with $\delta_1$ as in Lemma~\ref{lem:approx_rarefaction}(6).
Then there exists a constant $C>0$, depending only on $q$, such that for all $t \in [0,T]$, the following holds:
\begin{equation} \label{eq:R}
\begin{aligned}
\int_0^t |R(\tau)| \,d\tau &\le \frac{15}{256} u_m^2 \int_0^t \|\phi_\xi\|_{L^2}^2 \,d\tau + C u_m^2 \sup_{0 \le \tau \le t} \|\phi(\tau)\|_{L^2}^{2/3} \Bigg[ \delta_S^{-4/5} \delta_R^{4/15} T_0^{-1/15} \\
&\qquad + \left( \delta_S^{4/3} \varepsilon_R^{-7/6} + \delta_S^{-1} \right) \delta_R^{1/6} T_0^{-1/6} + \delta_S^{-4/3} T_0^{-1/3}\\
&\qquad + \left( \delta_R^{4/3} + \delta_R^{1/3} \right) \varepsilon_R^{\frac{q-4}{3(q-1)}}
+ \kappa^{4/3}\left( \delta_R^{4/3} + \delta_R^{1/3} \right) \varepsilon_R^{\frac{2q-16}{3(q-2)}}\Bigg].
\end{aligned}
\end{equation}
\end{lemma}
\begin{proof}
Using the Sobolev embedding $\|\phi\|_{L^\infty}^2 \le 2\|\phi\|_{L^2}\|\phi_\xi\|_{L^2}$ and Young's inequality, we obtain
\begin{equation} \label{eq:R(t)-est}
\begin{aligned}
|R(t)| \le \|w\|_{L^\infty} \|\phi\|_{L^\infty} \int_{\mathbb{R}} |E| \,d\xi 
&\le C u_m^2 \|\phi\|_{L^2}^{1/2} \|\phi_\xi\|_{L^2}^{1/2} \int_{\mathbb{R}} |E| \,d\xi \\
&\le \frac{15}{256} u_m^2 \|\phi_\xi\|_{L^2}^2 + C u_m^2 \|\phi\|_{L^2}^{2/3} \left( \int_{\mathbb{R}} |E| \,d\xi \right)^{4/3}.
\end{aligned}
\end{equation}
Here the $L^1$-norm of the error term $E$ is bounded by
\begin{align*}
\int_{\mathbb{R}} |E| \,d\xi &\le C \int_{\mathbb{R}} |u^S - u_m| u^R_\xi \,d\xi + C \int_{\mathbb{R}} |u^R - u_m| u^S_\xi \,d\xi + \int_{\mathbb{R}} |u^R_{\xi\xi}| \,d\xi + \kappa \int_{\mathbb{R}} |u^R_{\xi\xi\xi}| \,d\xi \\
&=:E_{11}^I+E_{12}^I+E_{11}^R+E_{12}^R.
\end{align*}
We estimate the four terms on the right-hand side in turn. Throughout we write $\tau' := \tau + T_0$, and by translation invariance we absorb the shift $X(\tau)$ via $\xi - X(\tau) \mapsto \xi$.

\vspace{2mm}
\noindent \textbf{(i) $E_{11}^I$:}
By the interaction estimates \eqref{eq:int1_c} and \eqref{eq:int2_c} of Lemma~\ref{lem:wave_interaction}, we have, for every $\tau\in[0,t]$,
\begin{equation*}
\int_{\mathbb{R}} |u^S - u_m| u^R_\xi \,d\xi 
\le C \delta_S \delta_R \bigl(\delta_R \varepsilon_R (\tau+T_0)\bigr)^{-\theta_1(1-\frac{1}{2q})} + C \delta_S^{2\theta_2-1} \delta_R^{\theta_2} (\tau+T_0)^{-1+\theta_2},
\end{equation*}
where now the exponents $\theta_1,\theta_2\in(0,1)$ are chosen such that $\theta_1(1-\tfrac{1}{2q})=\tfrac{7}{8}$ and $\theta_2=\tfrac15$. Taking the $4/3$ power and integrating in time, we obtain
\begin{equation} \label{eq:E11I-est}
\int_0^t \left( \int_{\mathbb{R}} |u^S - u_m| u^R_\xi \,d\xi \right)^{4/3} d \tau \le C \delta_S^{4/3} \delta_R^{1/6} \varepsilon_R^{-7/6} T_0^{-1/6} + C \delta_S^{-4/5} \delta_R^{4/15} T_0^{-1/15}.
\end{equation}

\vspace{2mm}
\noindent \textbf{(ii) $E_{12}^I$:}
Combining \eqref{eq:int3_c}--\eqref{eq:int6_c} of Lemma~\ref{lem:wave_interaction}, we have, for every $\tau\in[0,t]$,
\begin{equation*}
\begin{aligned}
\int_{\mathbb{R}} |u^R - u_m| u^S_\xi \,d\xi 
&\le  \underbrace{C \delta_S \delta_R \bigl(\delta_R \varepsilon_R (\tau+T_0)\bigr)^{-\theta(1-\frac{1}{2q})}}_{\eqref{eq:int3_c}} + \underbrace{C \delta_S \delta_R \bigl(\delta_R \varepsilon_R (\tau+T_0)\bigr)^{-(1-\frac{1}{2q})}}_{\eqref{eq:int4_c}} \\
&\qquad + \underbrace{C \delta_S^{-1} (\tau+T_0)^{-1} \ln\bigl(1 + C\delta_S^2 \delta_R (\tau+T_0)\bigr)}_{\eqref{eq:int5_c}} + \underbrace{C \delta_S \delta_R \bigl(1 + C \delta_S^2 \delta_R (\tau+T_0)\bigr)^{-1}}_{\eqref{eq:int6_c}}, 
\end{aligned}
\end{equation*}
where $\theta\in(0,1)$ is chosen so that $\theta(1-\tfrac{1}{2q})=\tfrac{7}{8}$.
Here the contribution of \eqref{eq:int4_c} is absorbed into that of \eqref{eq:int3_c}, since $\tfrac43(1-\tfrac{1}{2q})\ge\tfrac76$ and $\delta_R\varepsilon_R(\tau+T_0)\ge\delta_R\varepsilon_R T_0\ge c>0$ by $0<\varepsilon_R<\delta_1$.
Estimating the logarithm by $(\ln(1+Ax))^{4/3}\le C(Ax)^{1/6}$ and integrating in time, we obtain
\begin{equation} \label{eq:E12I-est}
\begin{aligned}
\int_0^t \left( \int_{\mathbb{R}} |u^R - u_m| u^S_\xi \,d\xi \right)^{4/3} \, d \tau 
&\le C \int_0^t \Big[ \delta_S^{4/3} \delta_R^{4/3} \bigl(\delta_R \varepsilon_R (\tau+T_0)\bigr)^{-7/6} \\
&\qquad\qquad + \delta_S^{-4/3} (\tau+T_0)^{-4/3} \bigl( \delta_S^2 \delta_R (\tau+T_0) \bigr)^{1/6} \\
&\qquad\qquad + \delta_S^{4/3} \delta_R^{4/3} \bigl(1 + C \delta_S^2 \delta_R (\tau+T_0)\bigr)^{-4/3} \Big] d\tau \\
&\le C \left( \delta_S^{4/3} \varepsilon_R^{-7/6} + \delta_S^{-1} \right) \delta_R^{1/6} T_0^{-1/6} + C \delta_S^{-4/3} T_0^{-1/3}.
\end{aligned}
\end{equation}

\vspace{2mm}
\noindent \textbf{(iii) $E_{11}^R$ and $E_{12}^R$:}
Let $t_*$ and $t^*$ be given by
\[
t_* := \left( \frac{\delta_R + \delta_R^{1/q}}{\delta_R \varepsilon_R} \right)^{\frac{q}{q-1}} \quad \text{and} \quad 
t^* := \left( \frac{\delta_R + \delta_R^{2/q}}{\delta_R \varepsilon_R^2} \right)^{\frac{q}{q-2}}.
\]
By Lemma \ref{lem:approx_rarefaction}, we have
\begin{equation*} 
\|u^R_{\xi\xi}(\tau+T_0,\cdot)\|_{L^1}  \le C 
\begin{cases} 
\delta_R \varepsilon_R, &  \tau+T_0  \le t_*, \\
(\delta_R + \delta_R^{1/q})(\tau+T_0)^{-(1-1/q)} , &  \tau+T_0 > t_*,
\end{cases}
\end{equation*}
and
\begin{equation*}
\|u^R_{\xi\xi\xi}(\tau+T_0,\cdot)\|_{L^1} \le C 
\begin{cases} 
\delta_R \varepsilon_R^2, & \tau+T_0 \le t^*, \\
(\delta_R + \delta_R^{2/q})(\tau+T_0)^{-(1-2/q)} , & \tau+T_0 > t^*.
\end{cases}
\end{equation*}
Using the change of variable $s := \tau + T_0$ and \(q\ge 10\) for the time-integrability, we obtain
\begin{equation} \label{eq:E11R-est}
\begin{aligned}
\int_0^t \|u^R_{\xi\xi}\|_{L^1}^{4/3} \,d\tau
\le \int_0^\infty \|u^R_{\xi\xi}\|_{L^1}^{4/3} \,ds 
&\le \int_0^{t_*} (\delta_R \varepsilon_R)^{4/3} \,ds + \int_{t_*}^\infty \left( (\delta_R + \delta_R^{1/q}) s^{-(1-1/q)} \right)^{4/3} ds \\
&\le C (\delta_R \varepsilon_R)^{4/3} t_* \le C \left( \delta_R^{4/3} + \delta_R^{1/3} \right) \varepsilon_R^{\frac{q-4}{3(q-1)}},
\end{aligned}
\end{equation}
and, similarly, we have
\begin{equation} \label{eq:E12R-est}
\begin{aligned}
\int_0^t \|u^R_{\xi\xi\xi}\|_{L^1}^{4/3} \,d\tau
\le \int_0^\infty \|u^R_{\xi\xi\xi}\|_{L^1}^{4/3} \,ds 
&\le \int_0^{t^*} (\delta_R \varepsilon_R^2)^{4/3} \,ds + \int_{t^*}^\infty \left( (\delta_R + \delta_R^{2/q}) s^{-(1-2/q)} \right)^{4/3} ds \\
&\le C (\delta_R \varepsilon_R^2)^{4/3} t^* \le C \left( \delta_R^{4/3} + \delta_R^{1/3} \right) \varepsilon_R^{\frac{2q-16}{3(q-2)}}.
\end{aligned}
\end{equation}

Substituting \eqref{eq:E11I-est}--\eqref{eq:E12R-est} into \eqref{eq:R(t)-est} yields \eqref{eq:R}.
This completes the proof of Lemma \ref{lem:R(t)}.
\end{proof}

\subsubsection{Conclusion}
Gathering the above estimates, we now prove Proposition \ref{prop:a-priori-L2}.
\begin{lemma} \label{lem:a-priori-L^2}    
Let $\delta_0$ be given in Lemma~\ref{lem:GS}, and let $\delta_2$ and $\varepsilon_1=\varepsilon_1(\delta_S)$ be given in Lemma~\ref{lem:GS-N}. If the dispersion coefficient $\kappa > 0$, the shock strength $\ds=|u_m-u_-|=3u_m>0$, the rarefaction strength $\delta_R=|u_+-u_m| > 0$, and the smoothing scale $\varepsilon_R$ satisfy
\begin{equation*}
\kappa u_m^2 \le \delta_0, \qquad
\delta_R \le \frac{5}{18} \delta_S \; \left(= \frac{5}{6}u_m\right), \qquad
0 < \varepsilon_R < \delta_2,
\end{equation*}
and the perturbation $\phi \in C([0,T];H^1(\mathbb{R})) \cap L^2(0,T;H^2(\mathbb{R}))$ solves \eqref{eq:perturbed_general} on $[0,T]$ with
\[
\sup_{t \in [0, T]} \| \phi (t,\cdot) \|_{H^1(\mathbb{R})} \le \varepsilon_1,
\]
then there exists a constant $C > 0$, independent of $\kappa$, $\varepsilon_R$, and $T$, such that for all $t \in [0, T]$,
\begin{equation} \label{eq:a-priori-L2-est}
\begin{aligned}
&\frac{1}{2} \int_\mathbb{R} (\phi^2 w)(t,\xi) \, d \xi +\frac{15}{128}u_m^2 \int_0^t \int_\mathbb{R} \phi_\xi^2 \,d\xi d\tau \\
&\qquad +\frac{1}{4}u_m^3 \int_0^t \int_\mathbb{R} \phi^2 (u^S_\xi + u^R_\xi ) \,d\xi d\tau +\frac{25}{128}u_m^2 \int_0^t  |\dot{X}(\tau)|^2 \, d\tau \\
&\le \frac{1}{2} \int_\mathbb{R} (\phi^2 w)(0,\xi) \, d \xi
+ C \varepsilon_R^{1/6} \left( \sup_{0 \le \tau \le t} \|\phi(\tau)\|_{L^2}^2 + 1 \right).
\end{aligned}
\end{equation}
\end{lemma}

\begin{proof}
Since $0<\varepsilon_R<\delta_2\le\delta_1$, Lemmas~\ref{lem:GS-N}, \ref{lem:J(t)}, and \ref{lem:R(t)} apply. Combining them, we have
\begin{equation} \label{eq:L2-est-comb}
\begin{aligned}
&\frac{1}{2} \int_\mathbb{R} (\phi^2 w)(t,\xi) \, d \xi +\frac{15}{128}u_m^2 \int_0^t \int_\mathbb{R} \phi_\xi^2(\tau,\xi) \,d\xi d\tau \\
&\qquad +\frac{1}{4}u_m^3 \int_0^t \int_\mathbb{R} \bigl[ \phi^2 (u^S_\xi + u^R_\xi ) \bigr] (\tau,\xi) \,d\xi d\tau +\frac{25}{128}u_m^2 \int_0^t  |\dot{X}(\tau)|^2 \, d\tau \\
&\le \frac{1}{2} \int_\mathbb{R} (\phi^2 w)(0,\xi) \, d \xi + C\varepsilon_R^{-8/5}T_0^{-3/5} \sup_{0 \le \tau \le t} \|\phi(\tau)\|_{L^2}^2 \\
&\qquad + C \sup_{0 \le \tau \le t} \|\phi(\tau)\|_{L^2}^{2/3} \Bigl[ T_0^{-1/15} + \varepsilon_R^{-7/6}T_0^{-1/6} + T_0^{-1/3} + \varepsilon_R^{\frac{q-4}{3(q-1)}} + \varepsilon_R^{\frac{2q-16}{3(q-2)}} \Bigr],
\end{aligned}
\end{equation}
where we used $\varepsilon_R<1$ and $\kappa\le\delta_0u_m^{-2}$. Taking $T_0=\varepsilon_R^{-8}$, we have
\[
\varepsilon_R^{-8/5}T_0^{-3/5}=\varepsilon_R^{16/5}, \qquad \varepsilon_R^{-7/6}T_0^{-1/6}=\varepsilon_R^{1/6},
\]
and, for $q\ge10$,
\[
\frac{q-4}{3(q-1)}\ge\frac29, \qquad \frac{2q-16}{3(q-2)}\ge\frac16.
\]
Hence, as $\varepsilon_R<1$ and $\|\phi\|_{L^2}^{2/3}\le\|\phi\|_{L^2}^2+1$, the right-hand side of \eqref{eq:L2-est-comb} is bounded by
\[
\frac{1}{2} \int_\mathbb{R} (\phi^2 w)(0,\xi) \, d \xi + C\varepsilon_R^{1/6}\Bigl(\sup_{0 \le \tau \le t} \|\phi(\tau)\|_{L^2}^2 + 1\Bigr),
\]
which yields the desired result \eqref{eq:a-priori-L2-est}.
\end{proof}

\begin{proof}[Proof of Proposition~\ref{prop:a-priori-L2}]
Let $\delta_0$, $\varepsilon_1$, and $\delta_2$ be the constants of Lemma~\ref{lem:a-priori-L^2}. By the a priori assumption $\sup_{0\le\tau\le t}\|\phi(\tau)\|_{L^2}\le\varepsilon_1$, the right-hand side of \eqref{eq:a-priori-L2-est} is bounded by
\[
\frac12\int_\mathbb{R}(\phi^2 w)(0,\xi)\,d\xi + C\varepsilon_R^{1/6}.
\]
Since $\tfrac{15}{8}u_m^2 \le w \le \tfrac{15}{2}u_m^2$ by Lemma~\ref{lem:weight_properties}(2), we have
\[
\frac12\int_\mathbb{R}(\phi^2 w)(t,\xi)\,d\xi \ge \frac{15}{16}u_m^2\|\phi(t)\|_{L^2}^2,
\qquad
\frac12\int_\mathbb{R}(\phi^2 w)(0,\xi)\,d\xi \le \frac{15}{4}u_m^2\|\phi(0)\|_{L^2}^2.
\]
Dividing \eqref{eq:a-priori-L2-est} by $\tfrac{15}{16}u_m^2$ gives the desired result \eqref{eq:a-priori-L2}.
\end{proof}

\subsection{Higher-Order Estimates}
\begin{proposition}\label{prop:a-priori-H1}
There exist constants $\delta_0, \varepsilon_1, \delta_2 > 0$ such that if the dispersion coefficient $\kappa > 0$, the shock strength $\ds=|u_m-u_-|=3u_m>0$, the rarefaction strength $\delta_R=|u_+-u_m| > 0$, and the smoothing scale $\varepsilon_R$ satisfy
\begin{equation*}
\kappa u_m^2 \le \delta_0, \qquad \delta_R \le \frac{5}{18} \delta_S \; \left(= \frac{5}{6}u_m\right), \qquad 0 < \varepsilon_R < \delta_2,
\end{equation*}
and the perturbation $\phi \in C([0,T];H^1(\mathbb{R})) \cap L^2(0,T;H^2(\mathbb{R}))$ solves \eqref{eq:perturbed_general} on $[0,T]$ with
\[\sup_{t \in [0, T]} \| \phi (t) \|_{H^1(\mathbb{R})} \le \varepsilon_1,\]
then there exists a constant $C > 0$, independent of $\kappa$, $\varepsilon_R$, and $T$, such that for all $t \in [0, T]$,
\begin{multline} \label{eq:a-priori-H1}
\|\phi_\xi(t)\|_{L^2(\RR)}^2 +\int_0^t \|\phi_{\xi\xi}(\tau) \|_{L^2(\RR)}^2 \,d\tau
\le \| \phi_\xi(0)\|_{L^2(\RR)}^2 + C \int_0^t \|\phi_\xi(\tau)\|_{L^2(\RR)}^2 \, d \tau\\
+ C  \int_0^t\int_\mathbb{R}  \phi^2 \bigl( \us_\xi + \ur_\xi\bigr) \, d \xi \, d \tau +C\int_0^t |\dot{X}(\tau)|^2 \, d \tau +C \varepsilon_R^{4/3}.
\end{multline}
\end{proposition}

\begin{proof}
Multiplying the perturbation equation \eqref{eq:perturbed_general} by $-\phi_{\xi\xi}$ and integrating over $\mathbb{R}$, since
\[
    \int_{\mathbb{R}} \bigl( \sigma \phi_\xi - \kappa \phi_{\xi\xi\xi} \bigr) \phi_{\xi\xi} \,d\xi 
    = \frac{1}{2} \int_{\mathbb{R}}  \bigl( \sigma \phi_\xi^2 - \kappa \phi_{\xi\xi}^2 \bigr)_\xi \,d\xi = 0,
\]
integrating the remaining terms by parts yields
\begin{equation} \label{eq:H1-energy-1}
\begin{aligned}
    \frac{1}{2}\frac{d}{dt} \|\phi_{\xi}\|_{L^2}^2 + \|\phi_{\xi \xi}\|_{L^2}^2
    &= \int_\mathbb{R} \left( f(\phi+\tilde{u}) - f(\tilde{u}) \right)_\xi \phi_{\xi \xi} \, d\xi \\
    &\qquad - \dot{X}(t) \int_{\mathbb{R}} \phi_{\xi \xi} (\us_\xi + \ur_\xi ) \,d\xi + \int_\mathbb{R} \phi_{\xi \xi} E \, d\xi.
\end{aligned}
\end{equation}
Applying Young's inequality to the right-hand side of \eqref{eq:H1-energy-1} and using Theorem~\ref{thm:shock_properties} and Lemma~\ref{lem:approx_rarefaction}, we obtain
\begin{equation} \label{eq:H1-energy-2}
\begin{aligned}
    \frac{1}{2}\frac{d}{dt} \|\phi_{\xi}\|_{L^2}^2 + \frac{1}{2}\|\phi_{\xi \xi}\|_{L^2}^2 
    &\le 2 \int_\mathbb{R} |(f(\phi+\tilde{u}) - f(\tilde{u}))_\xi|^2 \, d\xi  + 2 \int_\mathbb{R} |E|^2 \, d\xi \\
    &\qquad + C |\dot{X}(t)|^2 (\|\us_\xi\|_{L^2}^2 + \|\ur_\xi\|_{L^2}^2)\\
    &\le 2 \int_\mathbb{R} |(f(\phi+\tilde{u}) - f(\tilde{u}))_\xi|^2 \, d\xi  + 2 \int_\mathbb{R} |E|^2 \, d\xi + C |\dot{X}(t)|^2.
\end{aligned}
\end{equation}

\noindent \textbf{(i) Flux:} For the flux term on the right-hand side of \eqref{eq:H1-energy-2}, since $f(u)=u^3$, the a priori bound $\|\phi\|_{L^\infty}\le \sqrt2\,\varepsilon_1$ and the boundedness of $\tilde u$ yield
\begin{equation} \label{eq:flux-x-L2-bound}
    \begin{aligned}
    \int_\mathbb{R} |(f(\phi+\tilde{u}) - f(\tilde{u}))_\xi|^2 \, d\xi
    &\le C \int_\mathbb{R} (\phi+\tilde{u})^4 \phi_\xi^2 \, d\xi
       + C \int_\mathbb{R} \phi^2 (\phi+2\tilde{u})^2 \,\tilde{u}_\xi^2 \, d\xi \\
    &\le C \|\phi_\xi\|_{L^2}^2
       + C (\|\us_\xi\|_{L^\infty}+\|\ur_\xi\|_{L^\infty}) \int_\mathbb{R} \phi^2 \bigl(  \us_\xi +  \ur_\xi \bigr) \, d\xi \\
    &\le C \|\phi_\xi\|_{L^2}^2
       + C (\delta_S^3 + \delta_R \varepsilon_R) \int_\mathbb{R}  \phi^2 \bigl( \us_\xi + \ur_\xi\bigr) \, d \xi .
    \end{aligned}
\end{equation}

\noindent \textbf{(ii) $\|E\|_{L^2}^2$:} Next, we expand the square of the error term $E$ in $L^2$:
\begin{equation} \label{eq:E-L2-bound}
    \begin{aligned}
        \int_{\mathbb{R}} |E|^2 \,d\xi &\le C \int_{\mathbb{R}} |\us - u_m|^2 (\ur_\xi)^2 \,d\xi
        + C \int_{\mathbb{R}} |\ur - u_m|^2 (\us_\xi)^2 \,d\xi \\
        &\qquad + C \int_{\mathbb{R}} |\ur_{\xi\xi}|^2 \,d\xi + C \kappa^2 \int_{\mathbb{R}} |\ur_{\xi\xi\xi}|^2 \,d\xi
        =:E_{21}^I+E_{22}^I+E_{21}^R+E_{22}^R.
    \end{aligned}
\end{equation}
We estimate the four terms on the right-hand side of \eqref{eq:E-L2-bound} in turn. Throughout we write $\tau' := \tau + T_0$, and by translation invariance we absorb the shift $X(\tau)$ via $\xi - X(\tau) \mapsto \xi$.

\vspace{2mm}
\noindent \textbf{(ii-a) $E_{21}^I$:}
Using the interaction estimates \eqref{eq:int1_c}--\eqref{eq:int2_c} of Lemma \ref{lem:wave_interaction}, we have
\begin{equation*}
\begin{aligned}
\int_{\mathbb{R}} |\us - u_m|^2 (\ur_\xi)^2 \,d\xi
&\le C \|\us-u_m\|_{L^\infty} \|\ur_\xi\|_{L^\infty} \int_{\mathbb{R}} |\us - u_m|\, \ur_\xi \,d\xi \\
&\le C \frac{\delta_S}{\tau'} \int_{\mathbb{R}} |\us - u_m|\, \ur_\xi \,d\xi 
\le C \frac{\delta_S^2 \delta_R(\delta_R \varepsilon_R )^{-\theta_1(1-\frac{1}{2q})}}{(\tau')^{1+\theta_1(1-\frac{1}{2q})}}
+ C \frac{\delta_S^{2\theta_2} \delta_R^{\theta_2}}{(\tau')^{2-\theta_2}},
\end{aligned}
\end{equation*}
where we choose $\theta_1, \theta_2 \in (0,1)$ such that $\theta_1(1-\frac{1}{2q}) = \tfrac12$ and $\theta_2=\tfrac12$. Integrating in time yields
\begin{equation} \label{eq:E21I-est}
\int_0^t \int_{\mathbb{R}} |\us - u_m|^2 (\ur_\xi)^2 \,d\xi\, d\tau
\le C \max\{ \delta_S^2 \varepsilon_R^{-1/2}, \delta_S \} \delta_R^{1/2} T_0^{-1/2}.
\end{equation}

\noindent \textbf{(ii-b) $E_{22}^I$:} We decompose the spatial domain into three regions: $(-\infty,0]$, $[0, (w_+-w_m)\tau']$, and $((w_+-w_m)\tau', +\infty)$. We estimate the integral on each region as follows:

\vspace{2mm}
For $\xi \le 0$, by applying Lemma \ref{lem:approx_rarefaction}(6) (with $\sigma=w_m$ and $\theta(1-\frac{1}{2q})=\frac78$) and Theorem \ref{thm:shock_properties}, we get
\begin{equation*}
\int_{-\infty}^0 |\ur-u_m|^2 (\us_\xi)^2\,d\xi 
\le \sup_{\xi \le 0} |\ur-u_m|^2 \int_{-\infty}^0 (\us_\xi)^2\,d\xi 
\le C\delta_S^4 \delta_R^2 (\delta_R\varepsilon_R\tau')^{-7/4}.
\end{equation*}
Integrating over time yields
\begin{equation} \label{eq:E22I-est-1}
    \int_0^t\int_{-\infty}^0 |\ur-u_m|^2 (\us_\xi)^2\,d\xi\, d\tau \le C\delta_S^4 \delta_R^{1/4}\varepsilon_R^{-7/4}T_0^{-3/4}.
\end{equation}

On $0 \le \xi \le (w_+-w_m)\tau'$, we use the inequality $|\ur-u_m|^2 \le 2|\ur-u^r|^2 + 2|u^r-u_m|^2$.\\
For the first term, Lemma \ref{lem:approx_rarefaction}(7) implies
\begin{equation*}
    \int_0^t\int_0^{(w_+-w_m)\tau'} |\ur-u^r|^2(\us_\xi)^2\,d\xi\, d\tau 
    \le C \int_0^t \|\ur-u^r\|_{L^\infty}^2 \|\us_\xi\|_{L^2}^2 \,d\tau 
    \le C\delta_S^4 \delta_R^{1/q}\varepsilon_R^{-2+1/q}T_0^{-1+1/q}.
\end{equation*}
For the second term, using $|u^r-u_m|\le C|\xi|/\tau'$ and $\us_\xi \le C\delta_S^3(1+C\delta_S^2\xi)^{-2}$, we have
\begin{equation*}
    \int_0^t\int_0^{(w_+-w_m)\tau'} |u^r-u_m|^2(\us_\xi)^2\,d\xi\,d\tau 
    \le \int_0^t \frac{C}{(\tau')^2}\int_0^\infty \frac{\delta_S^6\,\xi^2}{(1+C\delta_S^2\xi)^4}\,d\xi\, d\tau \le CT_0^{-1}.
\end{equation*}
Combining these two estimates gives
\begin{equation} \label{eq:E22I-est-2}
    \int_0^t\int_0^{(w_+-w_m)\tau'} |\ur-u_m|^2(\us_\xi)^2\,d\xi\, d\tau \le C\delta_S^4\delta_R^{1/q}\varepsilon_R^{-2+1/q}T_0^{-1+1/q} + CT_0^{-1}.
\end{equation}

On $\xi > (w_+-w_m)\tau'$, using the uniform bound $|\ur-u_m|\le\delta_R$ and $w_+-w_m \ge C\delta_S\delta_R$, we have 
\[
\int_{(w_+-w_m)\tau'}^\infty (\us_\xi)^2\,d\xi
\le C\delta_S^6 \int_{C\delta_S\delta_R\tau'}^\infty (1+C\delta_S^2\xi)^{-4}\,d\xi 
\le C\delta_S^4(1+C\delta_S^3\delta_R\tau')^{-3}.
\]
Integrating with respect to time implies
\begin{equation} \label{eq:E22I-est-3}
\begin{aligned}
    \int_0^t\int_{(w_+-w_m)\tau'}^\infty |\ur-u_m|^2(\us_\xi)^2\,d\xi\, d\tau 
    &\le C\delta_S^4\delta_R^2 \int_0^\infty \bigl(1+C\delta_S^3\delta_R \tau'\bigr)^{-3}\,d\tau \\
    &\le C\delta_S\delta_R\bigl(1+C\delta_S^3\delta_R T_0\bigr)^{-2}
    \le C\delta_R^{2/3}T_0^{-1/3},
\end{aligned}
\end{equation}
where we used $(1+X)^{-2}\le X^{-1/3}$ with $X=C\delta_S^3\delta_R T_0$.

Therefore, combining the estimates \eqref{eq:E22I-est-1}, \eqref{eq:E22I-est-2}, and \eqref{eq:E22I-est-3}, we have
\begin{equation} \label{eq:E22I-est}
\begin{aligned}
&\int_0^t \int_\mathbb{R} |\ur-u_m|^2 (\us_\xi)^2 \, d\xi d \tau\\
&\qquad
\le C\delta_S^4 \delta_R^{1/4}\varepsilon_R^{-7/4}T_0^{-3/4}
+ C\delta_S^4\delta_R^{1/q}\varepsilon_R^{-2+1/q}T_0^{-1+1/q} 
+ CT_0^{-1} +C \delta_R^{2/3}T_0^{-1/3}.
\end{aligned}
\end{equation}

\noindent \textbf{(ii-c) $E_{21}^R$ and $E_{22}^R$:} By Lemma \ref{lem:approx_rarefaction}(2) with $p=2$, the $L^2$-norms of the higher-order derivatives of the approximate rarefaction wave satisfy
\begin{equation*}
    \|\ur_{\xi\xi}(\tau+T_0, \cdot)\|_{L^2}^2 \le C \min \left\{ \delta_R^2 \varepsilon_R^3, \, \frac{\delta_R + \delta_R^{2/q}}{(\tau+T_0)^{2-2/q}} \right\}
\end{equation*}
and
\begin{equation*}
    \|\ur_{\xi\xi\xi}(\tau+T_0, \cdot)\|_{L^2}^2 \le C \min \left\{ \delta_R^2 \varepsilon_R^5, \, \frac{\delta_R + \delta_R^{4/q}}{(\tau+T_0)^{2-4/q}} \right\}.
\end{equation*}
By changing variable to $s := \tau + T_0$ and splitting the integrals at $t_{**}$ and $t^{**}$, which are given by
\[
    t_{**} := \left( \frac{\delta_R^{1/2} + \delta_R^{1/q}}{\delta_R \varepsilon_R^{3/2}} \right)^{\frac{q}{q-1}}, \qquad
    t^{**} := \left( \frac{\delta_R^{1/2} + \delta_R^{2/q}}{\delta_R \varepsilon_R^{5/2}} \right)^{\frac{q}{q-2}},
\]
and noting that the assumption $q \ge 10$ ensures time integrability, we directly obtain
\begin{equation} \label{eq:E21R-est}
\begin{aligned}
\int_0^t \|\ur_{\xi\xi}\|_{L^2}^2 \,d\tau &\le \int_0^\infty \|\ur_{\xi\xi}\|_{L^2}^2 \,ds 
\le \int_0^{t_{**}} \delta_R^2 \varepsilon_R^3 \,ds + \int_{t_{**}}^\infty \frac{\delta_R + \delta_R^{2/q}}{s^{2-2/q}} \,ds \\
&\le C \delta_R^2 \varepsilon_R^3 t_{**} \le C \delta_R^{\frac{q-2}{q-1}} \bigl(\delta_R^{1/2} + \delta_R^{1/q}\bigr)^{\frac{q}{q-1}} \varepsilon_R^{\frac{3q-6}{2(q-1)}}
\le C \left( \delta_R^{\frac{3q-4}{2(q-1)}} + \delta_R \right) \varepsilon_R^{4/3}
\end{aligned}
\end{equation}
and, similarly,
\begin{equation} \label{eq:E22R-est}
\begin{aligned}
\int_0^t \|\ur_{\xi\xi\xi}\|_{L^2}^2 \,d\tau &\le \int_0^\infty \|\ur_{\xi\xi\xi}\|_{L^2}^2 \,ds 
\le \int_0^{t^{**}} \delta_R^2 \varepsilon_R^5 \,ds + \int_{t^{**}}^\infty \frac{\delta_R + \delta_R^{4/q}}{s^{2-4/q}} \,ds \\
&\le C \delta_R^2 \varepsilon_R^5 t^{**} \le C \delta_R^{\frac{q-4}{q-2}} \bigl(\delta_R^{1/2} + \delta_R^{2/q}\bigr)^{\frac{q}{q-2}} \varepsilon_R^{\frac{5q-20}{2(q-2)}}
\le C \left( \delta_R^{\frac{3q-8}{2(q-2)}} + \delta_R \right) \varepsilon_R^{15/8}.
\end{aligned}
\end{equation}
Since $0<\varepsilon_R<\delta_2\le\delta_1$, Lemma~\ref{lem:wave_interaction} applies. Integrating \eqref{eq:H1-energy-2} over $[0,t]$ and applying \eqref{eq:flux-x-L2-bound}, \eqref{eq:E22I-est}, \eqref{eq:E21R-est}, and \eqref{eq:E22R-est}, we have
\begin{equation} \label{eq:H1-est-comb}
\begin{aligned}
    &\frac{1}{2} \| \phi_\xi(t)\|_{L^2}^2 + \frac{1}{2}\int_0^t \| \phi_{\xi\xi}(\tau)\|_{L^2}^2 \, d \tau \\
    &\le \frac{1}{2} \| \phi_\xi(0)\|_{L^2}^2 + C \int_0^t \|\phi_\xi(\tau)\|_{L^2}^2 \, d \tau + C \int_0^t\int_\mathbb{R}  \phi^2 \bigl( \us_\xi + \ur_\xi\bigr) \, d \xi \, d \tau +C\int_0^t |\dot{X}(\tau)|^2 \, d \tau \\
    &\quad + C \Bigl[ \varepsilon_R^{-1/2}T_0^{-1/2} + \varepsilon_R^{-7/4}T_0^{-3/4} + \varepsilon_R^{-2+\frac1q}T_0^{-1+\frac1q} + T_0^{-1} + T_0^{-1/3} + \varepsilon_R^{4/3} + \varepsilon_R^{15/8} \Bigr],
\end{aligned}
\end{equation}
where we used $\varepsilon_R<1$ and $\kappa\le\delta_0u_m^{-2}$. Taking $T_0=\varepsilon_R^{-8}$, since $\varepsilon_R<1$ and $q\ge10$, the right-hand side of \eqref{eq:H1-est-comb} is bounded by
\[
    \frac{1}{2} \| \phi_\xi(0)\|_{L^2}^2 + C \int_0^t \|\phi_\xi(\tau)\|_{L^2}^2 \, d \tau + C \int_0^t\int_\mathbb{R}  \phi^2 \bigl( \us_\xi + \ur_\xi\bigr) \, d \xi \, d \tau +C\int_0^t |\dot{X}(\tau)|^2 \, d \tau + C\varepsilon_R^{4/3},
\]
and multiplying by $2$ gives \eqref{eq:a-priori-H1}.
This completes the proof.
\end{proof}

%%%%%%%%%%%%%%%%%%%%%%%%%%%%%%%%%%%%%%%%%%%%%%%%%%%%%%%%%%%%%%%%%%%%%%%%%%%%%%%%%%%%%%%%%%%%%%%%%%%%%%%%%%%%
%%%%%%%%%%%%%%%%%%%%%%%%%%%%%%%%%%%%%%%%%%%%%%%%%%%%%%%%%%%%%%%%%%%%%%%%%%%%%%%%%%%%%%%%%%%%%%%%%%%%%%%%%%%%
%%%%%%%%%%%%%%%%%%%%%%%%%%%%%%%%%%%%%%%%%%%%%%%%%%%%%%%%%%%%%%%%%%%%%%%%%%%%%%%%%%%%%%%%%%%%%%%%%%%%%%%%%%%%
%%%%%%%%%%%%%%%%%%%%%%%%%%%%%%%%%%%%%%%%%%%%%%%%%%%%%%%%%%%%%%%%%%%%%%%%%%%%%%%%%%%%%%%%%%%%%%%%%%%%%%%%%%%%
%%%%%%%%%%%%%%%%%%%%%%%%%%%%%%%%%%%%%%%%%%%%%%%%%%%%%%%%%%%%%%%%%%%%%%%%%%%%%%%%%%%%%%%%%%%%%%%%%%%%%%%%%%%%
%%%%%%%%%%%%%%%%%%%%%%%%%%%%%%%%%%%%%%%%%%%%%%%%%%%%%%%%%%%%%%%%%%%%%%%%%%%%%%%%%%%%%%%%%%%%%%%%%%%%%%%%%%%%
%%%%%%%%%%%%%%%%%%%%%%%%%%%%%%%%%%%%%%%%%%%%%%%%%%%%%%%%%%%%%%%%%%%%%%%%%%%%%%%%%%%%%%%%%%%%%%%%%%%%%%%%%%%%
%%%%%%%%%%%%%%%%%%%%%%%%%%%%%%%%%%%%%%%%%%%%%%%%%%%%%%%%%%%%%%%%%%%%%%%%%%%%%%%%%%%%%%%%%%%%%%%%%%%%%%%%%%%%
\begin{appendix}
\section{Proofs of Lemmas \ref{lem:weight_properties} and \ref{lem:weight-shock}} \label{app:weight}
\begin{proof}[Proof of Lemma~\ref{lem:weight_properties}]
\noindent \textit{(1):}
To show that \(w \in C^3([-2s, s))\), we need to verify the continuity of \(w\) and its derivatives up to the third order at \(\us=0\) and \(\us=s/2\).

\vspace{2mm}
\(\bullet\) At $\us = 0$: Let $w_L(\us) = \frac{5}{2}s^2 - \frac{5}{2}s\us$.
Then, it simply follows that
\[
(w_L, w_L', w_L'', w_L''')|_{\us=0} = \Big(\frac{5}{2}s^2, -\frac{5}{2}s, 0, 0\Big).
\]
Given that $P_6(\us) = w_L(\us) + (\us)^4 P_2(\us)$, the term involving $(\us)^4$ and its first three derivatives vanish at $\us=0$. Thus, the values of $P_6$ and its first three derivatives are the same as those of $w_L$, which ensures $C^3$ continuity at $\us = 0$.

\vspace{2mm}
\(\bullet\) At $\us = s/2$: A direct computation of $P_6(\us)$ and its derivatives yields that
\begin{align*}
&P_6(s/2) = s^2 \left( \frac{5}{2} - \frac{5}{4} + \frac{50}{16} - \frac{120}{32} + \frac{80}{64} \right) = \frac{15}{8}s^2, 
&&P_6'(s/2) = s \left( -\frac{5}{2} + \frac{200}{8} - \frac{600}{16} + \frac{480}{32} \right) = 0, \\
&P_6''(s/2) = \left( \frac{600}{4} - \frac{2400}{8} + \frac{2400}{16} \right) = 0, 
&&P_6'''(s/2) = s^{-1} \left( \frac{1200}{2} - \frac{7200}{4} + \frac{9600}{8} \right) = 0.
\end{align*}
Since \(P_6(s/2)=15s^2/8\) and its derivatives vanish at $\us = s/2$, we conclude $w \in C^3([-2s, s))$.

\medskip
\noindent \textit{(2), (3):}
First of all, we introduce a scale-invariant variable $z = \us/s \in [-2, 1)$.
The weight function is written as $w(\us) = s^2 W(z)$, where $W(z)$ is a piecewise polynomial function (independent of $s$):
\begin{equation*}
W(z) =
\begin{cases}
\frac{5}{2}(1 - z), & z \in [-2, 0), \\
\frac{5}{2}(1 - z) + 50z^4 - 120z^5 + 80z^6, & z \in [0, 1/2), \\
\frac{15}{8}, & z \in [1/2, 1).
\end{cases}
\end{equation*}
By the chain rule, $w^{(k)}(\us) = s^{2-k} W^{(k)}(z)$ for $k \in \{0, 1, 2, 3\}$.
Moreover, since $W(z) \in C^3([-2, 1])$, its derivatives are uniformly bounded, and thus, it is enough to compute the following:
\[
C_k = \sup_{z \in [-2, 1)} |W^{(k)}(z)|.
\]

\noindent For $k=0$:
We have $W'(z) = -5/2$ on $[-2, 0)$.
Moreover, on $[0, 1/2)$, $W''(z) = 600z^2(1-2z)^2 \ge 0$, so $W'(z)$ is increasing from \(-5/2\) to \(0\).
Hence, $W'(z) \le 0$ for all $z \in [-2, 1/2)$, and therefore \(W\) is monotonically decreasing.
Since $W''\equiv0$ on $[-2,0)\cup[1/2,1)$, this proves (2), with $\tfrac{15}{8}=W(1/2)\le W\le W(-2)=\tfrac{15}{2}$.
Consequently, the absolute maximum is attained at \(z=-2\) and 
\[
C_0 = W(-2) = \frac{15}{2}.
\]

\noindent For $k=1$:
As shown above, $W'(z) \le 0$ on $[0, 1/2)$, and \(W'(z)\) increases monotonically to $0$ on this interval.
It then follows that $|W'(z)|$ decreases monotonically from $5/2$ to $0$.
Therefore, we obtain
\[
C_1 = |W'(-2)| = \frac{5}{2}.
\]

\noindent For $k=2$: On $[0, 1/2)$, we have $W''(z) = 600z^2(1-2z)^2$.
The critical points of \(W''\) are determined by $W'''(z) = 1200z(1-2z)(1-4z) = 0$.
The maximum of \(W''\) on this interval is attained at $z = 1/4$: 
\[
C_2 = W''\Big(\frac{1}{4}\Big) = \frac{75}{8}.
\]

\noindent For $k=3$: On $[0, 1/2)$, $W'''(z) = 1200z(1-2z)(1-4z)$.
Since $W^{(4)}(z) = 1200(1-12z+24z^2)$, the local extreme value is attained at $z = \frac{3 \pm \sqrt{3}}{12}$.
Thus, evaluating the absolute value gives the maximum:
\[
C_3 = \bigg| W'''\Big(\frac{3 \pm \sqrt{3}}{12}\Big) \bigg| = \frac{100\sqrt{3}}{3}.
\]

In conclusion, it holds that $|w^{(k)}(\us)| \le C_k s^{2-k}$ for $\us \in [-2s, s)$, where
\[
C_0=\frac{15}{2}, \qquad
C_1=\frac{5}{2}, \qquad
C_2=\frac{75}{8}, \qquad
C_3=\frac{100\sqrt{3}}{3}.
\]
This completes the proof of Lemma \ref{lem:weight_properties}.
\end{proof}

\begin{proof}[Proof of Lemma~\ref{lem:weight-shock}]
\textit{(1):} 
We prove the inequalities by dividing the domain into two subintervals.
We first show that 
\begin{equation} \label{3:eq1}
1 - w(\us) \frac{s - 2\us}{5s(s-\us)^2} \ge \frac{1}{6}.
\end{equation}

\noindent For $\us \in [-2s, 0)$: Since $w(\us) = \frac{5}{2}s(s-\us)$, we have 
\[
1 - \frac{5}{2}s(s-\us) \frac{s - 2\us}{5s(s-\us)^2} = \frac{s}{2(s-\us)}.
\]
Then, using $s < s - \us \le 3s$, we find that
\[
\frac{s}{2(s-\us)} \ge \frac{1}{6}.
\]

\noindent For $\us \in [0, s/2)$: Let $z = \us/s \in [0, 1/2]$.
Since $W'(z) \le 0$ on this interval, $W(z) \le W(0) = \frac{5}{2}$.
This implies
\begin{equation*}
1 - W(z) \frac{1 - 2z}{5(1-z)^2} \ge 1 - \frac{5}{2} \frac{1 - 2z}{5(1-z)^2} = 1 - \frac{1 - 2z}{2(1-z)^2}.
\end{equation*}
%The function $g(z) =\frac{1 - 2z}{2(1-z)^2}$ satisfies $g'(z) = -\frac{z}{(1-z)^3} \le 0$ for $z \in [0, 1/2]$, we have \(g(z)\le g(0)=1/2\).
Define $g(z) = \frac{1 - 2z}{2(1-z)^2}$. Since $g'(z) = -\frac{z}{(1-z)^3} \le 0$ for $z \in [0, 1/2]$, we have $\max_{z} g(z) = g(0) = 1/2$.
Thus, we obtain
\begin{equation*}
1 - W(z) \frac{1 - 2z}{5(1-z)^2} \ge 1 - g(0) = \frac{1}{2} > \frac{1}{6}.
\end{equation*}
This establishes \eqref{3:eq1}.

\vspace{2mm}
The remaining part, which is technically more challenging, is to prove the following: 
\[
H_1+H_2 > 4s^4.
\]

\noindent For $\us \in [-2s, 0)$: Plugging $w = \frac{5}{2}s(s-\us)$, $w' = -\frac{5}{2}s$ and $w'' = 0$ into \eqref{eq:H1H2-def}, we obtain
\begin{align*}
H_1 &= 5s \left[ -\frac{15}{2}s^3 + \frac{15}{2}s(\us)^2 + \frac{15}{2}s\us(s-\us) \right] = -\frac{75}{2}s^4 + \frac{75}{2}s^3\us, \\
H_2 &= \frac{5}{2}s(s-\us) \left[ 10s(s-\us) + \frac{5}{2}s(4\us+3s) \right] = \frac{175}{4}s^4 - \frac{175}{4}s^3\us.
\end{align*}
The simple summation, together with $s - \us > s$, yields 
\[
H_1 + H_2 = \frac{25}{4}s^3(s - \us) > \frac{25}{4}s^4 > 4s^4.
\]

\noindent For $\us \in [0, s/2)$: We again use the scale-invariant variable $z = \us/s \in [0, 1/2]$.
Then, applying the scaling relations $w = s^2 W(z)$, $w' = s W'(z)$ and $w'' = W''(z)$, we factor out $s^4$ from \eqref{eq:H1H2-def} to have 
\[
H_1 + H_2 = s^4 \hat{H}(z),
\]
where $\hat{H}(z) := \hat{H}_1(z) + \hat{H}_2(z)$ with
\begin{align*}
    \hat{H}_1(z) &:= 5 \left[ 3(1-z^2)W'(z) + 3zW(z) - \frac{1}{2}(z+2)(1-z)^2 W''(z) \right], \\
    \hat{H}_2(z) &:= W(z) \left[ W''(z)(1-2z)(z+2) + 4W(z) - W'(z)(4z+3) \right].
\end{align*}

We now expand \(\hat{H}(z)\) as follows: we note that
\begin{align*}
W(z)&=\frac{5}{2}(1-z)+50z^4-120z^5+80z^6,\\
W'(z)&=-\frac{5}{2}+200z^3-600z^4+480z^5,\\
W''(z)&=600z^2-2400z^3+2400z^4.
\end{align*}
Firstly, we gather all the terms involving \(W''\):
\begin{align*}
&W''(z) \Big[-\frac{5}{2}(z+2)(1-z)^2 + W(z)(1-2z)(z+2)\Big]\\
%&=W''(z) (z+2) \Big[-\frac{5}{2}(1-z)^2 + W(z)(1-2z)\Big]\\
%&=W''(z) (z+2) \Big[-\frac{5}{2}(1-z)^2 + \frac{5}{2}(1-z)(1-2z)
%+(50z^4-120z^5+80z^6)(1-2z)\Big]\\
%&=W''(z) (z+2) \Big[-\frac{5}{2}z(1-z)
%+50z^4-120z^5+80z^6-100z^5+240z^6-160z^7\Big]\\
&=(600z^2-2400z^3+2400z^4) (z+2) \Big[-\frac{5}{2}z+\frac{5}{2}z^2
+50z^4-220z^5+320z^6-160z^7\Big]\\
%&=600z^3(2-7z+4z^2+4z^3) \Big[-\frac{5}{2}+\frac{5}{2}z
%+50z^3-220z^4+320z^5-160z^6\Big]\\
&=1500z^3(2-7z+4z^2+4z^3) \big[-1+z+20z^3-88z^4+128z^5-64z^6\big]\\
%&=1500z^3\Big[-2+9z-11z^2+40z^3-312z^4+952z^5-1296z^6+608z^7+256z^8-256z^9\Big]\\
&=1500\big[-2z^3+9z^4-11z^5+40z^6-312z^7+952z^8-1296z^9+608z^{10}+256z^{11}-256z^{12}\big].
\end{align*}
\begin{comment}
\begin{align*}
&W''(z) \Big[-\frac{5}{2}(z+2)(1-z)^2 + W(z)(1-2z)(z+2)\Big]\\
&=W''(z) (z+2) \Big[-\frac{5}{2}(1-z)^2 + W(z)(1-2z)\Big]\\
&=W''(z) (z+2) \Big[-\frac{5}{2}(1-z)^2 + \frac{5}{2}(1-z)(1-2z)
+(50z^4-120z^5+80z^6)(1-2z)\Big]\\
&=W''(z) (z+2) \Big[-\frac{5}{2}z(1-z)
+50z^4-120z^5+80z^6-100z^5+240z^6-160z^7\Big]\\
&=(600z^2-2400z^3+2400z^4) (z+2) \Big[-\frac{5}{2}z+\frac{5}{2}z^2
+50z^4-220z^5+320z^6-160z^7\Big]\\
&=600z^3(2-7z+4z^2+4z^3) \Big[-\frac{5}{2}+\frac{5}{2}z
+50z^3-220z^4+320z^5-160z^6\Big]\\
&=1500z^3(2-7z+4z^2+4z^3) \Big[-1+z
+20z^3-88z^4+128z^5-64z^6\Big]\\
&=1500z^3\Big[-2+9z-11z^2+40z^3-312z^4+952z^5-1296z^6+608z^7+256z^8-256z^9\Big]\\
&=1500\Big[-2z^3+9z^4-11z^5+40z^6-312z^7+952z^8-1296z^9+608z^{10}+256z^{11}-256z^{12}\Big]
\end{align*}
\end{comment}
Secondly, we consider all the terms with \(W'\):
\begin{align*}
&W'(z)\Big[15(1-z^2)-W(z)(4z+3)\Big]\\
%&=W'(z)\Big[15(1-z^2)-\frac{5}{2}(1-z)(4z+3)
%-(50z^4-120z^5+80z^6)(4z+3)\Big]\\
%&=W'(z)\Big[(1-z)\Big(15(1+z)-\frac{5}{2}(4z+3)\Big)
%-(50z^4-120z^5+80z^6)(4z+3)\Big]\\
%&=W'(z)\Big[(1-z)\Big(15+15z-10z-\frac{15}{2}\Big)
%-(50z^4-120z^5+80z^6)(4z+3)\Big]\\
%&=W'(z)\Big[(1-z)\Big(\frac{15}{2}+5z\Big)
%-(50z^4-120z^5+80z^6)(4z+3)\Big]\\
&=\Big(-\frac{5}{2}+200z^3-600z^4+480z^5\Big)
\Big[\frac{15}{2}-\frac{5}{2}z-5z^2
-150z^4+160z^5+240z^6-320z^7\Big]\\
&=-\frac{75}{4}
+\frac{25}{4}z
+\frac{25}{2}z^2
+1500z^3
-4625z^4
+3700z^5
+1200z^6
-31600z^7
+122000z^8\\
&\qquad
-120000z^9
-131200z^{10}
+307200z^{11}
-153600z^{12}.
\end{align*}
Lastly, we collect all the terms involving \(W\):
\begin{align*}
&15zW(z)+4(W(z))^2
=W(z)\big(15z+4W(z)\big)\\
%&=\Big[\frac{5}{2}-\frac{5}{2}z+50z^4-120z^5+80z^6\Big]
%\Big[10+5z+200z^4-480z^5+320z^6\Big]\\
&=25
-\frac{25}{2}z
-\frac{25}{2}z^2
+1000z^4
-2650z^5
+2200z^6
-400z^7
+10000z^8
-48000z^9\\
&\qquad
+89600z^{10}
-76800z^{11}
+25600z^{12}.
\end{align*}
Then, summing up, we obtain
\begin{align*}
\hat{H}(z)&=
-512000z^{12}
+614400z^{11}
+870400z^{10}
-2112000z^9
+1560000z^8
-500000z^7\\
&\qquad
+63400z^6
-15450z^5
+9875z^4
-1500z^3
-\frac{25}{4}z
+\frac{25}{4}.
\end{align*}
Thus, the desired inequality boils down to the following: for all \(z\in[0,1/2]\),
\begin{align*}
H(z) &\coloneqq
-512000z^{12}
+614400z^{11}
+870400z^{10}
-2112000z^9
+1560000z^8\\
&\qquad -500000z^7
+63400z^6
-15450z^5
+9875z^4
-1500z^3
-\frac{25}{4}z
+\frac{9}{4} >0.
\end{align*}
To this end,  we divide the interval $[0, 1/2]$ into three subintervals: $[0, 1/8]$, $[1/8, 1/4]$, and $[1/4, 1/2]$.

\vspace{2mm}
\noindent \textbf{Case 1:} $z \in [0, 1/8]$.
By dropping the positive terms of degree \(k\ge6\) and evaluating the remaining negative terms of degree \(k\ge5\) at the right endpoint \(z=1/8\), we obtain
\begin{align*}
&-512000z^{12} - 2112000z^9 - 500000z^7 - 15450z^5 \\
&\qquad\ge -512000(1/8)^{12} - 2112000(1/8)^9 - 500000(1/8)^7 - 15450(1/8)^5
%\\
%&= -\frac{125}{16777216} - \frac{4125}{262144} - \frac{15625}{65536} - \frac{7725}{16384} =- \frac{12174525}{16777216}
> - \frac{3}{4}.
\end{align*}
Then, the following holds:
\begin{equation} \label{eq:P1(z)}
    H(z) > 9875z^4 - 1500z^3 - \frac{25}{4}z + \frac{3}{2} \coloneqq P_1(z).
\end{equation}
Thanks to the weighted AM--GM inequality $3 a^4 + b^4 \ge 4 a^3 b$, $\forall a,b \ge 0$, we find that 
\[
3z^4 + \Big(\frac{15}{128}\Big)^4 \ge \frac{15}{32}z^3,
\]
which implies 
\[
-1500z^3 \ge -9600z^4 - \frac{1265625}{2097152}.
\]
Using this together with \eqref{eq:P1(z)} and dropping the remaining positive term $275z^4$, we deduce
\[
P_1(z) \ge - \frac{25}{4}z + \frac{1880103}{2097152} \ge -\frac{25}{32} + \frac{1880103}{2097152} >0,
\qquad \forall z\in[0,1/8].
\]
%Since $z \le 1/8$, we have $-\frac{25}{4}z \ge -\frac{25}{32} = -\frac{1638400}{2097152}.$ Therefore,
%\[P_1(z) \ge -\frac{1638400}{2097152} + \frac{1880103}{2097152} = \frac{241703}{2097152} > 0,\]
%which implies
Thus, we conclude that
\[
H(z) > 0, \qquad \forall z \in [0, 1/8].
\]

\vspace{2mm}
\noindent \textbf{Case 2:} $z \in [1/8, 1/4]$. 
First of all, since $z \le 1/4$, the two highest-order terms satisfy
\[
614400z^{11} - 512000z^{12}
= 102400z^{11}(6 - 5z)
\ge 102400z^{11}\Big(6 - \frac{5}{4}\Big)
= 486400z^{11} > 0,
\]
allowing us to discard them.
Then, we sequentially group adjacent powers into the following form:
\[
z^{k-2}(az^2 - bz + c) \ge 0.
\]
%where $c$ is chosen to bound the quadratic part $az^2 - bz$ from below on $[1/8, 1/4]$.
In particular, we consider 
\[
z^8(870400z^2 - 2112000z + 473600) \ge 0,
\]
and
\[
z^6(1086400z^2 - 500000z + 57530) \ge 0.
\]
It is straightforward to verify that these are positive on the interval.
\begin{comment}
For powers 10, 9, and 8, we have 
\[ 870400z^{10} - 2112000z^9 + c z^8 = z^8(870400z^2 - 2112000z + c). \]
The axis of symmetry for the quadratic $870400z^2 - 2112000z$ is $z = \frac{2112000}{2(870400)} \approx 1.213$, which is strictly greater than $1/4$. Thus, it is strictly decreasing on $[1/8, 1/4]$, and its minimum is attained at the right endpoint $z=1/4$:
\[ 870400\left(\frac{1}{4}\right)^2 - 2112000\left(\frac{1}{4}\right) = 54400 - 528000 = -473600. \]
Choose $c=473600$ so that
\[ z^8(870400z^2 - 2112000z + 473600) \ge 0. \]
Then the remainder for the $z^8$ coefficient is given by
\[ 1560000 - 473600 = 1086400. \]

For powers 8, 7, and 6, we have 
\[ 1086400z^8 - 500000z^7 + c z^6 = z^6(1086400z^2 - 500000z + c). \]
The axis of symmetry for the quadratic $1086400z^2 - 500000z$ is $z = \frac{500000}{2(1086400)} = \frac{625}{2716} \approx 0.230$, which lies inside the interval $[1/8, 1/4]$. Its global minimum on this interval is thus attained at its vertex:
\[ 1086400\left(\frac{625}{2716}\right)^2 - 500000\left(\frac{625}{2716}\right) = -\frac{500000^2}{4(1086400)} = -\frac{39062500}{679} \approx -57529.45. \]
Choose $c=57530$ so that
\[ z^6(1086400z^2 - 500000z + 57530) \ge 0. \]
Then the remainder for the $z^6$ coefficient is given by
\[ 63400 - 57530 = 5870. \]
\end{comment}
Dropping these non-negative groups yields the following lower bound of the function \(H\):
\[
H(z) \ge 5870z^6 - 15450z^5 + 9875z^4 - 1500z^3 - \frac{25}{4}z + \frac{9}{4} \coloneqq P_2(z).
\]
We now show that \(P_2(z)\) is convex.
To this end, we first compute its first and second derivatives:
\begin{align*}
P_2'(z) &= 35220z^5 - 77250z^4 + 39500z^3 - 4500z^2 - \frac{25}{4}, \\
P_2''(z) &= z (\underbrace{176100z^3 - 309000z^2 + 118500z - 9000}_{\eqqcolon g(z)}).
\end{align*}
Then, the derivative of \(g(z)\) is 
\[
g'(z) = 528300z^2 - 618000z + 118500,
\]
whose roots are given by 
\[
z = \frac{2060 - \sqrt{1461220}}{3522} \approx 0.242, \qquad 
z = \frac{2060 + \sqrt{1461220}}{3522} \approx 0.928.
\]
On the interval $[1/8, 1/4]$, the derivative $g'(z)$ changes sign only at $z \approx 0.242$, where \(g(z)\) attains a local maximum.
Therefore, it suffices to evaluate \(g\) at the endpoints:
\begin{align*}
g(1/8) &= 176100(1/8)^3 - 309000(1/8)^2 + 118500(1/8) - 9000 %\\
%&= \frac{44025}{128} - \frac{618000}{128} + \frac{1896000}{128} - \frac{1152000}{128} = \frac{170025}{128} 
\approx 1328.32 > 0, \\
g(1/4) &= 176100(1/4)^3 - 309000(1/4)^2 + 118500(1/4) - 9000 %\\
%&= \frac{44025}{16} - \frac{309000}{16} + \frac{474000}{16} - \frac{144000}{16} = \frac{65025}{16}
\approx 4064.06 > 0.
\end{align*}
\begin{comment}
\begin{align*}
g(1/8) &= 176100(1/8)^3 - 309000(1/8)^2 + 118500(1/8) - 9000 \\
&= \frac{44025}{128} - \frac{618000}{128} + \frac{1896000}{128} - \frac{1152000}{128} = \frac{170025}{128} \approx 1328.32 > 0, \\[6pt]
g(1/4) &= 176100(1/4)^3 - 309000(1/4)^2 + 118500(1/4) - 9000 \\
&= \frac{44025}{16} - \frac{309000}{16} + \frac{474000}{16} - \frac{144000}{16} = \frac{65025}{16}\approx 4064.06 > 0.
\end{align*}
\end{comment}
Since both are positive and the only critical point in the interval is a local maximum, $g(z)$ is strictly positive on $[1/8, 1/4]$. Consequently, we obtain $P_2''(z) > 0$, confirming that $P_2(z)$ is strictly convex.

Next, we evaluate $P_2'(z)$ at $z = 0.173,0.174$ to locate the interval containing the unique minimizer:
\begin{align*}
P_2'(0.173) &= 35220(0.173)^5 - 77250(0.173)^4 + 39500(0.173)^3 - 4500(0.173)^2 - \frac{25}{4} \\
%&= 5.45782827951546 - 69.19630441725 + 204.5198215 - 134.6805 - 6.25 \\
&= -0.14915463773454 < 0, \\
P_2'(0.174) &= 35220(0.174)^5 - 77250(0.174)^4 + 39500(0.174)^3 - 4500(0.174)^2 - \frac{25}{4} \\
%&= 5.61740314465728 - 70.810144596 + 208.086948 - 136.242 - 6.25 \\
&= 0.40220654865728 > 0.
\end{align*}
Combined with the strict convexity of $P_2(z)$, this sign change guarantees that the unique minimizer $z_0$ lies strictly between \(0.173\) and \(0.174\).
Moreover, by the convexity of $P_2(z)$, its graph lies strictly above its tangent line at $z = 0.173$.
Thus, the minimum $P_2(z_0)$ satisfies the following lower bound:
\[
P_2(z_0) \ge P_2(0.173) + P_2'(0.173)(z_0 - 0.173).
\]
Since $z_0 < 0.174$ and $P_2'(0.173) < 0$, we find that 
\begin{align*}
P_2(z_0) 
&> P_2(0.173) + P_2'(0.173)(0.174 - 0.173)
= P_2(0.173) + 0.001 \times P_2'(0.173).\\
&=0.01083202909751243 + 0.001 \times (-0.14915463773454)
= 0.01068287445977789 >0.
\end{align*}
\begin{comment}
Evaluating $P_2(0.173)$ by expanding its terms exactly gives
\begin{equation} \label{eq:P2(0.173)}
    \begin{aligned}
P_2(0.173) &= 5870(0.173)^6 - 15450(0.173)^5 + 9875(0.173)^4 - 1500(0.173)^3 - \frac{25}{4}(0.173) + \frac{9}{4} \\
&= 0.15736738205936243 - 2.39419213283685 + 8.845482279875 - 7.7665755 - 1.08125 + 2.25 \\
&= 0.01083202909751243.
    \end{aligned}
\end{equation}
Substituting \eqref{eq:P2(0.173)} into the lower bound~\eqref{eq:P2(z0)-lower}, we obtain
\begin{align*}
P_2(z_0) &> 0.01083202909751243 + 0.001 \times (-0.14915463773454) \\
&= 0.01083202909751243 - 0.00014915463773454 \\
&= 0.01068287445977789 > 0.
\end{align*}
\end{comment}
Therefore, the minimum of $P_2(z)$ on $[1/8, 1/4]$ is strictly positive, which concludes that 
\[H(z) \ge P_2(z) > 0, \qquad \forall z \in [1/8, 1/4].\]

\noindent \textbf{Case 3:} $z \in [1/4, 1/2]$.
On this interval we subtract nonnegative polynomials from $H(z)$ so that the remainder is a quadratic with negative discriminant.
To this end, we successively construct auxiliary polynomials that cancel the highest-degree terms.
Their non-negativity is evident from the way they are chosen.

We consider the following polynomials:
\begin{align*}
&S_1(z) = -512000 z^{10} \Big(z-\frac{1}{4}\Big)\Big(z-\frac{1}{2}\Big) \ge 0, 
&&S_2(z) = 230400 z^9 \Big(z-\frac{1}{2}\Big)^2 \ge 0, \\
&S_3(z) = 1164800 z^6 \Big(z-\frac{1}{4}\Big)^2 \Big(z-\frac{1}{2}\Big)^2 \ge 0,
&&S_4(z) = -422400 z^6 \Big(z-\frac{1}{2}\Big)^3 \ge 0, \\
&S_5(z) = -20000 z^4 \Big(z-\frac{1}{4}\Big) \Big(z-\frac{1}{2}\Big)^3 \ge 0, 
&&S_6(z) = 200z^4\Big(z-\frac{1}{2}\Big)^2 \Big(z-\frac{1}{4}\Big) \ge 0, \\
&S_7(z) = 15150z^3 \Big(z-\frac{1}{4}\Big) \Big(z-\frac{1}{2}\Big)^2 \ge 0, 
&&S_8(z) = -\frac{5725}{2} \Big(z-\frac{1}{4}\Big)^2 \Big(z-\frac{1}{2}\Big)^3 \ge 0, \\
&S_9(z) = -\frac{5575}{2} \Big(z-\frac{1}{4}\Big)^3 \Big(z-\frac{1}{2}\Big) \ge 0, 
&&S_{10}(z) = \frac{13925}{32} \Big(z-\frac{1}{4}\Big)\Big(z-\frac{1}{2}\Big)^2 \ge 0.
\end{align*}
Then, we consider \(H(z) = Q(z)+R(z)\), where \(Q(z)= \sum_{i=1}^{10} S_i(z) \ge 0\) and \(R(z)\) is a quadratic polynomial given by 
\[
R(z)
=H(z) - Q(z)
= \frac{1}{512} (211100 z^2 - 110400 z + 14777).
\]
Since the discriminant of the quadratic numerator is negative and its leading coefficient is positive, \(R(z)\) is strictly positive.
Therefore, $H(z) = Q(z) + R(z) > 0$ for any $z \in [1/4, 1/2]$, as desired.

\vspace{2mm}
\noindent \textit{(2):}
Recall $\zeta = s^2 \xi$, $z(\zeta)=\us(\xi)/s $ and $s^2 W(z)=w(\us)$.
Then, using $w^{(k)}(\us) = s^{2-k} W^{(k)}(z)$, we substitute these relations into \eqref{eq:M1234-def}, except for the derivatives \(\us_\x\), \(\us_{\x\x}\), and \(\us_{\x\x\x}\), which are kept unchanged so that Theorem \ref{thm:shock_properties} can be applied directly.
This yields
\begin{align*}
M_{11} &= \sup \frac{3|W'|}{2W} \cdot \frac{\us_\x}{s^3}, \\
M_{12} &= \sup W \left| \frac{2(1-2z)}{5(1-z)^2} \right| \cdot \frac{|\us_{\x\x}|}{s^2 \us_\x}, \\
M_{21} &= \sup \left| W W'' \frac{1-2z}{5(1-z)^2} - W W' \frac{2z}{5(1-z)^3} - \frac{1}{2} W'' \right|
\cdot \frac{|\us_{\x\x}|}{s^5}, \\
M_{22} &= \sup \left| W W' \frac{1-2z}{5(1-z)^2} \right|
\cdot \frac{|\us_{\x\x\x}|}{s^4\us_\x}, \\
M_{23} &= \sup \frac{1}{2}\bigg| 2W' \cdot \frac{\us_{\x\x\x}}{s^4 \us_\x}
+  W''' \cdot \frac{(\us_\x)^2}{s^6}
+ 3W'' \cdot \frac{\us_{\x\x}}{s^5} \bigg|, \\
M_{24} &= \sup (W')^2 \left| \frac{2(1-2z)}{5(1-z)^2} \right| \cdot \frac{|\us_{\x\x}|}{s^5},
\end{align*}
where all suprema are taken over \(z\in[-2,1/2]\) and \(\x \in (-\infty,\x_*]\).
Since $W(z)$ and its derivatives are continuous, they are bounded on the compact interval $[-2, 1/2]$.
%{\color{blue}The rational factors and $1/W$ are also bounded there, since $1-z\ge\frac12$ and $W\ge W(1/2)=\frac{15}{8}$ by \textit{(2)}.}
All factors, except for the powers of \(s\) and the derivatives of \(\us\), are bounded, since $1-z\ge\frac12$ and $W\ge W(1/2)=\frac{15}{8}$ by Lemma \ref{lem:weight_properties}(2).
It therefore remains to bound the factors involving the shock profile, namely
\[
\frac{\us_\x}{s^3}, \qquad
\frac{|\us_{\x\x}|}{s^2 \us_\x}, \qquad
\frac{|\us_{\x\x}|}{s^5}, \qquad
\frac{\us_{\x\x\x}}{s^4 \us_\x}, \qquad
\frac{(\us_\x)^2}{s^6},
\]
uniformly in $s$, and hence uniformly in $\kappa$.
The first factor is bounded above by the upper bound in \eqref{eq:key-ineq}, which gives $\us_\x\le 2(\us-s)^2(\us+2s)\le Cs^3$.
The factors $|\us_{\x\x}|/(s^2\us_\x)$ and $\us_{\x\x\x}/(s^4\us_\x)$ are bounded because of the absolute estimates \eqref{eq:absbdd}.
The remaining two factors can be written as
\begin{align*}
&\frac{|\us_{\x\x}|}{s^5}=\frac{|\us_{\x\x}|}{s^2\us_\x}\cdot\frac{\us_\x}{s^3}
&&\frac{(\us_\x)^2}{s^6}=\Big(\frac{\us_\x}{s^3}\Big)^2
\end{align*}
and so they are bounded.
%Consequently, $M_{11}, \ldots, M_{24}$ are bounded by positive constants independent of $s$ and $\kappa$.
Thus, $M_{11}, \ldots, M_{24}$ are uniformly bounded with respect to \(s\) and \(\k\).
\end{proof}
\end{appendix}

\bibliographystyle{plain}
\bibliography{reference}

\end{document}